\documentclass{article}

\usepackage[final, nonatbib]{neurips_2020}

\usepackage{microtype}
\usepackage{graphicx}
\usepackage{subfigure}
\usepackage{booktabs}

\usepackage[utf8]{inputenc} 
\usepackage[T1]{fontenc}    
\usepackage{hyperref}       
\usepackage{url}            
\usepackage{booktabs}       
\usepackage{amsfonts}       
\usepackage{nicefrac}       
\usepackage{microtype}      

\usepackage{amssymb,amsmath}
\usepackage{amsthm}
\usepackage{mathtools}  
\usepackage{bbm}  
\usepackage{cleveref}
\usepackage{enumerate}   
\usepackage{wrapfig}
\usepackage[numbers]{natbib}
\usepackage{subfiles}
\usepackage{etoolbox}
\newtoggle{supplementary}

\toggletrue{supplementary}

\usepackage{algorithm}
\usepackage{algorithmic}

\usepackage{./notations}

\date{}

\usepackage{hyperref}

\title{Entropic Optimal Transport between Unbalanced Gaussian Measures has a Closed Form }

\author{
Hicham Janati \\
	Inria Saclay\\
	Paris-Saclay, France\\
	\texttt{hicham.janati@inria.fr} \\
	\And
	Boris Muzellec \\
	 ENSAE,\\
	Paris-Saclay, France\\
	\texttt{boris.muzellec@ensae.fr} \\
	\And
	Gabriel Peyr\'e \\
    CNRS and ENS, PSL University \\
    Paris, France\\
\texttt{gabriel.peyre@ens.fr}
	\And
	Marco Cuturi \\
	Google Brain, ENSAE \\
    Paris Saclay, France\\
\texttt{cuturi@google.com}
}

\begin{document}
\maketitle
\begin{abstract}
Although optimal transport (OT) problems admit closed form solutions in a very few notable cases, e.g. in 1D or between Gaussians, these closed forms have proved extremely fecund for practitioners to define tools inspired from the OT geometry. On the other hand, the numerical resolution of OT problems using entropic regularization has given rise to many applications, but because there are no known closed-form solutions for entropic regularized OT problems, these approaches are mostly algorithmic, not informed by elegant closed forms. In this paper, we propose to fill the void at the intersection between these two schools of thought in OT by proving that the entropy-regularized optimal transport problem between two Gaussian measures admits a closed form. Contrary to the unregularized case, for which the explicit form is given by the Wasserstein-Bures distance, the closed form we obtain is differentiable everywhere, even for Gaussians with degenerate covariance matrices. We obtain this closed form solution by solving the fixed-point equation behind Sinkhorn's algorithm, the default method for computing entropic regularized OT. Remarkably, this approach extends to the generalized \textit{unbalanced} case --- where Gaussian measures are scaled by positive constants. This extension leads to a closed form expression for unbalanced Gaussians as well, and highlights the mass transportation / destruction trade-off seen in unbalanced optimal transport. Moreover, in both settings, we show that the optimal transportation plans are (scaled) Gaussians and provide analytical formulas of their parameters. These formulas constitute the first non-trivial closed forms for entropy-regularized optimal transport, thus providing a ground truth for the analysis of entropic OT and Sinkhorn's algorithm.
\end{abstract}

\section{Introduction}
Optimal transport (OT) theory~\cite{Villani09,figalli2017monge} has recently inspired several works in data science, where dealing with and comparing probability distributions, and more generally positive measures, is an important staple (see~\citep{Peyre2019computational} and references therein). For these applications of OT to be successful, a belief now widely shared in the community is that some form of regularization is needed for OT to be both scalable and avoid the curse of dimensionality~\cite{dereich2013constructive,fournier2015rate}. Two approaches have emerged in recent years to achieve these goals: either regularize directly the measures themselves, by looking at them through a simplified lens; or regularize the original OT problem using various modifications. The first approach exploits well-known closed-form identities for OT when comparing two univariate measures or two multivariate Gaussian measures. In this approach, one exploits those formulas and operates by summarizing complex measures as one or possibly many univariate or multivariate Gaussian measures.
The second approach builds on the fact that for arbitrary measures, regularizing the OT problem, either in its primal or dual form, can result in simpler computations and possibly improved sample complexity. The latter approach can offer additional benefits for data science: because the original marginal constraints of the OT problem can also be relaxed, regularized OT can also yield useful tools to compare measures with different total mass --- the so-called ``unbalanced'' case~\citep{benamou03}--- which provides a useful additional degree of freedom. Our work in this paper stands at the intersection of these two approaches. To our knowledge, that intersection was so far empty: no meaningful closed-form formulation was known for regularized optimal transport. We provide closed-form formulas of entropic (OT) of two Gaussian measures for balanced and unbalanced cases.
\paragraph{Summarizing measures \emph{vs.} regularizing OT.} 
Closed-form identities to compute OT distances (or more generally recover Monge maps) are known when either (1) both measures are univariate and the ground cost is submodular~\cite[\S2]{SantambrogioBook}: in that case evaluating OT only requires integrating that submodular cost w.r.t. the quantile distributions of both measures; or (2) both measures are Gaussian, in a Hilbert space, and the ground cost is the squared Euclidean metric~\cite{dowson82frechet,gelbrich1990}, in which case the OT cost is given by the Wasserstein-\citeauthor{bures1969extension} metric~\citep{bhatia2018bures,malago2018wasserstein}. These two formulas have inspired several works in which data measures are either projected onto 1D lines~\citep{rabin-ssvm-11,2013-Bonneel-barycenter}, with further developments in~\cite{paty2019subspace,kolouri2019generalized,titouan2019sliced}; or represented by Gaussians, to take advantage of the simpler computational possibilities offered by the Wasserstein-Bures metric~\cite{heusel2017gans,muzellec2018elliptical,chen2018optimal}.

Various schemes have been proposed to regularize the OT problem in the primal~\cite{cuturi2013sinkhorn,FrognerNIPS} or the dual~\citep{shirdhonkar2008approximate,arjovsky2017wasserstein,2016-Cuturi-siims}. We focus in this work on the formulation obtained by~\cite{chizat17}, which combines entropic regularization~\cite{cuturi2013sinkhorn} with a more general formulation for unbalanced transport~\cite{2017-chizat-focm,LieroMielkeSavareShort,LieroMielkeSavareLong}. The advantages of unbalanced entropic transport are numerous: it comes with favorable sample complexity regimes compared to unregularized OT~\citep{genevay19sample}, can be cast as a loss with favorable properties~\citep{genevay18a,feydy2019}, and can be evaluated using variations of the Sinkhorn algorithm~\cite{genevay2016stochastic}. 

\paragraph{On the absence of closed-form formulas for regularized OT.} 
Despite its appeal, one of the shortcomings of entropic regularized OT lies in the absence of simple test-cases that admit closed-form formulas. While it is known that regularized OT can be related, in the limit of infinite regularization, to the energy distance~\cite{ramdas2017wasserstein}, the absence of closed-form formulas for a fixed regularization strength poses an important practical problem to evaluate the performance of stochastic algorithms that try to approximate regularized OT: we do not know of any setup for which the ground truth value of entropic OT between continuous densities is known. The purpose of this paper is to fill this gap, and provide closed form expressions for balanced and unbalanced OT for Gaussian measures. We hope these formulas will prove useful in two different ways: as a solution to the problem outlined above, to facilitate the evaluation of new methodologies building on entropic OT, and more generally to propose a more robust yet well-grounded replacement to the Bures-Wasserstein metric.

\paragraph{Related work.} From an economics theory perspective, \citet{galichon} provided a closed form for an ``equilibrium 2-sided matching problem'' which is equivalent to entropy-regularized optimal transport. Second, a sequence of works in optimal control theory \cite{chen16, chen18, chen20} studied stochastic systems, of which entropy regularized optimal transport between Gaussians can be seen as a special case, and found a closed form of the optimal dual potentials. Finally, a few recent concurrent works provided a closed form of entropy regularized OT between Gaussians: first \citet{gerolin20} found a closed form in the univariate case, then \citet{mallasto20} and \citet{delbarrio20} generalized the formula for multivariate Gaussians. The closest works to this paper are certainly those of \citet{mallasto20} and \citet{delbarrio20} where the authors solved the balanced entropy regularized OT and studied the Gaussian barycenters problem. To the best of our knowledge, the closed form formula we provide for unbalanced OT is novel. Other differences between this paper and the aforementioned papers are highlighted below.

\paragraph{Contributions.} Our contributions can be summarized as follows:
\begin{itemize}
\item \Cref{thm:otclosedform} provides a closed form expression of the entropic (OT) plan $\pi$, which is shown to be a Gaussian measure itself (also shown in \citep{galichon, chen20, mallasto20, delbarrio20}). Here, we furthermore study the properties of the OT loss function: it remains well defined, convex and differentiable even for singular covariance matrices unlike the Bures metric.
\item Using the definition of debiased Sinkhorn barycenters~\citep{luise19,janati20}, \Cref{thm:barycenters} shows that the entropic barycenter of Gaussians is Gaussian and its covariance verifies a fixed point equation similar to that of \citet{agueh11}. \citet{mallasto20} and \citet{delbarrio20} provided similar fix point equations however by restricting the barycenter problem to the set of Gaussian measures whereas we consider the larger set of sub-Gaussian measures.
\item As in the balanced case,  \Cref{thm:unbalanced} provides a closed form expression of the unbalanced Gaussian transport plan. The obtained formula sheds some light on the link between mass destruction and the distance between the means of $\alpha, \beta$ in Unbalanced OT.
\end{itemize}
\paragraph{Notations.}
$\Scal^d$ denotes the set of square symmetric matrices in $\RR^{d\times d}$. $\Scal^d_{++}$ and $\Scal^d_+$ denote the cones of positive definite and positive semi-definite matrices in $\Scal^d$ respectively.
Let $\Ncal(\ba, \bA)$ denote the multivariate Gaussian distribution with mean $\ba \in \RR^d$ and variance $\bA \in \Scal^d_{++}$.
$f = \Qcal(\ba, \bA)$ denotes the quadratic form $f: x\mapsto -\frac{1}{2}(x^\top \bA x - 2\ba^\top x)$ with $\bA \in \Scal^d$. For short, we denote $\Qcal(\bA) = \Qcal(0, \bA)$. Whenever relevant, we follow the convention $0\log 0 = 0$. $\Mcal^+_p$ denotes the set of non-negative measures in $\RR^d$ with a finite p-th order moment and its subset of probablity measures $\Pcal_p$. For a non-negative measure $\alpha \in \Mcal^+_p(\RR^d)$, $\Lcal_2(\alpha)$ denotes the set of functions  $f: \RR^d\to\RR$ such that $\Esp_\alpha(|f|^2)=\int_{\RR^d} |f|^2 \dd \alpha < +\infty$. With $\bC\in\Scal^d_{++}$ and $\ba,\bb\in\RR^d$, we denote the squared Mahalanobis distance: $\|\ba-\bb\|_{\bC}^2 = (\ba-\bb)^\top\bC(\ba-\bb)$.

\section{Reminders on Optimal Transport}\label{subsec:ot_gaussians}
\paragraph{The Kantorovich problem.}
Let $\alpha, \beta \in \PP_2$ and let $\Pi(\alpha, \beta)$ denote the set of probability measures in $\PP_2$ with marginal distributions equal to $\alpha$ and $\beta$. 
%
The 2-Wasserstein distance is defined as:
\begin{align}
\label{eq:kantorovich}
W_2^2(\alpha, \beta) \defeq \min_{\pi \in \Pi(\alpha, \beta)} \int_{\RR^{d\times d}} \| x - y \|^2 \dd\pi(x, y).
\end{align}
This is known as the \emph{Kantorovich} formulation of optimal transport. When $\alpha$ is absolutely continuous with respect to the Lebesgue measure (i.e. when $\alpha$ has a density), \Cref{eq:kantorovich} can be equivalently rewritten using the \emph{Monge} formulation, where $T_\sharp\mu =  \nu$ \textit{i.f.f.} for all Borel sets $A$, $\nu(T(A)) = \mu(A)$:
\begin{align}
\label{eq:monge}
W_2^2(\alpha, \beta) = \min_{T: T_\sharp\alpha = \nu} \int_{\RR^{d}} \| x - T(x) \|^2 \dd\alpha(x).
\end{align}
The optimal map $T^*$ in \Cref{eq:monge} is called the Monge map.

\paragraph{The Wasserstein-Bures metric.}
Let $\Gauss{m}{\Sigma}$ denote the Gaussian distribution on $\RR^d$ with mean $m \in \RR^d$ and covariance matrix $\Sigma \in S^d_{++}$. A well-known fact~\citep{dowson82frechet,takatsu2011} is that \Cref{eq:kantorovich} admits a closed form for Gaussian distributions, called the Wasserstein-Bures distance (a.k.a. the \emph{Fr\'echet} distance):
\begin{align}\label{eq:wasserstein_bures}
    W_2^2(\Gauss{a}{\bA}, \Gauss{b}{\bB}) = \|a - b\|^2 + \Bures(\bA, \bB),
\end{align}
where $\mathfrak{B}$ is the \emph{Bures} distance~\citep{bhatia2018bures} between positive matrices:
\begin{align}\label{eq:bures}
    \Bures(\bA, \bB) \defeq \tr\bA + \tr\bB - 2 \tr(\bA\rt\bB\bA\rt)\rt .
\end{align}
Moreover, the Monge map between two Gaussian distributions admits a closed form: $T^\star:x\rightarrow \bT^{\bA\bB}(x - \ba) + \bb$, with
\begin{align}
\label{eq:monge_map}
    \begin{split}
    \bT^{\bA\bB} \defeq \bA\mrt(\bA\rt\bB\bA\rt)\rt\bA\mrt
                = \bB\rt(\bB\rt\bA\bB\rt)\mrt\bB\rt,
    \end{split}
\end{align}
which is related to the Bures gradient (w.r.t. the Frobenius inner product):
\begin{align}
\label{eq:bures-gradient}
    \nabla_\bA\Bures(\bA, \bB) = \Id - \bT^{\bA\bB}.
\end{align}
$\Bures(\bA, \bB)$ and its gradient can be computed efficiently on GPUs using Newton-Schulz iterations which are provided in \Cref{alg:NewtonSchulz} along with numerical experiments in the appendix.
%

\section{Entropy-Regularized Optimal Transport between Gaussians}
Solving \eqref{eq:kantorovich} can be quite challenging, even in a discrete setting~\citep{Peyre2019computational}. Adding an entropic regularization term to \eqref{eq:kantorovich} results in a problem which can be solved efficiently using Sinkhorn's algorithm~\citep{cuturi2013sinkhorn}. Let $\sigma > 0$. This corresponds to solving the following problem: 
\begin{align}\label{eq:ent_ot}
\OT_\sigma(\alpha, \beta) \defeq \min_{\pi \in \Pi(\alpha, \beta)} \int_{\RR^d\times\RR^d} \| x - y \|^2 \dd\pi(x, y) + 2\sigma^2 \KL(\pi \| \alpha\otimes\beta),
\end{align}
where $\KL(\pi \| \alpha\otimes\beta) \defeq \int_{\RR^d} \log\left(\tfrac{\dd\pi}{\dd\alpha\dd\beta}\right)\dd\pi$ is the Kullback-Leibler divergence (or relative entropy). 
As in the original case \eqref{eq:kantorovich}, $\OT_\sigma$ can be studied with centered measures (i.e zero mean) with no loss of generality:
\begin{lemma}
\label{lemma:centering}
Let $\alpha, \beta \in \Pcal$ and  $\bar{\alpha}, \bar{\beta}$ their respective centered transformations. It holds that
\begin{equation}
    \label{eq:centering}
    \OT_\sigma(\alpha, \beta) = \OT_\sigma(\bar{\alpha}, \bar{\beta}) +  \| \ba - \bb \|^2.
\end{equation}
\end{lemma}
\paragraph{Dual problem and Sinkhorn's algorithm.}
Compared to \eqref{eq:kantorovich}, \eqref{eq:ent_ot} enjoys additional properties, such as the uniqueness of the solution $\pi^*$. Moreover, problem \eqref{eq:ent_ot} has the following dual formulation:
\begin{align}
\label{eq:dual-ent-ot}
\OT_{\sigma}(\alpha, \beta) = &\max_{\substack{f \in \Lcal_1(\alpha), \\ g \in \Lcal_1(\beta)}} \Esp_\alpha(f) + \Esp_\beta(g) - 2\sigma^2\left( \int_{{\RR^d} \times {\RR^d}}\hspace{-2em} e^{\tfrac{f(x) + g(y) - \|x-y\|^2}{2\sigma^2}}\dd\alpha(x)\dd\beta(y) -1\right).
\end{align}
If $\alpha$ and $\beta$ have finite second order moments, a pair of dual potentials $(f, g)$ is optimal if and only they verify the following optimality conditions $\beta$-a.s and $\alpha$-a.s respectively~\cite{mena19}:
\begin{align}
    \label{eq:optimality-potentials}
    \begin{split}
    e^{\tfrac{f(x)}{2\sigma^2}} \left(\int_{\RR^d} e^{\tfrac{-\|x-y\|^2 + g(y)}{2\sigma^2}}  \dd\beta(y)\right) = 1, \quad
    e^{\tfrac{g(x)}{2\sigma^2}} \left(\int_{\RR^d} e^{\tfrac{-\|x-y\|^2 + f(y)}{2\sigma^2}}  \dd\alpha(y)\right) = 1.
    \end{split}
\end{align}
Moreover, given a pair of optimal dual potentials $(f, g)$, the optimal transportation plan is given by
\begin{equation}
    \label{eq:transportion_plan}
    \frac{\dd \pi^\star}{\dd \alpha \dd \beta}(x, y) = e^{\frac{f(x) + g(y) - \|x-y\|^2}{2\sigma^2}}.
\end{equation}
Starting from a pair of potentials $(f_0, g_0)$, the optimality conditions \eqref{eq:optimality-potentials} lead to an alternating dual ascent algorithm, which is equivalent to Sinkhorn's algorithm in log-domain:
\begin{align}
\label{eq:alg_sinkhorn}
    \begin{split}
    g_{n+1} &= \left(y \in {\RR^d} \rightarrow -2\sigma^2 \log\int_{\RR^d} e^{\tfrac{-\|x- y\|^2 + f_n(x)}{2\sigma^2}} \dd\alpha(x)\right),\\
    f_{n+1} &= \left(x \in {\RR^d} \rightarrow -2\sigma^2 \log\int_{\RR^d} e^{\tfrac{-\|x-y\|^2 + g_{n + 1}(y)}{2\sigma^2}}  \dd\beta(y)\right).
    \end{split}
\end{align}
\citet{sejourne19} showed that when the support of the measures is compact, Sinkhorn's algorithm converges to a pair of dual potentials. Here in particular, we study Sinkhorn's algorithm when $\alpha$ and $\beta$ are Gaussian measures.
\paragraph{Closed form expression for Gaussian measures.}\label{subsec:ent_ot_gaussians}
%
%
\begin{theorem}\label{thm:otclosedform}
Let $\bA, \bB \in \PD$ and $\alpha \sim \Ncal(\ba, \bA)$ and $\beta \sim \Ncal(\bb, \bB)$. Define $\bD_\sigma = (4\bA^{\frac{1}{2}} \bB \bA^{\frac{1}{2}} +\sigma^4 \Id )^{\frac{1}{2}}$. Then,
\begin{equation}\label{eq:means_plus_bures_sk}
    \OT_{\sigma}(\alpha, \beta) = \|\ba - \bb\|^2 + \Bcal^2_{\sigma}(\bA, \bB), \text{ where } 
\end{equation}
\begin{align}
\label{eq:otclosedform}
\begin{split}
\Bcal^2_{\sigma}(\bA, \bB) &= \tr(\bA) + \tr(\bB) - \tr(\bD_\sigma) + d\sigma^2(1 - \log(2\sigma^2)) + \sigma^2\log\det\left(\bD_\sigma + \sigma^2 \Id\right).
\end{split}
\end{align}
Moreover, with $\bC_{\sigma} = \frac{1}{2}\bA\rt \bD_\sigma\bA\mrt -\frac{\sigma^2}{2} \Id$, the Sinkhorn optimal transportation plan is also a Gaussian measure over $\RR^d\times \RR^d$ given by
\begin{equation}\label{eq:sinkorn_plan}
    \pi^\star \sim 
    \Ncal\left(\left(
            \begin{smallmatrix}
            \ba \\ 
            \bb
            \end{smallmatrix}\right),
         \left(\begin{smallmatrix}
         \bA & \bC_\sigma \\ \bC_\sigma^\top & \bB
         \end{smallmatrix}\right)
        \right).
\end{equation}
\end{theorem}
\begin{remark}
While for our proof it is necessary to assume that $\bA$ and $\bB$ are positive definite in order for them to have a Lebesgue density, notice that the closed form formula given by \Cref{thm:otclosedform} remains well-defined for positive semi-definite matrices. Moreover, unlike the Bures-Wasserstein metric, $\OT_{\sigma}$ is differentiable even when $\bA$ or $\bB$ are singular.
\end{remark}

The proof of \ref{thm:otclosedform} is broken down into smaller results, \Cref{prop:sinkhorn-transform,prop:bures_sinkhorn_convergence,prop:closed_form_matrix,lemma:optim-loss}. Using \Cref{lemma:centering}, we can focus in the rest of this section on centered Gaussians without loss of generality.
\paragraph{Sinkhorn's algorithm and quadratic potentials.}
We obtain a closed form solution of $\OT_\sigma$ by considering quadratic solutions of \eqref{eq:optimality-potentials}. The following key proposition characterizes the obtained potential after a pair of Sinkhorn iterations with quadratic forms.
\begin{proposition}
\label{prop:sinkhorn-transform}
Let $\alpha \sim \Ncal(0, \bA)$ and $\beta \sim \Ncal(0, \bB)$ and the Sinkhorn transform $T_\alpha: \RR^{\RR^d} \to \RR^{\RR^d}$: 
\begin{equation}
    T_\alpha(h)(x) \defeq - \log\int_{\RR^d} e^{\tfrac{-\|x- y\|^2}{2\sigma^2} + h(y)} \dd\alpha(y).
\end{equation}
Let $\bX \in \Scal_d$.
If $h = m + \Qcal(\bX)$ i.e  $h(x) = m - \frac{1}{2} x^\top \bX x$ for some $m \in \RR$, then $T_\alpha(h)$ is well-defined if and only if  $\bX' \defeq \sigma^2\bX + \sigma^2\bA^{-1} + \Id \succ 0$. In that case,
\begin{enumerate}[(i)]
    \item $T_\alpha(h)= \Qcal(\bY) + m'$ where $\bY = \frac{1}{\sigma^2}(\bX'^{-1} - \Id)$ and $m' \in \RR$ is an additive constant, 
    \item $T_\beta(T_\alpha(h))$ is well-defined and is also a quadratic form up to an additive constant, since $\bY' \defeq \sigma^2\bY + \sigma^2\bB^{-1} + \Id = \bX'^{-1} + \sigma^2 \bB^{-1} \succ 0$ and (i) applies.
\end{enumerate}
\end{proposition}
Consider the null inialization $f_0 = 0 = \Qcal(0)$.
Since $\sigma^2 \bA^{-1} + \Id \succ0$, \Cref{prop:sinkhorn-transform} applies with $\bX = 0$ and a simple induction shows that $(f_n, g_n)$ remain quadratic forms for all $n$. Sinkhorn's algorithm can thus be written as an algorithm on positive definite matrices.
\begin{proposition}
\label{prop:bures_sinkhorn_convergence}
Starting with null potentials, Sinkhorn's algorithm is equivalent to the iterations:
\begin{align}
    \label{eq:sinkhorn-fg}
    \bF_{n+1} = \sigma^2 \bA^{-1} + \bG^{-1}_n, \hspace{3em}
    \bG_{n+1} =\sigma^2 \bB^{-1} + \bF^{-1}_{n+1},
\end{align}
 with $\bF_0 = \sigma^2 \bA^{-1} + \Id$ and $\bG_0 = \sigma^2 \bB^{-1} + \Id$.
\end{proposition}
Moreover, the sequence $(\bF_n, \bG_n)$ is contractive (in the matrix operator norm) and converges towards a pair of positive definite matrices $(\bF, \bG)$.
At optimality, the dual potentials are determined up to additive constants $f_0$ and $g_0$: $\frac{f}{2\sigma^2} = \Qcal(\bU) + f_0$ and $\frac{g}{2\sigma^2} = \Qcal(\bV) + g_0$ where $\bU$ and $\bV$ are given by
%
%
\begin{align}
\label{eq:changeofvariable}
    \bF = \sigma^2 \bU + \sigma^2 \bA^{-1} + \Id, \hspace{4em}
    \bG = \sigma^2 \bV + \sigma^2 \bB^{-1} + \Id.
\end{align}
\paragraph{Closed form solution.}
Taking the limit of Sinkhorn's equations \eqref{eq:sinkhorn-fg} along with the change of variable \eqref{eq:changeofvariable},
there exists a pair of optimal potentials determined up to an additive constant:
%
%
\begin{align}
\label{eq:optim-dual}
    \frac{f}{2\sigma^2} = \Qcal(\bU) = \Qcal\left(\frac{1}{\sigma^2}(\bG^{-1} - \Id)\right), \hspace{2em}
    \frac{g}{2\sigma^2} = \Qcal(\bV) = \Qcal\left(\frac{1}{\sigma^2}(\bF^{-1} - \Id)\right),
\end{align}
where $(\bF, \bG)$ is the solution of the fixed point equations
%
%
\begin{align}
    \label{eq:optim-fg}
    \bF = \sigma^2 \bA^{-1} + \bG^{-1}, \hspace{4em}
    \bG =\sigma^2 \bB^{-1} + \bF^{-1}.
\end{align}
Let $\bC \defeq \bA\bG^{-1}$. Combining both equations of \eqref{eq:optim-fg} in one leads to $\bG = \sigma^2 \bB^{-1} + (\bG^{-1} + \sigma^2 \bA^{-1})^{-1}$, which can be shown to be equivalent to 
\begin{equation}\label{eq:fixedpoint-C}
\bC^2 + \sigma^2\bC - \bA\bB = 0.
\end{equation}
Notice that since $\bA$ and $\bG^{-1}$ are positive definite, their product $\bC = \bA\bG^{-1}$ is similar to $\bA\rt\bG^{-1}\bA\rt$. Thus it has positive eigenvalues. \Cref{prop:closed_form_matrix} provides the only feasible solution of \eqref{eq:fixedpoint-C}.
\begin{proposition}
\label{prop:closed_form_matrix}
Let $\sigma^2 \geq 0$ and $\bC$ satisfying \Cref{eq:fixedpoint-C}. Then,
 \begin{align}
 \bC = \left(\bA\bB + \frac{\sigma^4}{4} \Id\right)\rt -\tfrac{\sigma^2}{2} \Id
     = \bA\rt(\bA\rt\bB\bA\rt + \tfrac{\sigma^4}{4} \Id)\rt\bA\mrt -\tfrac{\sigma^2}{2} \Id 
\label{eq:closed-form-C}.
\end{align}
\end{proposition}
\begin{corollary}
\label{cor:closedform-potentials}
The optimal dual potentials of \eqref{eq:optim-dual} can be given in closed form by:
\begin{equation}
    \begin{split}
    \bU =  \frac{\bB}{\sigma^2}(\bC +
    \sigma^2\Id)^{-1} - \frac{\Id}{\sigma^2}
    , 
    \hspace{2em}
    \bV =  (\bC + 
    \sigma^2\Id)^{-1}\frac{\bA}{\sigma^2} - \frac{\Id}{\sigma^2}.
    \end{split}
\end{equation}
Moreover, $\bU$ and $\bV$ remain well-defined even for singular matrices $\bA$ and $\bB$.
\end{corollary}
%
%
%
\paragraph{Optimal transportation plan and $\OT_\sigma$.}
Using \Cref{cor:closedform-potentials} and \eqref{eq:optim-dual}, \Cref{eq:transportion_plan} leads to a closed form expression of $\pi$. To conclude the proof of \Cref{thm:otclosedform}, we introduce lemma \ref{lemma:optim-loss} that computes the $\OT_\sigma$ loss at optimality. Detailed technical proofs are provided in the appendix.
\begin{lemma}
\label{lemma:optim-loss}
Let $\bA, \bB, \bC$ be invertible matrices such that 
$\bH = \left(\begin{smallmatrix}
       \bA & \bC \\
         \bC^\top  & \bB\\
    \end{smallmatrix}\right) \succ 0$.
Let $\alpha = \Ncal(0, \bA), \beta = \Ncal(0, \bB)$, and $\pi = \Ncal(0, \bH)$. Then, 
\begin{align}
    &\int_{\RR^d\times\RR^d} \|x - y\|^2 \d \pi(x, y) = \tr(\bA) + \tr(\bB) - 2\tr(\bC), \label{eq:optim-transport}\\
   &\KL\left(\pi\|\alpha\otimes\beta\right) = \tfrac{1}{2}\left( \log\det \bA + \log\det \bB - \log\det \left(\begin{smallmatrix} \bA & \bC \\ \bC^T & \bB \end{smallmatrix}\right) \right).
   \label{eq:optim-ent}
\end{align}
\end{lemma}

%
\paragraph{Properties of $\OT_{\sigma}$.}
%
\Cref{thm:otclosedform} shows that $\pi$ has a Gaussian density. \Cref{prop:bures_sinkhorn_matrix_problem} allows to reformulate this optimization problem over couplings in $\RR^{d\times d}$ with a positivity constraint. 
\begin{proposition}\label{prop:bures_sinkhorn_matrix_problem}
Let $\alpha = \Ncal(0, \bA), \beta = \Ncal(0, \bB)$, and $\sigma^2 > 0$. Then, 
\begin{align}
\OT_\sigma(\alpha, \beta) &=\hspace{-1.5em} \underset{\bC : \left(\begin{smallmatrix} \bA & \bC \\ \bC^T & \bB \end{smallmatrix}\right) \geq 0}{\min}\! \Big\{ \tr(\bA) + \tr(\bB) - 2\tr(\bC) + \sigma^2( \log\det \bA\bB -  \log\det\left(\begin{smallmatrix} \bA & \bC \\ \bC^T & \bB \end{smallmatrix}\right)) \Big\}\label{eq:ent_bures}\\
 &= \underset{\bK \in \RR^{d\times d} : \|\bK\|_{\text{op}} \leq 1}{\min} \tr\bA + \tr\bB - 2 \tr{\bA\rt \bK \bB\rt} - \sigma^2 \ln \det(\Id - \bK \bK^\top).\label{eq:primal_K}
\end{align}
Moreover, both \eqref{eq:ent_bures} and \eqref{eq:primal_K} are convex problems.
\end{proposition}
We now study the convexity and differentiability of $\OT_{\sigma}$, which are more conveniently derived from the dual problem of \eqref{eq:ent_bures} given as a positive definite program:
\begin{proposition}\label{prop:dual_problem}
The dual problem of \eqref{eq:ent_bures} can be written with no duality gap as
\begin{align}\label{eq:dual_ent_bures}
&\underset{\bF, \bG \succ 0}{\max}\Big\{
   \dotp{\Id - \bF}{\bA}  + \dotp{\Id - \bG}{\bB} + \sigma^2\log\det\left(\frac{\bF\bG - \Id}{\sigma^4}\right)  + \sigma^2 \log\det\bA \bB  + 2d\sigma^2\Big\}.
\end{align}
\end{proposition}
\citet{feydy2019} showed that on compact spaces, the gradient of $\OT_{\sigma}$ is given by the optimal dual potentials. This result was later extended by \citet{janati20} to sub-Gaussian measures with unbounded supports. The following proposition re-establishes this statement for Gaussians. 
\begin{proposition}\label{prop:envelop_theorem}
Assume $\sigma > 0$ and consider the pair $\bU, \bV$ of \Cref{cor:closedform-potentials}. Then
\begin{itemize}
\item[(i)] The optimal pair $(\bF^*, \bG^*)$ of \eqref{eq:dual_ent_bures} is a solution to the fixed point problem \eqref{eq:optim-fg},
\item[(ii)]
$\mathfrak{B}_{\sigma^2}$ is differentiable and: $\nabla \mathfrak{B}_{\sigma^2}(\bA, \bB) = -(\sigma^2\bU, \sigma^2\bV)$. Thus:
$\nabla_\bA \mathfrak{B}_{\sigma^2}(\bA, \bB) = \Id - \bB\rt\left((\bB\rt\bA\bB\rt + \frac{\sigma^4}{4}\Id)\rt + \frac{\sigma^2}{2}\Id\right)^{-1}\bB\rt$,
\item[(iii)] $(\bA, \bB) \mapsto \mathfrak{B}_{\sigma^2}(\bA, \bB)$ is convex in $\bA$ and in $\bB$ but not jointly.
\item[(iv)] For a fixed $\bB$ with its spectral decomposition $\bB = \bP\Sigma\bP^\top$, the function $\phi_\bB: \bA \mapsto \mathfrak{B}_{\sigma^2}(\bA, \bB)$ is minimized at $\bA_0 = \bP(\Sigma - \sigma^2\Id)_+\bP^\top$ where the thresholding operator $_+$ is defined by $x_+ = \max(x, 0)$ for any $x \in \RR$ and extended element-wise to diagonal matrices.
\end{itemize}
\end{proposition}
When $\bA$ and $\bB$ are not singular, by letting $\sigma\to 0$ in $\nabla_\bA \mathfrak{B}_{\sigma^2}(\bA, \bB)$, we recover the gradient of the Bures metric given in \eqref{eq:bures-gradient}. Moreover, (iv) illustrates the entropy bias of $\mathfrak{B}_{\sigma^2}$. \citet{feydy2019} showed that it can be circumvented by considering the Sinkhorn divergence: 
\begin{align}
    \label{eq:def-div}
S_\sigma: (\alpha, \beta) \mapsto \OT_{\sigma}(\alpha, \beta) - \frac{1}{2}( \OT_{\sigma}(\alpha, \alpha) + \OT_{\sigma}(\beta, \beta))
\end{align}
which is non-negative and equals 0 if and only if $\alpha = \beta$. Using the differentiability and convexity of $S_\sigma$ on sub-Gaussian measures~\citep{janati20}, we conclude this section by showing that the debiased Sinkhorn barycenter of Gaussians remains Gaussian:
\begin{theorem}
\label{thm:barycenters}
Consider the restriction of $\OT_\sigma$ to the set of sub-Gaussian measures $\Gcal \defeq \{\mu \in \Pcal_2 | \exists q > 0,\, \Esp_\mu(e^{q\|X\|^2}) < +\infty\}$ and let $K$ Gaussian measures $\alpha_k \sim \Ncal(\ba_k, \bA_k)$ with a sequence of positive weights $(w_k)_k$ summing to 1. Then, the weighted debiased barycenter defined by:
\begin{equation}
\label{eq:def-bar}
    \beta \defeq \argmin_{\beta \in \Gcal} \sum_{k=1}w_k S_\sigma(\alpha_k, \beta)
\end{equation}
is a Gaussian measure given by $\Ncal\left(\sum_{k=1}^K w_k\ba_k, \bB\right)$ where $\bB \in \Scal^d_+$ is a solution of the equation:
\begin{equation}
    \label{eq:bar-fixed-point}
    \sum_{k=1}^K w_k (\bB\rt \bA_k\bB\rt + \frac{\sigma^4}{4}\Id)\rt = (\bB^2 + \frac{\sigma^4}{4}\Id)\rt
\end{equation}
\end{theorem}
\section{Entropy Regularized OT between Unbalanced Gaussians}
We proceed by considering a more general setting, in which measures $\alpha, \beta \in \Mcal^+_2(\RR^d)$ have finite integration masses $m_\alpha = \alpha(\RR^d)$ and $m_\beta=\beta(\RR^d)$ that are not necessarily the same. Following \citep{chizat17}, we define entropy-regularized unbalanced OT as:
\begin{equation}
\label{eq:unbalanced_ot}
    \UOT_\sigma(\alpha, \beta) \defeq \hspace{-.5em} \inf_{\pi \in \Mcal^+_2} \int_{\RR^{d}\times \RR^d}\hspace{-2.em} \| x - y \|^2 \dd\pi(x,y) + 2\sigma^2 \KL(\pi \Vert \alpha\otimes\beta)
    + \gamma \KL(\pi_1\Vert \alpha) + \gamma \KL(\pi_2\Vert \beta),
\end{equation}
where $\gamma > 0$ and $\pi_1$, $\pi_2$ are the marginal distributions of the coupling $\pi \in \Mcal_2^+(\RR^2\times\RR^d)$.
\paragraph{Duality and optimality conditions.}
By definition of the $\KL$ divergence, the term $\KL(\pi \Vert \alpha\otimes\beta)$ in \eqref{eq:unbalanced_ot} is finite if and only if $\pi$ admits a density with respect to $\alpha\otimes\beta$. Therefore \eqref{eq:unbalanced_ot} can be formulated as a variational problem:
\begin{equation}\label{eq:unbalanced_ot_variational}
    \begin{split}
    \UOT_\sigma(\alpha, \beta) \defeq  \inf_{r \in \Lcal_1(\alpha\otimes\beta)}\Big \{& \int_{\RR^{d}\times \RR^d} \| x - y \|^2 r(x, y)\dd\alpha(x)\dd\beta(y)\\ 
    & + 2\sigma^2 \KL(r \Vert \alpha\otimes\beta)
    + \gamma \KL(r_1\Vert \alpha) + \gamma \KL(r_2\Vert \beta)\Big\},
    \end{split}
\end{equation}
where $r_1 \defeq \int_{\RR^d} r(.,y)\dd\beta(y)$ and $r_2 \defeq \int_{\RR^d} r(x,.)\dd\alpha(x)$ correspond to the marginal density functions and the Kullback-Leibler divergence is defined  as: $\KL(f\Vert\mu) = \int_{\RR^d}(f\log(f) + f -1)\dd\mu$.
As in~\citep{chizat17}, Fenchel-Rockafellar duality holds and \eqref{eq:unbalanced_ot_variational} admits the following dual problem:
\begin{equation}
\label{eq:unbalanced_ot_dual}
    \begin{split}
    \UOT_\sigma(\alpha, \beta) = \sup_{\substack{f \in \Lcal_\infty(\alpha)\\g\in\Lcal_\infty(\beta)}}\Big\{&
    \gamma\int_{\RR^d}(1 - e^{-\frac{f}{\gamma}})\dd\alpha  + \gamma\int_{\RR^d}(1 -e^{-\frac{g}{\gamma}})\dd\beta \\ &-2\sigma^2\int_{\RR^d\times\RR^d}(e^{\frac{-\|x-y\|^2 + f(x) + g(y)}{2\sigma^2}}-1)\dd\alpha(x)\dd\beta(y)\Big\},
    \end{split}
\end{equation}
for which the necessary optimality conditions read, with  $\tau \defeq \frac{\gamma}{\gamma + 2\sigma^2}$:
\begin{align}
    \label{eq:opt-conditions-unbalanced}
    \frac{f(x)}{2\sigma^2} \overset{a.s}{=} -\tau\log\int_{\RR^d} e^{\frac{g(y) - \|x - y \|^2}{2\sigma^2}}\dd\beta(y), \hspace{1em}
    \frac{g(x)}{2\sigma^2} \overset{a.s}{=} -\tau\log\int_{\RR^d} e^{\frac{f(y) - \|x - y \|^2}{2\sigma^2}}\dd\alpha(y).
\end{align}
Moreover, if such a pair of dual potentials exists, then the optimal transportation plan is given by
\begin{equation}
    \label{eq:transport-plan-unbalanced}
    \frac{\dd \pi}{\dd\alpha \otimes \dd\beta}(x, y) = e^{\frac{f(x) + g(y) - \|x - y \|^2}{2\sigma^2}}.
\end{equation}
The following proposition provides a simple formula to compute $\UOT_\sigma$ at optimality. It shows that it is sufficient to know the total transported mass $\pi(\RR^d \times \RR^d)$.
\begin{proposition}
Assume there exists an optimal transportation plan $\pi^*$, solution of \eqref{eq:unbalanced_ot}. Then
\label{prop:unbalanced-loss-opt}
\begin{equation}
    \label{eq:unbalanced-loss-opt}
    \UOT_\sigma(\alpha, \beta) = \gamma (m_\alpha + m_\beta) + 2\sigma^2m_\alpha m_\beta - 2(\sigma^2 + \gamma)\pi^*(\RR^d \times \RR^d).
\end{equation}
\end{proposition}

\paragraph{Unbalanced OT for scaled Gaussians.}
Let $\alpha$ and $\beta$ be unbalanced Gaussian measures. Formally, $\alpha =  m_\alpha \Ncal(\ba, \bA)$ and $\beta = m_\beta  \Ncal(\bb, \bB)$ with $m_\alpha, m_\beta > 0$. Unlike balanced OT, $\alpha$ and $\beta$ cannot be assumed to be centered without loss of generality. However, we can still derive a closed form formula for $\UOT_{\sigma}(\alpha, \beta)$ by considering quadratic potentials of the form
\begin{align}
    \label{eq:potentials-as-quads}
    \frac{f(\bx)}{2\sigma^2} = -\frac{1}{2}(x^\top \bU \bx - 2x^\top\bu) + \log(m_u), \quad
    \frac{g(x)}{2\sigma^2} = -\frac{1}{2}(x^\top \bV\bx - 2x^\top\bv)  + \log(m_v).
\end{align}
Let $\sigma$ and $\gamma$ be the regularization parameters as in \Cref{eq:unbalanced_ot_variational}, and $\tau \defeq \frac{\gamma}{2\sigma^2 + \gamma}$, $\lambda \defeq \frac{\sigma^2}{1-\tau} = \sigma^2 + \frac{\gamma}{2}$. Let us define the following useful quantities:
\begin{align}
    \mu &= \begin{pmatrix}
            \ba + \bA\bX^{-1}(\bb -\ba) \\
            \bb + \bB\bX^{-1}(\ba -\bb)
            \end{pmatrix}\\
    \bH &=     \begin{pmatrix}
     (\Id + \frac{1}{\lambda}\bC)(\bA-\bA\bX^{-1}\bA) &\bC +  (\Id + \frac{1}{\lambda}\bC)\bA\bX^{-1}\bB  \\ 
     \bC^\top + (\Id + \frac{1}{\lambda}\bC^\top)\bB\bX^{-1}\bA   & (\Id +  \frac{1}{\lambda}\bC^\top)(\bB-\bB\bX^{-1}\bB)
     \end{pmatrix}\\
    m_\pi &= \sigma^{\frac{d\sigma^2}{\gamma + \sigma^2}} \left(m_\alpha m_\beta \det(\bC)\sqrt{\frac{\det(\widetilde{\bA}\widetilde{\bB})^\tau}{\det( \bA\bB)}}\right)^{\frac{1}{\tau + 1}}
    \frac{e^{-\frac{\|\ba - \bb\|_{\bX^{-1}}^2}{2(\tau + 1)}}}{\sqrt{\det(\bC - \frac{2}{\gamma}\widetilde{\bA}\widetilde{\bB})}}, 
\end{align}
with
\begin{align*}
    \bX &= \bA + \bB + \lambda\Id, &\widetilde{\bA} = \frac{\gamma}{2}(\Id - \lambda(\bA + \lambda\Id)^{-1}), \\
    \widetilde{\bB} &= \frac{\gamma}{2}(\Id - \lambda(\bB + \lambda\Id)^{-1}), &\bC = \left(\frac{1}{\tau}\widetilde{\bA}\widetilde{\bB} + \frac{\sigma^4}{4}\Id\right)^{\frac{1}{2}} - \frac{\sigma^2}{2}\Id.
\end{align*}
\begin{theorem}
\label{thm:unbalanced}
Let $\alpha = m_\alpha\Ncal(\ba, \bA)$ and $\beta = m_\beta \Ncal(\bb, \bB)$ be two unbalanced Gaussian measures. Let $\tau = \frac{\gamma}{2\sigma^2 + \gamma}$ and $\lambda \defeq \frac{\sigma^2}{1-\tau} = \sigma^2 + \frac{\gamma}{2}$ and $\mu$, $\bH$, and $m_\pi$ be as above. Then
\begin{enumerate}[(i)]
\item The unbalanced optimal transport plan, minimizer of \eqref{eq:unbalanced_ot}, is also an unbalanced Gaussian over $\RR^d \times \RR^d$ given by $
    \pi = m_\pi  \Ncal\left(\mu,
             \bH
            \right)$,
\item $\UOT_\sigma$ can be obtained in closed form using \Cref{prop:unbalanced-loss-opt} with $\pi(\RR^d \times \RR^d) = m_\pi$.
\end{enumerate}
%
\end{theorem}
%
%
%
\begin{remark}
The exponential term in the closed form formula above provides some intuition on how transportation occurs in unbalanced OT. When the difference between the means is too large, the transported mass $m_\pi^\star$ goes to $0$ and thus no transport occurs. However for fixed means $\ba, \bb$, when $\gamma \to +\infty$, $\bX^{-1} \to 0$ and the exponential term approaches 1. 
\end{remark}

\section{Numerical Experiments}
\label{sec:experiments}
\paragraph{Empirical validation of the closed form formulas.}
\Cref{fig:convergence} illustrates the convergence towards the closed form formulas of both theorems. For each dimension $d$ in [5, 10], we select a pair of Gaussians $\alpha \sim \Ncal(\ba, \bA)$ and $\beta \sim m_\beta\Ncal(\bb, \bB)$ with $m_\beta$ equals 1 (balanced) or 2 (unbalanced) and randomly generated means $\ba, \bb$ (uniform in $[-1, 1]^d$) and covariances $\bA, \bB \in S^d_{++}$ following the Wishart distribution $W_d(0.2 * \Id, d)$. We generate i.i.d datasets $\alpha_n \sim \Ncal(\ba, \bA)$ and $\beta_n \sim m_\beta\Ncal(\bb, \bB)$ with $n$ samples and compute $\OT_{\sigma}$ / $\UOT_\sigma$. We report means and $\pm$ shaded standard-deviation areas over 20 independent trials for each value of $n$.
\begin{figure}[H]
    \centering
    \includegraphics[width=0.9\linewidth]{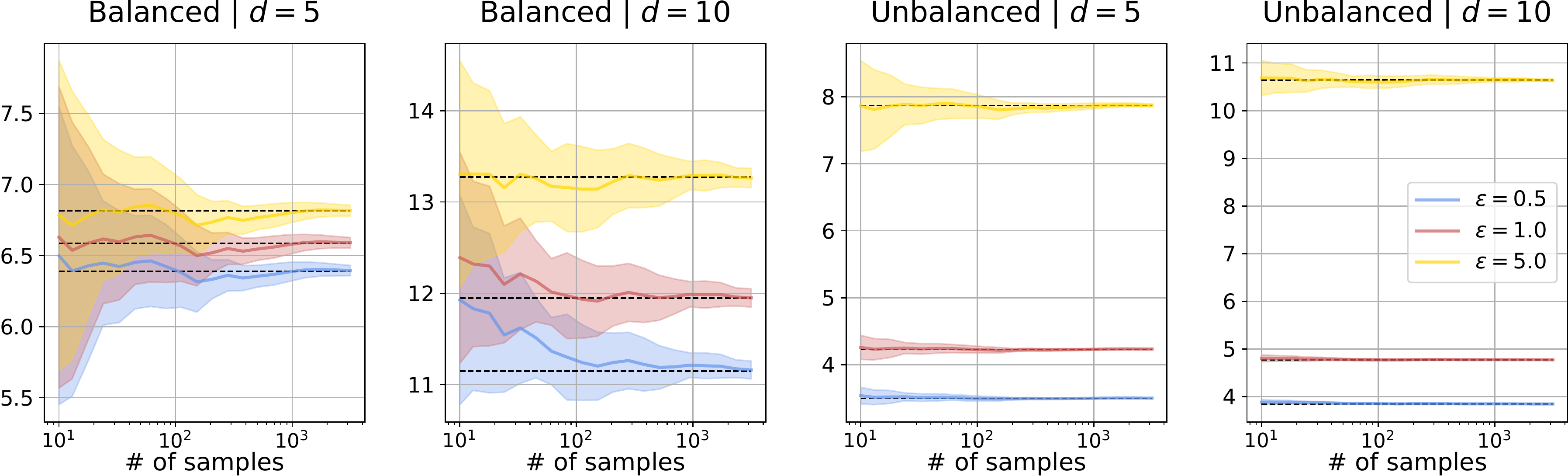}
    \caption{Numerical convergence the (n-samples) empirical estimation of $\OT(\alpha_n, \beta_n)$ computed using Sinkhorn's algorithm towards the closed form of $\OT_{\sigma}(\alpha, \beta)$ and $\UOT_\sigma(\alpha, \beta)$ (the theoretical limit is dashed) given by \Cref{thm:otclosedform} and \Cref{thm:unbalanced} for random Gaussians $\alpha, \beta$. For unbalanced OT, $\gamma = 1$.}
    \label{fig:convergence}
    \vskip-.5cm
\end{figure}
\paragraph{Transport plan visualization with $d=1$.} \Cref{fig:plans} confronts the expected theoretical plans (contours in black) given by \cref{thm:otclosedform,thm:unbalanced} to empirical ones (weights in shades of red) obtained with Sinkhorn's algorithm using 2000 Gaussian samples. The density functions (black) and the empirical histograms (red) of $\alpha$ (resp. $\beta$) with 200 bins are displayed on the left (resp. top) of each transport plan. The red weights are computed via a 2d histogram of the transport plan returned by Sinkhorn's algorithm with (200 x 200) bins. Notice the blurring effect of $\varepsilon$ and increased mass transportation of the Gaussian tails in unbalanced transport with larger $\gamma$.
\begin{figure}[H]
    \centering
    \includegraphics[width=0.24\linewidth]{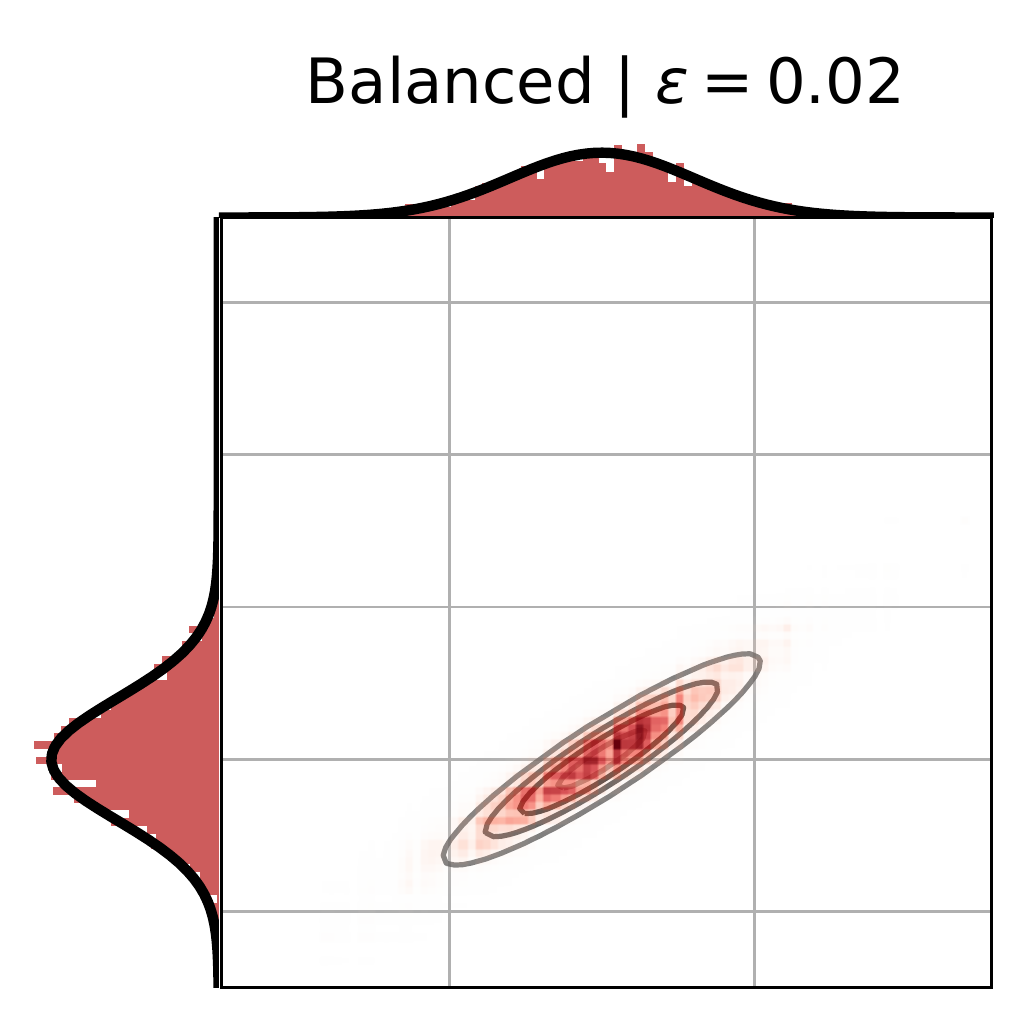}
    \includegraphics[width=0.24\linewidth]{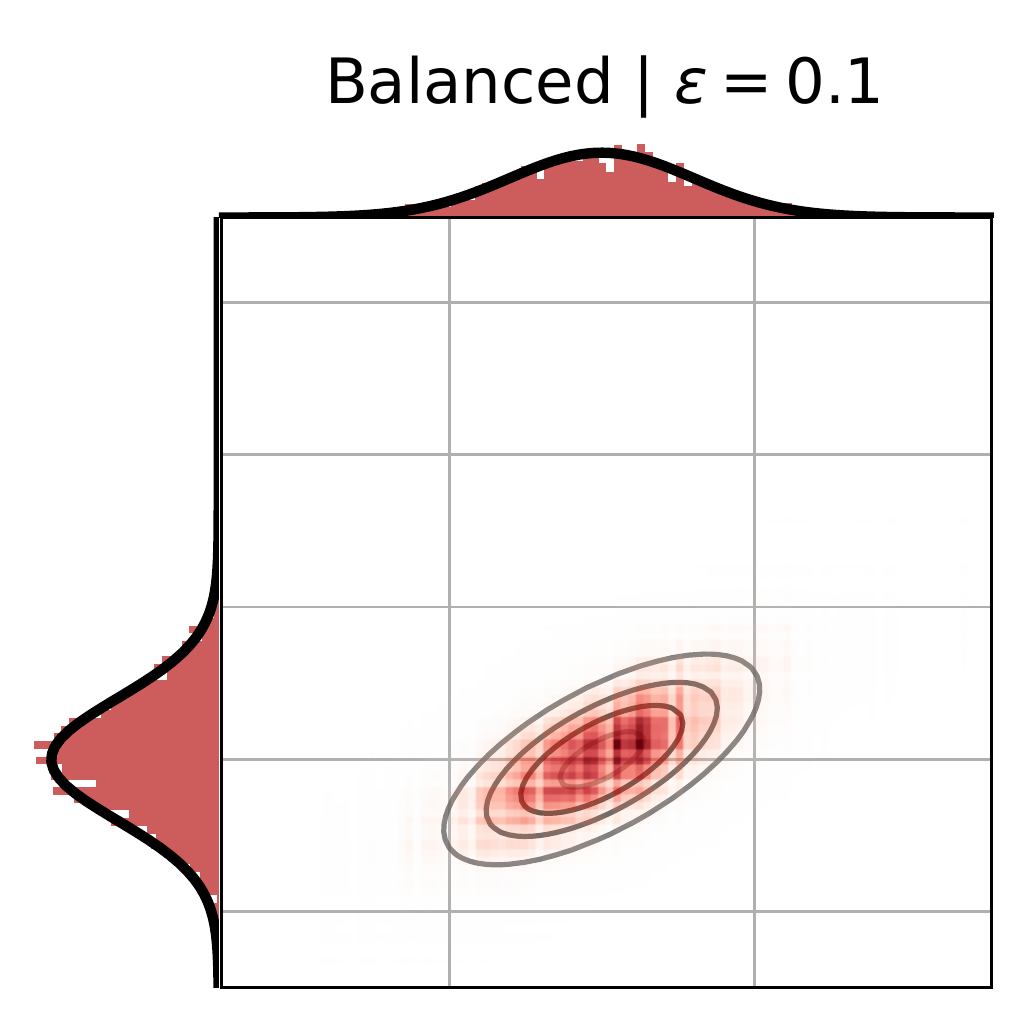}
   \includegraphics[width=0.24\linewidth]{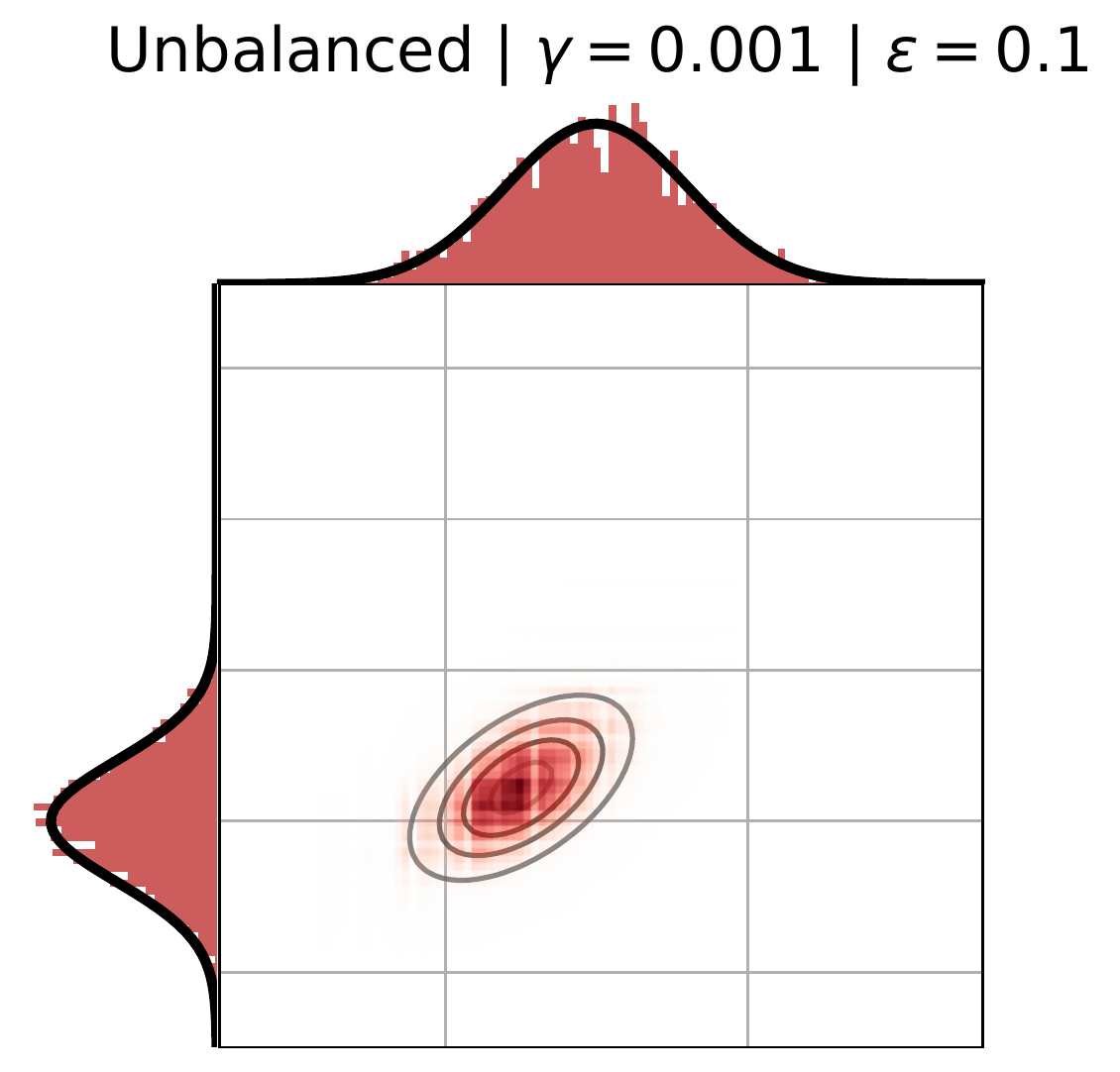}
    \includegraphics[width=0.24\linewidth]{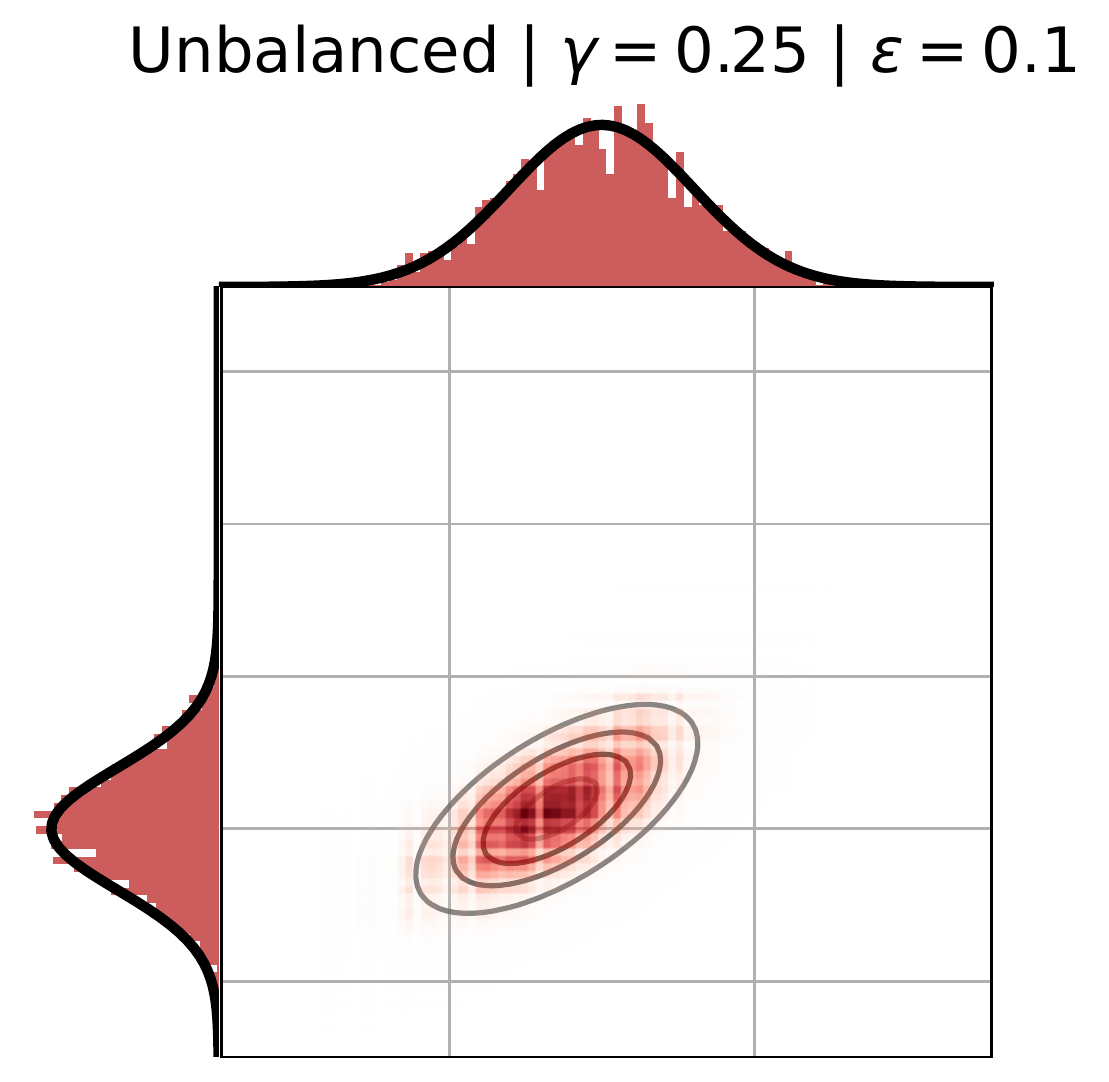}
    \caption{Effect of $\varepsilon$ in balanced OT and $\gamma$ in unbalanced OT. Empirical plans (red) correspond to the expected Gaussian contours depicted in black. Here $\alpha = \Ncal(0, 0.04)$ and $\beta = m_\beta\Ncal(0.5, 0.09)$ with $m_\beta = 1$ (balanced) and $m_\beta = 2$ (unbalanced). In unbalanced OT, the right tail of $\beta$ is not transported, and the mean of the transportation plan is shifted compared to that of the balanced case -- as expected from \Cref{thm:unbalanced} specially for low $\gamma$.
    \label{fig:plans}}
\end{figure}
\paragraph{Empirical estimation of the closed form mean and covariance of the unbalanced transport plan}
\Cref{fig:rebuttal} illustrates the convergence towards the closed form formulas of $\mu$ and $\bH$ of theorem \ref{thm:unbalanced}. For each dimension $d$ in [1, 2, 5, 10], we select a pair of Gaussians $\alpha \sim \Ncal(\ba, \bA)$ and $\beta \sim m_\beta\Ncal(\bb, \bB)$ with $m_\beta= 1.1$ and randomly generated means $\ba, \bb$ (uniform in $[-1, 1]^d$) and covariances $\bA, \bB \in S^d_{++}$ following the Wishart distribution $W_d(0.2 * \Id, d)$. We generate i.i.d datasets $\alpha_n \sim \Ncal(\ba, \bA)$ and $\beta_n \sim m_\beta\Ncal(\bb, \bB)$ with $n$ samples and compute $\OT_{\sigma}$ / $\UOT_\sigma$. We set $\varepsilon \defeq 2\sigma^2 - 0.5$ and $\gamma = 0.1$. Using the obtained empirical Sinkhorn transportation plan, we computed its empirical mean $\mu_n$ and covariance matrix $\Sigma_n$ and display their relative $\ell_{\infty}$ distance to $\mu$ and $\bH$ ($\Sigma$ in the figure) of theorem \ref{thm:unbalanced}. The means and $\pm$ sd intervals are computed over 50 independent trials for each value of $n$.
\begin{figure}[H]
    \centering
    \includegraphics[width=0.9\linewidth]{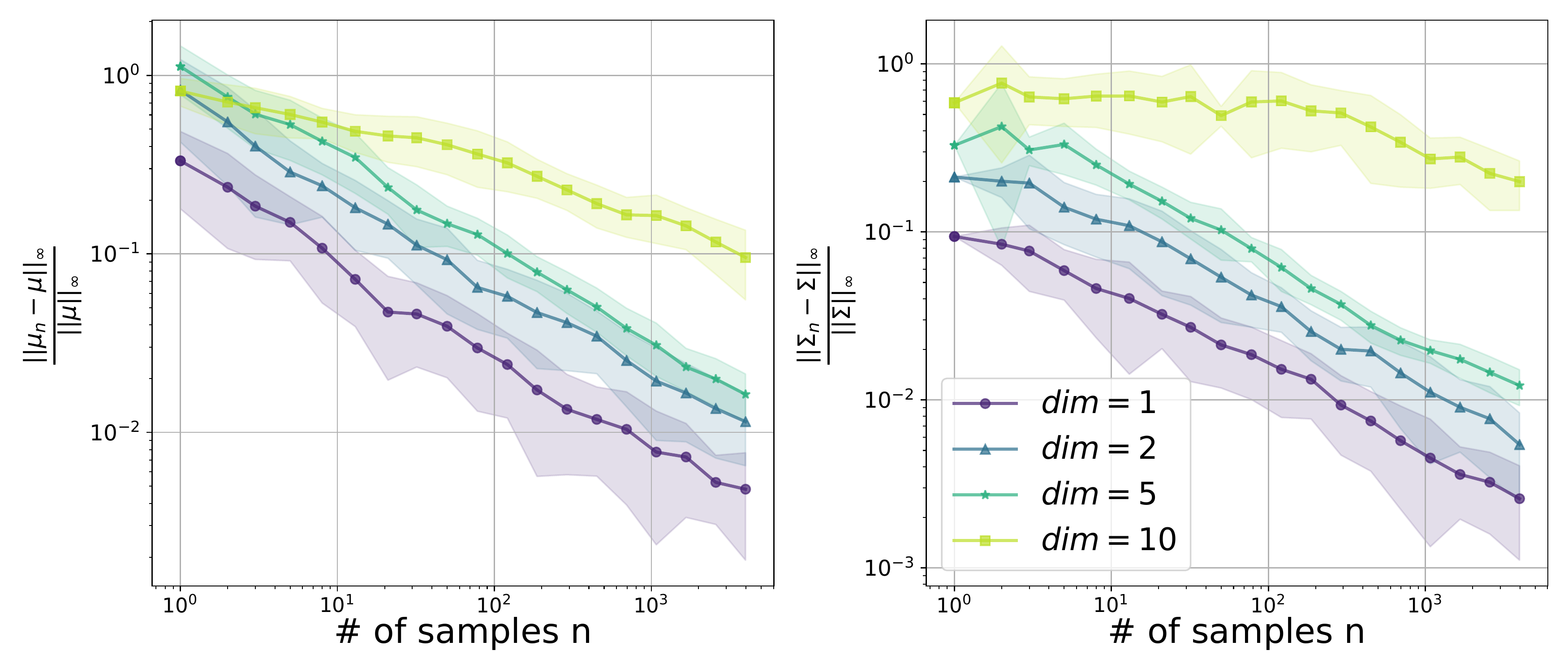}
    \caption{Numerical convergence the (n-samples) empirical estimation of the theoretical mean $\mu$ and covariance $\bH$ of theorem \ref{thm:unbalanced}. Empirical moments are computed computed using Sinkhorn's algorithm.
    \label{fig:rebuttal}}
    \vskip-.5cm
\end{figure}
%


\section*{Broader Impact}
We expect this work to benefit research on sample complexity issues in regularized optimal transport, such as \citep{genevay19sample} for balanced regularized OT, and future work on unbalanced regularized OT. By providing the first continuous test-case, we hope that researchers will be able to better test their theoretical bounds and benchmark their methods.

\section*{Acknowledgments}
 H. Janati, B. Muzellec and M. Cuturi were supported by a ``Chaire d’excellence de l’IDEX Paris Saclay''. H. Janati acknowledges the support of the ERC Starting Grant SLAB ERC-YStG-676943. The work of G. Peyr\'e was supported by the European Research Council (ERC project NORIA) and by the French government under management of Agence Nationale de la Recherche as part of the ``Investissements d’avenir'' program, reference ANR19-P3IA-0001 (PRAIRIE 3IA Institute).

\bibliographystyle{apalike}
\bibliography{references}

\iftoggle{supplementary}{
    \newpage
\section*{Appendix}\label{sec:appendix}

\subsection{The Newton-Schulz algorithm}


\begin{wrapfigure}{R}{0.4\textwidth}
\vskip-.7cm
\begin{minipage}{0.4\textwidth}
\begin{algorithm}[H]
\caption{NS Monge Iterations}\label{alg:NewtonSchulz}
\begin{algorithmic} 
\REQUIRE PSD matrix $\bA, \bB$, $\epsilon > 0$
\STATE $\bY \leftarrow \frac{\bB}{(1+\epsilon)\Vert\bB\Vert}, \bZ\leftarrow \frac{\bA}{(1+\epsilon)\Vert\bA\Vert}$
\WHILE{not converged}
\STATE $\bT \leftarrow (3 \eye - \bZ \bY)/2$
\STATE $\bY \leftarrow \bY\bT$
\STATE $\bZ \leftarrow \bT\bZ$
\ENDWHILE
\STATE $\bY \leftarrow \sqrt{\frac{\Vert\bB\Vert}{\Vert\bA\Vert}}\bY$, $\bZ \leftarrow \sqrt{\frac{\Vert\bA\Vert}{\Vert\bB\Vert}}\bZ$
\ENSURE $\bY = \bT^{\bA\bB}$, $\bZ =  \bT^{\bB\bA}$
\end{algorithmic}
\end{algorithm}
\end{minipage}
\vskip-.8cm
\end{wrapfigure}

The main bottleneck in computing $\bT^{\bA\bB}$ is that of computing matrix square roots. This can be performed using singular value decomposition (SVD) or, as suggested in \citep{muzellec2018elliptical}, using Newton-Schulz (NS) iterations~\citep[§5.3]{higham08functions}. In particular, Newton-Schulz iterations have the advantage of yielding both roots, and inverse roots. Hence, to compute $\bT^{\bA\bB}$, one would run NS a first time to obtain $\bA\rt$ and $\bA\mrt$, and a second time to get $(\bA\rt\bB\bA\rt)\rt$.

In fact, as a direct application of \citep[Theorem 5.2]{higham08functions}, one can even compute both $\bT^{\bA\bB}$ and $\bT^{\bB\bA} = \left(\bT^{\bA\bB}\right)^{-1}$ in a single run by initializing the Newton-Schulz algorithm with $\bA$ and $\bB$, as in \Cref{alg:NewtonSchulz}. Using \eqref{eq:bures-gradient}, and noting that $\Bures(\bA, \bB) = \tr\bA + \tr\bB - 2 \tr (\bT^{\bA\bB}\bA)$, this implies that a single run of NS is sufficient to compute $\Bures(\bA, \bB)$, $\nabla_\bA\Bures(\bA, \bB)$ and $\nabla_\bB\Bures(\bA, \bB)$ using basic matrix operations. The main advantage of Newton-Schultz over SVD is that it its efficient scalability on GPUs, as illustrated in \Cref{fig:newton-schulz}.

Newton-Schulz iterations are quadratically convergent under the condition $\|\Id - \left(\begin{smallmatrix} \bA & 0 \\ 0 & \bB\end{smallmatrix}\right)^2\| < 1$, as shown in \citep[Theorem 5.8]{higham08functions}. To meet this condition, it is sufficient to rescale $\bA$ and $\bB$ so that their norms equal $(1+\varepsilon)^{-1}$ for some $\varepsilon > 0$, as in the first step of \Cref{alg:NewtonSchulz} (which can be skipped if $\|\bA\| < 1 $ (resp. $\|\bB\| < 1$)). Finally, the output of the iterations are scaled back, using the homogeneity (resp. inverse homogoneity) of \cref{eq:monge_map} w.r.t. $\bA$ (resp. $\bB)$.

A rough theoretical analysis shows that both Newton-Schulz and SVD have a $O(d^3)$ complexity in the dimension. \Cref{fig:newton-schulz} compares the running times of Newton-Schulz iterations and SVD on CPU or GPU used to compute both $\bA\rt$ and $\bA\mrt$. We simulate a batch of positive definite matrices $\bA$ following the Wishart distribution $W(\Id_d, d)$ to which we add $0.1 \Id$ to avoid numerical issues when computing inverse square roots. We display the average run-time of 50 different trials along with its $\pm$ std interval. Notice the different magnitudes between CPUs and GPUs. As a termination criterion, we first run EVD to obtain $\bA\rt_{evd}$ and $\bA\mrt_{evd}$   and stop the Newton-Schultz algorithm when its n-th running estimate $\bA\rt_n$ verifies: $\|\bA\rt_n - \bA\rt_{evd}\|_1 \leq 10^{-4}$. Notice the different order of magnitude between CPUs and GPUs. Moreover, the computational advantage of Newton-Schultz on GPUs can be further increased when computing multiple square roots in parallel.
\begin{figure}[H]
    \centering
    \includegraphics[width=0.8\linewidth]{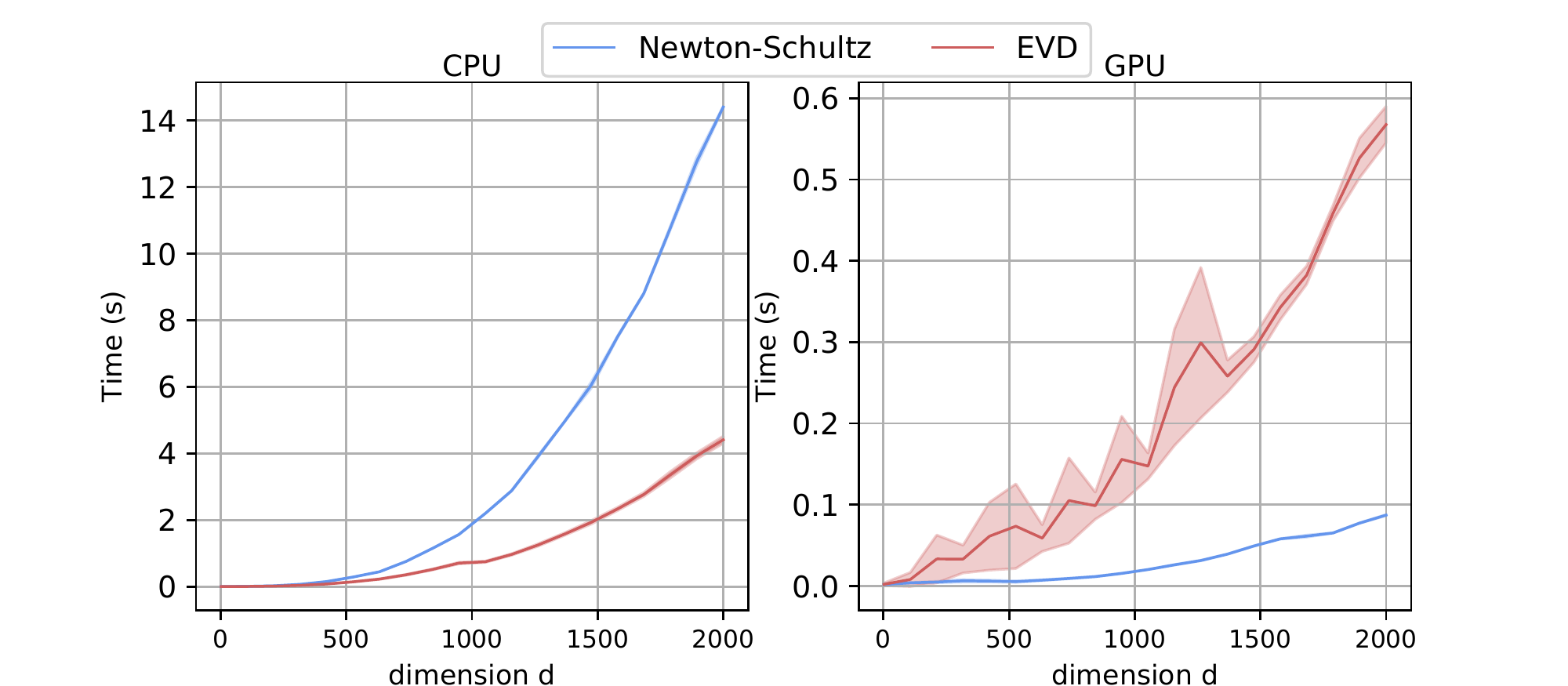}
    \caption{Average run-time of Newton-Schulz and EVD to compute on CPUs and GPUs.}
    \label{fig:newton-schulz}
\end{figure}
%
%
\subsection{Effects of regularization strength.} 
We provide numerical experiments to illustrate the behaviour of transportation plans and corresponding distances as $\sigma$ goes to $0$ or to infinity. As can be seen from \cref{eq:otclosedform}, when $\sigma \to 0$ we recover the Wasserstein-Bures distance \eqref{eq:wasserstein_bures}, and the optimal transportation plan converges to the Monge map \eqref{eq:monge_map}. When on the contrary $\sigma \to \infty$, Sinkhorn divergences $\mathfrak{S}_\varepsilon(\alpha, \beta) \defeq \OT_\varepsilon(\alpha, \beta) - \tfrac{1}{2}(\OT_\varepsilon(\alpha, \alpha) + \OT_\varepsilon(\beta, \beta))$ convergence to MMD with a $-c$ kernel (where $c$ is the optimal transport ground cost)~\citep{genevay18a}. With a $-\ell_2$ kernel, MMD is degenerate and equals $0$ for centered measures. 

\begin{figure}[t]
    \centering
    \includegraphics[width=\linewidth]{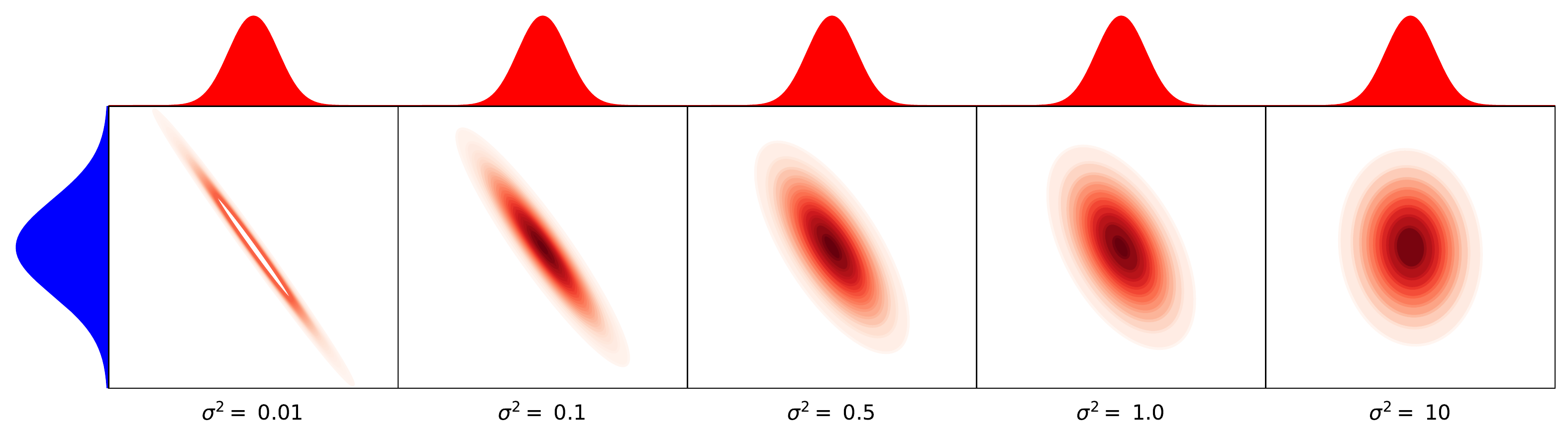}
    \vskip-.2cm
    \caption{Effect of regularization on transportation plans. When $\sigma$ goes to $0$ (left), the transportation plan concentrates on the graph of the linear Monge map. When $\sigma$ goes to infinity (right), the transportation plan converges to the independent coupling.}
    \label{fig:reg-effect}
\end{figure}

\begin{figure}[t]
    \centering
    \includegraphics[width=.5\linewidth]{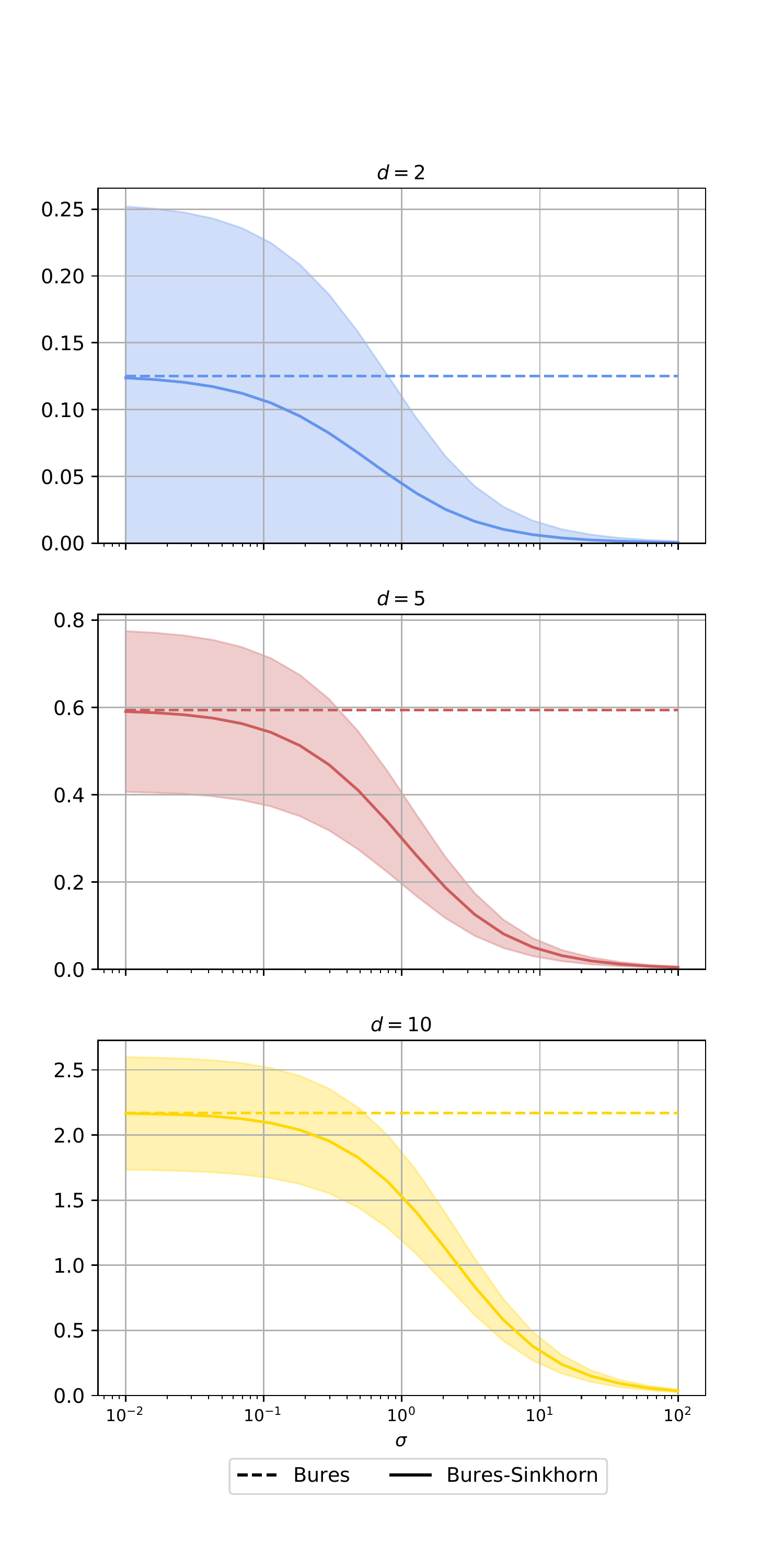}
    \caption{Numerical convergence of $\Bures_\sigma(\bA, \bB) - \tfrac{1}{2}(\Bures_\sigma(\bA, \bA) + \Bures_\sigma(\bB, \bB))$ to $\Bures(\bA, \bB)$ as $\sigma$ goes to $0$ and to $0$ as $\sigma$ goes to infinity.}
    \label{fig:reg-convergence}
\end{figure}

\begin{figure}[H]
    \centering
    \includegraphics[width=\linewidth]{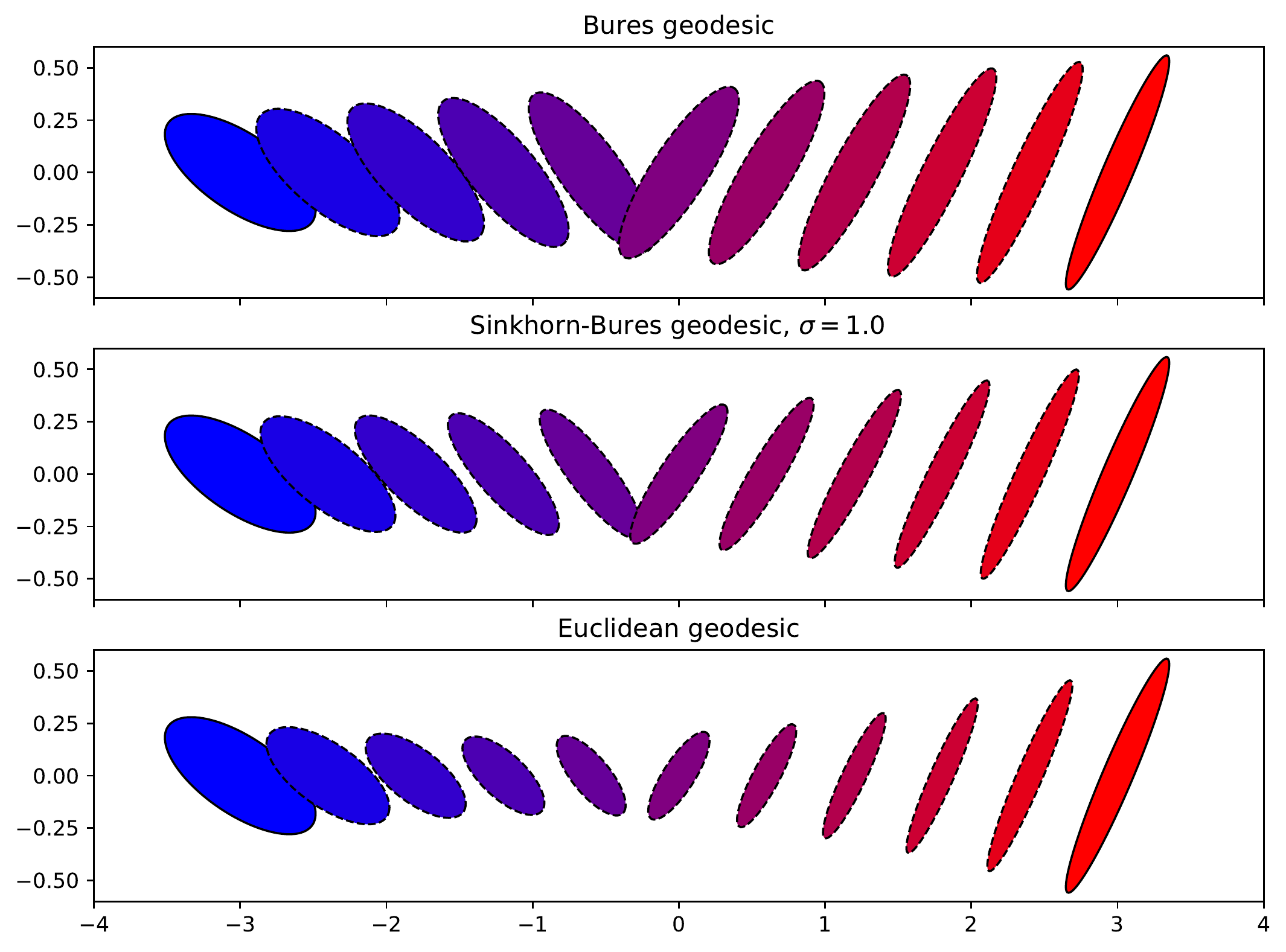}
    \caption{Bures, Sinkhorn-Bures, and Euclidean geodesics. Sinkhorn-Bures trajectories converge to Bures geodesics as $\sigma$ goes to $0$, and to Euclidean geodesics as $\sigma$ goes to infinity.}
    \label{fig:geodesics}
\end{figure}


\subsection{Proofs of technical results}

We provide in this appendix the proofs of the results in the paper, as well as some technical lemmas used in solving Sinkhorn's equations in closed form.


\paragraph{Proof of \Cref{lemma:centering}.}

\begin{proof}
 Let $\dd\bar{\alpha}(x) = \dd\alpha(x + \ba)$ (resp. $\dd\bar{\beta}(y) = \dd\beta(y + \bb)$, $\dd\bar{\pi}(x, y) = \dd\pi(x + \ba, y + \bb)$, such that $\bar{\alpha}, \bar{\beta}$ and $\bar{\pi}$ are centered. Then, $\forall \pi \in \Pi(\alpha, \beta)$,
\begin{enumerate}[(i)]
\item  $\bar{\pi}\in \Pi(\bar{\alpha}, \bar{\beta})$,
\item  $\KL(\pi \| \alpha\otimes\beta) = \KL(\bar{\pi} \| \bar{\alpha}\otimes\bar{\beta})$
\item $\int_{\RR^d\times\RR^d}  \| x - y \|^2 \dd\bar{\pi}(x, y) = \int_{\RR^d\times\RR^d}  \| (x - \ba) - (y - \bb) \|^2 \dd\pi(x, y) = \| \ba - \bb \|^2 + \int_{\RR^d\times\RR^d}  \| x - y \|^2 \dd\pi(x, y) $
\end{enumerate}
Plugging (i)-(iii) into \eqref{eq:ent_ot}, we get $\OT_\sigma(\alpha, \beta) = \OT_\sigma(\bar{\alpha}, \bar{\beta}) +  \| \ba - \bb \|^2$. 
\end{proof}

\paragraph{Proof of \Cref{prop:sinkhorn-transform}.}

\begin{proof}
The exponent inside the integral can be written as:
\begin{align*}
    e^{\tfrac{-\|x- y\|^2}{2\sigma^2} + h(y)} \dd\alpha(y) &\propto  e^{\tfrac{-\|x- y\|^2}{2\sigma^2} - \frac{1}{2}(y^\top\bX y - y^\top \bA^{-1}y)} \dd y \\
    &\propto e^{-\tfrac{1}{2}(y^\top(\frac{\Id}{\sigma^2} + \bX +\bA^{-1})y) + \frac{x^\top y}{\sigma^2}} \dd y
\end{align*}
which is integrable if and only if $\bX + \bA^{-1} + \frac{1}{\sigma^2}\Id 
\succ 0$.
Moreover, up to a multiplicative factor, the exponentiated Sinkhorn transform is equivalent to a Gaussian convolution of an exponentiated quadratic form. Lemma \ref{lem:convolution} applies:  
\begin{align*}
    e^{-T_\alpha(h)} &= \int_{\RR^d} e^{\tfrac{-\|x- y\|^2}{2\sigma^2} + f(y)}\dd \alpha(y) \\
&\propto \int_{\RR^d} e^{\tfrac{-\|x- y\|^2}{2\sigma^2} + \Qcal(\bX)(y) + \Qcal(\bA^{-1})(y)} \dd y \\
    &\propto \exp\left(\Qcal\left(\tfrac{\Id}{\sigma^2}\right)\right) \star \exp\left(\Qcal(\bX) +  \Qcal(\bA^{-1})\right)\\
         &\propto \exp\left(\Qcal\left(\tfrac{\Id}{\sigma^2}\right)\right) \star \exp\left(\Qcal(\bX + \bA^{-1})\right)\\
         &\propto \exp\left(\Qcal((\Id + \sigma^2\bX + \sigma^2\bA^{-1})^{-1}(\bX + \bA^{-1}))\right).\\
         &\propto \exp\left(\Qcal(\frac{1}{\sigma^2}\bX'^{-1}(\bX' - \Id))\right).\\
         &\propto \exp\left(\Qcal(\frac{1}{\sigma^2}(\Id - \bX'^{-1}))\right).
\end{align*}
Therefore $T_\alpha(h)$ is up to an additive constant given by $\Qcal(\frac{1}{\sigma^2}(\bX'^{-1} - \Id))$.

Finally, since $\bB$ and $\bX'$ are positive definite, the positivity condition of $\bY'$ holds and $T_\beta$ can be applied again to get $T_\beta(T_\alpha(h))$.
\end{proof}


\paragraph{Proof of \Cref{prop:bures_sinkhorn_convergence}.}

\begin{proof}
Let $\bU_0 = \bV_0 = 0$. Applying \Cref{prop:sinkhorn-transform}
to the initial pair of potentials $\Qcal(\bU_0), \Qcal(\bV_0)$ leads to the sequence of quadratic Sinkhorn potentials
 $\frac{f_n}{2\sigma^2} = \Qcal(\bU_n) $ and
 $\frac{g_n}{2\sigma^2} = \Qcal(\bV_n)$ where:
 \begin{align*}
    \begin{split}
     \bV_{n+1} = \frac{1}{\sigma^2}((\sigma^2 \bU_{n} + \sigma^2 \bA^{-1} + \Id)^{-1} - \Id) \\ 
      \bU_{n+1} = \frac{1}{\sigma^2}((\sigma^2 \bV_{n + 1} + \sigma^2\bB^{-1} +\Id)^{-1} - \Id).
      \end{split}
 \end{align*}
The change of variable:
\begin{align*}
 \begin{split}
    \bF_n = \sigma^2 \bU_n + \sigma^2 \bA^{-1} + \Id\\
    \bG_n = \sigma^2 \bV_n + \sigma^2 \bB^{-1} + \Id\\ \end{split}
\end{align*} leads to \eqref{eq:sinkhorn-fg}.

We turn to show that this algorithm converges.
First, note that since $\bF_0, \bG_0 \in \PD$, a straightforward induction shows that $\forall n \geq 0, \bF_n, \bG_n \in \PD$. Next, let us write the decoupled iteration on $\bF$:
\begin{equation}\label{eq:decoupled_F}
    \bF \leftarrow \sigma^2\bA^{-1} + (\sigma^2\bB^{-1}+\bF^{-1})^{-1}
\end{equation}
Let $\forall \bX \in \PD, \phi(\bX) \defeq \sigma^2\bA^{-1} + (\sigma^2\bB^{-1}+\bX^{-1})^{-1} \in \PD$. The first differential of $\phi$ admits the following expression:  
\begin{equation}
   \forall \bX \in \PD, \forall \bH \in \RR^{d \times d}, D\phi(\bX)[\bH] = (\Id + \sigma^2 \bX\bB^{-1})^{-1}\bH(\sigma^2 \bB^{-1}\bX + \Id)^{-1}.
\end{equation}
Hence, $\normop{D\phi(\bX)[\bH]} \leq \normop{(\Id + \sigma^2 \bX\bB^{-1})^{-1}}^2\normop{\bH}$. Plugging $\bH = \Id$, we get that $\normop{D\phi(\bX)} = \normop{(\Id + \sigma^2 \bX\bB^{-1})^{-1}}^2$.  Finally, by matrix similarity 
\begin{equation*}
\normop{(\Id + \sigma^2 \bX\bB^{-1})^{-1}} = \normop{(\Id + \sigma^2 \bB\mrt\bX\bB\mrt)^{-1}} < 1\enspace,
\end{equation*}
which implies that $\normop{D\phi(\bX)} < 1$ for $\bX \in \PD$ and $\sigma^2 > 0$. The same arguments hold for the iterates $(\bG_n)_{n \geq 0}$.

From \eqref{eq:decoupled_F} and using Weyl's inequality, we can bound the smallest eigenvalue of $\bF_n$ from under: $\forall n, \lambda_d(\bF_n) \geq \frac{\sigma^2}{\lambda_1(\bA)}$ (where $\lambda_d(\bF)$ is the smallest eigenvalue of $\bF$ and $\lambda_1(\bA)$ is the biggest eigenvalue of $\bA$). Hence, the iterates live in $\mathcal{A} \defeq \PD \cap \{\bX : \lambda_d(\bX) \geq \frac{\sigma^2}{\lambda_1(\bA)}\}$. Finally, for all $\bX \in \mathcal{A}$, 
\begin{align*}
    \normop{(\Id + \sigma^2 \bB\mrt\bX\bB\mrt)^{-1}} &= \frac{1}{\lambda_d(\Id + \sigma^2 \bB\smrt\bX\bB\smrt)}\\
    &= \frac{1}{1 + \sigma^2 \lambda_d(\bB\smrt\bX\bB\smrt)}\\
    &\leq \frac{1}{1 + \sigma^2 \lambda_d(\bB^{-1})\lambda_d(\bX)}\\
    &\leq \frac{1}{1 +\frac{\sigma^4}{\lambda_1(\bB)\lambda_1(\bA)}}
\end{align*}
Which proves the uniform bound $\blacksquare$

\end{proof}


\paragraph{Proof of \Cref{prop:closed_form_matrix}.}

\begin{proof}

Combining the two equations in \eqref{eq:optim-fg} yields
\begin{align}
&\bG = \sigma^2 \bB^{-1} + (\bG^{-1} + \sigma^2 \bA^{-1})^{-1} \nonumber \\
& \Leftrightarrow \bG\bA^{-1} = \sigma^2 \bB^{-1}\bA^{-1} + (\bA\bG^{-1} + \sigma^2 \Id)^{-1} \nonumber \\
 &\Leftrightarrow \bC^{-1} = \sigma^2 (\bA\bB)^{-1} + (\bC + \sigma^2 \Id)^{-1} \nonumber\\
&\Leftrightarrow \bC^{-1}(\bC + \sigma^2\Id) = \sigma^2 (\bA\bB)^{-1}(\bC + \sigma^2\Id) + \Id \nonumber\\
&\Leftrightarrow \Id + \sigma^2\bC^{-1} = \sigma^2 (\bA\bB)^{-1}(\bC + \sigma^2\Id) + \Id
 \nonumber\\
&\Leftrightarrow \bC + \sigma^2\Id = \sigma^2 (\bA\bB)^{-1}(\bC + \sigma^2\Id)\bC + \bC
\nonumber\\
&\Leftrightarrow \bC^2 + \sigma^2\bC - \bA\bB = 0.
\end{align}

Given that $\bA$ and $\bG^{-1}$ are positive, their product $\bC = \bA\bG^{-1}$ can be written: $\bA\bG^{-1} = \bA\rt(\bA\rt\bG^{-1}\bA\rt)\bA\mrt$, thus $\bA\bG^{-1}$ is similar to the positive matrix $\bA\rt\bG^{-1}\bA\rt$. Therefore, one can write an eigenvalue decomposition of $\bC = \bP\Sigma\bP^{-1}$ with a positive diagonal matrix $\Sigma$. Substituting in \eqref{eq:fixedpoint-C}, it follows that $\bC$ and $\bA\bB$ share the same eigenvectors with modified eigenvalues. Thus, it is sufficient to find the real roots of the polynomial $x \mapsto x^2 + \sigma^2 x - ab $ with $a, b \in \RR_{++}$ which are given by: $x_1 = -\frac{\sigma^{2}}{2} - \sqrt{ab + \frac{\sigma^4}{4}}$ and $x_2 = -\frac{\sigma^{2}}{2} + \sqrt{ab + \frac{\sigma^4}{4}}$. Since $\bC$ is the product of the positive definite matrices $\bG^{-1}$ and $\bA$, its eigenvalues are all positive. Discarding the negative root, the closed form follows immediately. 

Indeed, by direct calculation, computing the square of the solution $\bC$ leads to the equation \eqref{eq:fixedpoint-C}:
\begin{align*}
\bC^2 &= \bA\bB + \frac{\sigma^4}{2}\Id  - \sigma^2\left(\bA\bB + \frac{\sigma^4}{4} \Id\right)\rt \\
&=\bA\bB - \sigma^2\bC.
\end{align*}
The second equality is obtained by observing that
\begin{align*}
    &(\bA\rt(\bA\rt\bB\bA\rt + \tfrac{\sigma^4}{4}\Id)\rt\bA\mrt)^2 =\bA\rt(\bA\rt\bB\bA\rt + \tfrac{\sigma^4}{4}\Id)\bA\mrt = \bA\bB + \tfrac{\sigma^4}{4}\Id,
\end{align*}
i.e. that 
\begin{equation*}
\left(\bA\bB + \tfrac{\sigma^4}{4} \Id\right)\rt = \bA\rt(\bA\rt\bB\bA\rt +  \tfrac{\sigma^4}{4}\Id)\rt\bA\mrt .  
\end{equation*}
\end{proof}

\paragraph{Proof of \Cref{lemma:optim-loss}}

\begin{proof}
It follows from elementary properties of Gaussian measures that the first and second marginals of $\pi$ are respectively $\alpha$ and $\beta$. Hence,
\begin{align}
    \int_{\RR^d\times\RR^d} \|x - y\|^2 \d \pi(x, y) &= \int_{\RR^d\times\RR^d} \|x\|^2 \d \pi(x, y) + \int_{\RR^d\times\RR^d} \|y\|^2 \d \pi(x, y)  -2\int_{\RR^d\times\RR^d} \dotp{x}{y} \d \pi(x, y)\\
    &=\int_{\RR^d} \|x\|^2 \d \alpha(x) + \int_{\RR^d} \|y\|^2 \d \beta( y)  -2\int_{\RR^d\times\RR^d} \dotp{x}{y} \d \pi(x, y)\\
    &= \tr(\bA) + \tr(\bB) - 2\tr(\bC).
\end{align}
Next, using the closed form expression of the Kullback-Leibler divergence between Gaussian measures,
\begin{align}
  \KL\left(\pi\|\alpha\otimes\beta\right) &= \tfrac{1}{2}\left(\tr\left[\left(\begin{smallmatrix} \bA & 0 \\ 0 & \bB \end{smallmatrix}\right)^{-1}\left(\begin{smallmatrix} \bA & \bC \\ \bC^T & \bB \end{smallmatrix}\right)\right] - 2n + \log\det\left(\begin{smallmatrix} \bA & 0 \\ 0 & \bB \end{smallmatrix}\right) - \log\det\left(\begin{smallmatrix} \bA & \bC \\ \bC^T & \bB \end{smallmatrix}\right) \right)\\
  &=\tfrac{1}{2}\left( \log\det \bA + \log\det \bB - \log\det \left(\begin{smallmatrix} \bA & \bC \\ \bC^T & \bB \end{smallmatrix}\right) \right).
\end{align}
\end{proof}

\paragraph{Optimal transport plan and $\OT_{\sigma}$}
\begin{align*}
\frac{\dd\pi}{\dd x\dd y}(x, y) &= \exp\left(\frac{f(x) + g(y) - \|x - y\|^2}{2\sigma^2}\right) \frac{\dd\alpha}{\dd x}(x) \frac{\dd\beta}{\dd y}(y)\\
&\propto \exp\left(\Qcal(\bA^{-1})(x) + \frac{f(x) + g(y) - \|x - y\|^2}{2\sigma^2} + \Qcal(\bB^{-1})(y)\right) \\
    &\propto \exp\left(\Qcal(\bU + \bA^{-1})(x) + \Qcal(\bV + \bB^{-1})(y) + \Qcal(\begin{smallmatrix}\frac{\Id}{\sigma^2} & - \frac{\Id}{\sigma^2} \\ - \frac{\Id}{\sigma^2} & \frac{\Id}{\sigma^2} \end{smallmatrix})(x, y)\right)\\
    &= \exp\left(\Qcal(\begin{smallmatrix}\bU + \bA^{-1} & 0 \\ 0 & \bV + \bB^{-1} \end{smallmatrix})(x, y) + \Qcal(\begin{smallmatrix}\frac{\Id}{\sigma^2} & - \frac{\Id}{\sigma^2} \\ - \frac{\Id}{\sigma^2} & \frac{\Id}{\sigma^2} \end{smallmatrix})(x, y)\right)\\
    &= \exp\left( \Qcal(\begin{smallmatrix}\frac{\Id}{\sigma^2} + \bU + \bA^{-1} & - \frac{\Id}{\sigma^2}  \\ - \frac{\Id}{\sigma^2} & \frac{\Id}{\sigma^2} +\bV + \bB^{-1} \end{smallmatrix})(x, y)\right)\\
    &= \exp\left( \Qcal(\begin{smallmatrix}\frac{\bF}{\sigma^2} & - \frac{\Id}{\sigma^2}  \\ - \frac{\Id}{\sigma^2} &  \frac{\bG}{\sigma^2} \end{smallmatrix})(x, y)\right)\\
&=\exp\left( \Qcal(\Gamma)(x, y)\right)
\end{align*}
with $\Gamma \defeq \begin{pmatrix}\frac{\bF}{\sigma^2} & - \frac{\Id}{\sigma^2}  \\ - \frac{\Id}{\sigma^2} &  \frac{\bG}{\sigma^2} \end{pmatrix}$. Moreover, since $\frac{\bG}{2\sigma^2} \succ 0$ , and its Schur complement satisfies $\frac{\bF}{\sigma^2} - \frac{1}{\sigma^2}\bG^{-1} = \bA^{-1} \succ 0$, we have that $\Gamma \succ 0$. Therefore $\pi$ is a Gaussian $\Ncal(\bH)$ with the covariance matrix given by the block inverse formula:
\begin{align}
    \bH &= \Gamma^{-1} \\
    &= \sigma^2\begin{pmatrix} 
       (\bF-\bG^{-1})^{-1} & (\bG\bF-\Id)^{-1} \\
         (\bF\bG-\Id)^{-1}  &  (\bG-\bF^{-1})^{-1}\\
    \end{pmatrix} \\
    &= \begin{pmatrix} 
       \bA & \bC \\
         \bC^\top  & \bB\\
    \end{pmatrix},
\end{align}
where we used the optimality equations \eqref{eq:optim-fg} and the definition of $\bC = \bA\bG^{-1}.$

We can now conclude the proof of \Cref{thm:otclosedform} by computing $\OT_\sigma(\alpha, \beta)$ using \Cref{lemma:optim-loss}. 
Let $\bR = \bA^{\frac{1}{2}} \bB \bA^{\frac{1}{2}}$. Using the closed form expression of $\bC$ in \eqref{eq:closed-form-C}, it first holds that
\begin{align}
\label{eq:lemma-Z}
    \begin{split}
        \bZ \defeq \bA^{-\frac{1}{2}} \bC \bA^{\frac{1}{2}}
        = (\bR + \tfrac{\sigma^4}{4} \Id)\rt -\tfrac{\sigma^2}{2} \Id.
    \end{split}
\end{align}
Moreover, since $\bR = \bR^\top$, it holds that $\bZ = \bZ^\top$. Hence,
\begin{align}
    \label{eq:logdet}
    \begin{split}
        \det \left(\begin{smallmatrix} \bA & \bC \\ \bC^T & \bB \end{smallmatrix}\right) = &\det(\bA)\det(\bB - \bC^{\top}\bA^{-1}\bC) \\
        &=\det(\bA^{\frac{1}{2}} \bB \bA^{\frac{1}{2}} - \bA^{\frac{1}{2}}\bC^{\top}\bA^{-1}\bC \bA^{\frac{1}{2}}) \\
        &= \det(\bR - \bZ^\top \bZ) \\
        &= \det(\bR - \bZ^2) \\
        &= \det(\sigma^2 (\bR + \frac{\sigma^4}{4}\Id)^{\frac{1}{2}} - \frac{\sigma^4}{2}\Id)\\
        &=({\frac{\sigma^2}{2}})^d \det((4\bR + \sigma^4 \Id)^{\frac{1}{2}} - \sigma^2 \Id).
        \end{split}
\end{align}
Since the matrices inside the determinant commute, we can use the identity $\bP - \bQ = (\bP^2 - \bQ^2)(\bP + \bQ)^{-1}$ to get rid of the negative sign. \Cref{eq:logdet} then becomes:
\begin{align*}
({\frac{\sigma^2}{2}})^d \det((4\bR + \sigma^4 \Id)^{\frac{1}{2}} - \sigma^2 \Id) 
&= ({\frac{\sigma^2}{2}})^d \det(4\bR)\det\left(((4\bR + \sigma^4 \Id)^{\frac{1}{2}} + \sigma^2 \Id)^{-1}\right) \\
&= ({2\sigma^2})^d \det(\bA\bB)\det\left(((4\bR + \sigma^4 \Id)^{\frac{1}{2}} + \sigma^2 \Id)^{-1}\right).
\end{align*}
Plugging this expression in \eqref{eq:optim-ent}, the determinant of $\bA$ and $\bB$ cancel out and we finally get:
\begin{align*}
\begin{split}
\Bures_{\sigma}(\bA, \bB) &= \tr(\bA) + \tr(\bB) - \tr(4\bA^{\frac{1}{2}} \bB \bA^{\frac{1}{2}} +\sigma^4 \Id )^{\frac{1}{2}} + d\sigma^2 - \\ & \sigma^2d\log(2\sigma^2) + \sigma^2\log\det\left((4\bA^{\frac{1}{2}} \bB \bA^{\frac{1}{2}} + \sigma^4 \Id)^{\frac{1}{2}} + \sigma^2 \Id\right).
\end{split}
\end{align*}


\paragraph{Proof of \Cref{prop:bures_sinkhorn_matrix_problem}}

\begin{proof}
Using \Cref{lemma:optim-loss}, \cref{eq:ent_ot} becomes
\begin{align*}
    \mathfrak{B}_\sigma^2(\bA, \bB) = \underset{\bC : \left(\begin{smallmatrix} \bA & \bC \\ \bC^T & \bB \end{smallmatrix}\right) \geq 0}{\min} & \Big\{ \tr(\bA) + \tr(\bB) - 2\tr(\bC) + \sigma^2( \log\det \bA + \log\det \bB - \log\det\left(\begin{smallmatrix} \bA & \bC \\ \bC^T & \bB \end{smallmatrix}\right)) \Big\},
\end{align*}
which gives \cref{eq:ent_bures}. Let us now prove \cref{eq:primal_K}. A necessary and sufficient condition for $\left(\begin{smallmatrix} \bA & \bC \\ \bC^T & \bB \end{smallmatrix}\right) \geq 0$ is that there exists a contraction $\bK$ (i.e. $\bK \in \RR^d : \|\bK\|_{\text{op}} \leq 1$) such that $\bC = \bA\rt \bK \bB\rt$~\citep[Ch. 1]{bhatia07psd}.\footnote{Another immediate NSC is $\bA \geq \bC\bB^{-1}\bC^T$} With this parameterization, we have (using Schur complements) that
\begin{align*}
    \det\left(\begin{smallmatrix} \bA & \bC \\ \bC^T & \bB \end{smallmatrix}\right) &= \det \bB \det (\bA - \bC \bB^{-1}\bC^\top)\\
    &= \det \bB \det \bA  \det (\Id - \bK \bK^\top)
\end{align*}
Hence, injecting this in \Cref{eq:ent_bures}, we have the following equivalent problem:
\begin{align}
   \mathfrak{B}_\sigma^2(\bA, \bB) = \underset{\bK \in \RR^{d\times d} : \|\bK\|_{\text{op}} \leq 1}{\min} \tr\bA + \tr\bB - 2 \tr{\bA\rt \bK \bB\rt} - \sigma^2 \ln \det(\Id - \bK \bK^\top)
\end{align}
Let's prove that both problems are convex.
\begin{itemize}
\item \eqref{eq:ent_bures}: The set $\{\bC : \left( \begin{smallmatrix} \bA & \bC\\ \bC^T &\bB \end{smallmatrix} \right) \geq 0 \}$ is convex, since $\left( \begin{smallmatrix} \bA & \bC_1\\ \bC_1^T & \bB \end{smallmatrix} \right) \geq 0$ and $\left( \begin{smallmatrix} \bA & \bC_2\\ \bC_2^T&  \bB \end{smallmatrix} \right) \geq 0$ implies that $\left( \begin{smallmatrix} \bA & (1-\theta)\bC_1 + \theta\bC_2 \\ (1-\theta)\bC_1^T + \theta\bC_2^T & \bB \end{smallmatrix} \right) = (1-\theta)\left( \begin{smallmatrix} \bA & \bC_1\\ \bC_1^T &  \bB \end{smallmatrix} \right) + \theta\left( \begin{smallmatrix} \bA & \bC_2\\ \bC_2^T & \bB \end{smallmatrix} \right) \geq 0.$
Following the same decomposition, the concavity of the $\log\det$ function implies that $\bC \rightarrow \log\det  \left( \begin{smallmatrix} \bA & \bC\\ \bC^T &\bB \end{smallmatrix} \right)$ is concave, and hence that the objective function of \eqref{eq:ent_bures} is convex.
\item \eqref{eq:primal_K}: The ball $\Bcal_{\text{op}} \defeq \{\bK \in \RR^{d\times d} : \|\bK\|_{\text{op}} \leq 1\}$ is obviously convex. Hence, there remains to prove that $f(\bK) : \bK \in \Bcal_{\text{op}} \rightarrow \log\det(\Id - \bK\bK^\top)$ is concave. Indeed, it holds that $f(\bK) =
\log\det\left(\begin{smallmatrix} \Id & \bK \\ \bK^T & \Id \end{smallmatrix}\right)$. Hence, $\forall \bK, \bH \in \Bcal_{\text{op}}, \forall t \in [0, 1]$, 
\begin{align*}
    f((1-t)\bK + t \bH) &= \log\det \left\{(1-t)\left(\begin{smallmatrix} \Id & \bK \\ \bK^T & \Id \end{smallmatrix}\right) + t\left(\begin{smallmatrix} \Id & \bH \\ \bH^T & \Id \end{smallmatrix}\right) \right\}\\
    &\geq (1-t)\log\det \left(\begin{smallmatrix} \Id & \bK \\ \bK^T & \Id \end{smallmatrix}\right) + t\log\det\left(\begin{smallmatrix} \Id & \bH \\ \bH^T & \Id \end{smallmatrix}\right)\\
    &= (1-t) f(\bK) + t f(\bH),
\end{align*}
where the second line follows from the concavity of $\log\det$.
\end{itemize}
\end{proof}


\paragraph{Proof of \Cref{prop:dual_problem}}

\begin{proof}
By \Cref{prop:bures_sinkhorn_matrix_problem}, \eqref{eq:ent_bures} is convex, hence strong duality holds. Ignoring the terms not depending on $\bC$, problem \eqref{eq:ent_bures} can be written using the redundant parameterization $\bX = \left(\begin{smallmatrix} \bX_1 & \bX_2 \\ \bX_3 & \bX_4 \end{smallmatrix}\right)$:
\begin{align}\label{eq:ent_bures_relaxed}
    \mathfrak{D}(\bA, \bB) &\defeq \underset{\substack{\bX \succ 0 \\
    \bX_1 = \bA, \bX_4 = \bB }}{\min}
      - \tr(\bX_2)  - \tr(\bX_3)
    - \sigma^2 \log\det\left(\bX\right) \\
        &=\underset{\substack{\bX \succ 0 \\
    \bX_1 = \bA, \bX_4 = \bB }}{\min}
      - \langle \bX, \left(\begin{smallmatrix} 0 & \Id\\ \Id & 0\end{smallmatrix}\right)\rangle
    - \sigma^2 \log\det\left(\bX\right) \\
        &=\underset{\substack{\bX \succ 0 \\
    \bX_1 = \bA, \bX_4 = \bB }}{\min} \Fcal(\bX),
\end{align}
where the functional $\Fcal$ is convex. Moreover, its Legendre transform is given by:
\begin{align*}
\Fcal^\star(\bY)& = \max_{\substack{\bX \succ 0}}
\langle \bX, \bY +  \left(\begin{smallmatrix} 0 & \Id\\ \Id & 0\end{smallmatrix}\right) \rangle + \sigma^2 \log\det\left(\bX\right) \\
&=
\left(-\sigma^2\log\det)^\star \left(\bY +  \left(\begin{smallmatrix} 0 & \Id\\ \Id & 0\end{smallmatrix}\right)
\right)\right) \\
&= \sigma^2(-\log\det)^\star \left(\frac{1}{\sigma^2}\left(\bY +  \left(\begin{smallmatrix} 0 & \Id\\ \Id & 0\end{smallmatrix}\right)
\right)\right) \\
&= - \sigma^2{\log\det} \left(-\frac{1}{\sigma^2}\left(\bY +  \left(\begin{smallmatrix} 0 & \Id\\ \Id & 0\end{smallmatrix}\right)
\right)\right) - 2\sigma^2 d \\
&= - \sigma^2{\log\det} \left(-\left(\bY +  \left(\begin{smallmatrix} 0 & \Id\\ \Id & 0\end{smallmatrix}\right)
\right)\right) - 2d(\sigma^2 - \sigma^2\log(\sigma^2)).
\end{align*}
Let $\Hcal$ be the linear operator $\Hcal: \bX \mapsto (\bX_1, \bX_4)$. Its conjugate operator is defined on $\Scal^d_{++} \times \Scal^d_{++}$ and is given by $\Hcal^\star(\bF, \bG) =\left(\begin{smallmatrix} \bF &0 \\ 0 & \bG\end{smallmatrix}\right)$. Therefore, Fenchel's duality theorem leads to:
\begin{align*}
\label{eq:dual_ent_bures_simplified}
\mathfrak{D}(\bA, \bB) &=\underset{\bF, \bG \succ 0}{\max} 
  -\langle \bF, \bA\rangle - \langle \bG, \bB\rangle -
   \Fcal^\star\left(-\Hcal^\star(\bF, \bG)\right)\\
  &=\underset{\bF, \bG \succ 0}{\max} 
  -\langle \bF, \bA\rangle - \langle \bG, \bB\rangle + \sigma^2
   \log\det\left(\begin{smallmatrix} \bF & -\Id \\ -\Id & \bG\end{smallmatrix} \right) +  2d(\sigma^2 - \sigma^2\log(\sigma^2))
\\
    &= \underset{\bF, \bG \succ 0}{\max}
  - \langle \bF, \bA\rangle  - \langle \bG, \bB\rangle + {\sigma^2} \log\det\left(\bF\bG - \Id\right) + 2d(\sigma^2 - \sigma^2\log(\sigma^2))
\end{align*}
Where the last equality follows from the fact that $\Id$ and $\bG$ commute.
Therefore, reinserting the discarded trace terms, the dual problem of \eqref{eq:ent_bures} can be written as
\begin{align}
&\underset{\bF, \bG \succ 0}{\max}\Big\{
   -\dotp{\bF}{\bA}  - \dotp{\bG}{\bB} + \sigma^2\log\det\left(\bF\bG - \Id\right) \nonumber \\
   & +\tr(\bA) + \tr(\bB) +\sigma^2\log\det\bA \bB  + 2 d\sigma^2(1 - \log\sigma^2))\Big\}.
\end{align}
\end{proof}


\paragraph{Proof of \Cref{prop:envelop_theorem}}

\begin{proof}
(i) \emph{Optimality:} Canceling out the gradients in \cref{eq:dual_ent_bures} leads to the following optimality conditions:
\begin{align}
    \begin{split}
    \label{eq:optimality-conditions}
    -A + {\sigma^2} \bG(\bF\bG - \Id)^{-1} = 0 \\
    -B + {\sigma^2} (\bF\bG - \Id)^{-1}\bF = 0,
    \end{split}
\end{align}
i.e.
\begin{align}
    \begin{split}
    \label{eq:optimality-conditions_2}
    \bF &= \sigma^2 \bA^{-1} + \bG^{-1} \\
    \bG &= \sigma^2 \bB^{-1} + \bF^{-1}
    \end{split}
\end{align}
Thus $(\bF, \bG)$ is a solution of the Sinkhorn fixed point equation \eqref{eq:optim-fg}. 

(ii) \emph{Differentiabilty:} Using Danskin's theorem on problem \eqref{eq:dual_ent_bures} leads to the formula of the gradient as a function of the optimal dual pair $(\bF, \bG)$. Indeed, keeping in mind that $\nabla_{\bA} \log\det(\bA) = -\bA^{-1}$ and using the change of variable of \Cref{prop:bures_sinkhorn_convergence}, we recover the dual potentials of \Cref{cor:closedform-potentials}:
\begin{align*}
    \nabla \mathfrak{B}_{\sigma^2}(\bA, \bB) &= \left(\Id - \bF^{*} + \sigma^2\bA^{-1} , \Id - \bG^{*} + \sigma^2\bB^{-1}\right) \\
    &= -\sigma^2 (\bU, \bV)
\end{align*}
Using \Cref{cor:closedform-potentials}, it holds that
\begin{align*}
    \nabla_{\bA} \mathfrak{B}_{\sigma^2}(\bA, \bB) &= -\sigma^2{\bU}\\
    &= \Id - \bB(\bC + \sigma^2\Id)^{-1} \\
    &= \Id - \bB\left((\bA\bB + \frac{\sigma^4}{4}\Id)\rt + \frac{\sigma^2}{2}\Id\right)^{-1} \\
    &= \Id - \bB\rt\left((\bB\rt\bA\bB\rt + \frac{\sigma^4}{4}\Id)\rt + \frac{\sigma^2}{2}\Id\right)^{-1}\bB\rt \\
    &= \Id - \bB\rt \left(\bD\rt + \frac{\sigma^2}{2}\Id\right)^{-1}\bB\rt,
\end{align*}
where $\bD\defeq \bB\rt\bA\bB\rt + \frac{\sigma^4}{4}\Id$.

(iii) \emph{Convexity:} Assume without loss of generality that $\bB$ is fixed and let $G: \bB \mapsto  \nabla_{\bA} \mathfrak{B}_{\sigma^2}(\bA, \bB)$. As long as $\sigma > 0$, $G$ is differentiable as a composition of differentiable functions. Let's show that the Hessian of $\psi: \bA \mapsto  \mathfrak{B}_{\sigma^2}(\bA, \bB)$ is a positive quadratic form. Take a direction $\bH \in \Scal^d_{+}$. It holds:
\begin{align*}
    \nabla^2_{\bA}\mathfrak{B}_{\sigma^2}(\bA, \bB)(\bH, \bH) &= \langle \bH, \jac_{G}(\bA)(\bH)\rangle \\
    &= \tr(\bH \jac_{G}(\bA)(\bH)).
\end{align*}
For the sake of clarity, let's write $G(\bA) = \Id - L(W(\phi(\bA)))$ with the following intermediary functions:
\begin{align*}
L: \bA &\mapsto \bB\rt\bA\bB\rt  \\
Q: \bA &\mapsto \bA\rt \\
\phi: \bA &\mapsto Q(L(\bA) + \frac{\sigma^4}{4}\Id) \\
W: \bA &\mapsto (\bA + \frac{\sigma^2}{2}\Id)^{-1}.
\end{align*}
Moreover, their derivatives are given by:
\begin{align*}
    \jac_L(\bA)(\bH) &= \bB\rt\bH\bB\rt \\
    \jac_W(\bA)(\bH) &= -  (\bA + \frac{\sigma^2}{2}\Id)^{-1}\bH  (\bA + \frac{\sigma^2}{2}\Id)^{-1}\\
    \jac_Q(\bA)(\bH) &= \bZ,
\end{align*}
where $\bZ \in \Scal^d_{+}$ is the unique solution of the Sylvester equation: $\bZ\bA\rt + \bA\rt\bZ = \bH$.

Using the chain rule:
\begin{align*}
    \jac_{G}(\bA)(\bH) &= -\jac_L(W(\phi(\bA)))(\jac_W(\phi(\bA))(\jac_\phi(\bA)(\bH))) \\
    &= -\bB\rt \jac_W(\phi(\bA))(\jac_\phi(\bA)(\bH)) \bB\rt\\
    &= \bB\rt  \left(\phi(\bA) + \frac{\sigma^2}{2}\Id\right)^{-1} \jac_\phi(\bA)(\bH)  \left(\phi(\bA) + \frac{\sigma^2}{2}\Id\right)^{-1} \bB\rt\\
   &= \bB\rt  \left(\bD\rt + \frac{\sigma^2}{2}\Id\right)^{-1} \jac_\phi(\bA)(\bH)  \left(\bD\rt + \frac{\sigma^2}{2}\Id\right)^{-1} \bB\rt.
\end{align*}
Again using the chain rule:
\begin{align*}
    \bY\defeq \jac_\phi(\bA)(\bH) &= \jac_Q(L(\bA) +\frac{\sigma^4}{4}\Id)((\jac_L(\bA))(\bH)) \\
    &= \jac_Q(L(\bA) +\frac{\sigma^4}{4}\Id)(\bB\rt\bH\bB\rt)\\
    &= \jac_Q(\bD)(\bB\rt\bH\bB\rt).
\end{align*}
Therefore, $\bY \succ 0 $ is the unique solution of the Sylvester equation:
\begin{equation*}
    \bY\bD\rt + \bD\rt\bY = \bB\rt\bH\bB\rt.
\end{equation*}
Combining everything:
\begin{align*}
\nabla^2_{\bA}\mathfrak{B}_{\sigma^2}(\bA, \bB)(\bH, \bH) &= \langle \bH, \jac_{G}(\bA)(\bH)\rangle \\
    &= \tr\left(\bH \jac_{G}(\bA)(\bH)\right)\\
    &= \tr\left(\bH \bB\rt  \left(\bD\rt + \frac{\sigma^2}{2}\Id\right)^{-1}\bY \left(\bD\rt + \frac{\sigma^2}{2}\Id\right)^{-1} \bB\rt\right) \\
    &= \tr\left(\bB\rt\bH \bB\rt  \left(\bD\rt + \frac{\sigma^2}{2}\Id\right)^{-1}\bY \left(\bD\rt + \frac{\sigma^2}{2}\Id\right)^{-1}\right).
\end{align*}
Since $\bH$ and $\bY$ are positive, the matrices $\bB\rt\bH \bB\rt$ and $\left(\bD\rt + \frac{\sigma^2}{2}\Id\right)^{-1}\bY \left(\bD\rt + \frac{\sigma^2}{2}\Id\right)^{-1}$ are positive semi-definite as well. Their product is similar to a positive semi-definite matrix, therefore the trace above is non-negative. 

Given that $\bA$ and $\bH$ are arbitrary positive semi-definite matrices, it holds that
$$ \nabla^2_{\bA}\mathfrak{B}_{\sigma^2}(\bA, \bB)(\bH, \bH)\geq 0$$
Therefore, $\bA \mapsto \mathfrak{B}_{\sigma^2}(\bA, \bB)$ is convex.

\emph{Counter-example of joint convexity:} If $\mathfrak{B}_{\sigma^2}$ were jointly convex , then $\delta \defeq: \bA \rightarrow \mathfrak{B}_{\sigma^2}(\bA, \bA)$ would be a convex function.

In the 1-dimensional case with $\sigma = 1$, one can see that this would be equivalent to $x \rightarrow \ln((x^2 + 1)\rt + 1) - (x^2 + 1)\rt$ being convex, whereas it is in fact strictly concave.

\emph{(iv) Minimizer of $\phi_\bB$}
With fixed $\bB$, cancelling the gradient of $\phi_\bB \defeq: \bA \mapsto \mathfrak{B}_{\sigma^2}(\bA, \bB)$ leads to $\bA = \bB - \sigma^2\Id$ which is well defined if and only if $\bB \succeq\sigma^2\Id$. However, if $\bB - \sigma^2\Id$ is not positive semi-definite, write the eigenvalue decomposition: $\bB = \bP\Sigma\bP^\top$ and define $\bA_0\defeq \bP(\Sigma - \sigma^2\Id)_+\bP^\top$ where the operator $x_+ = \max(x, 0)$ is applied element-wise.  Then: 
\begin{align*}
\nabla_\bA \phi_\bB(\bA_0) &=
\Id - \bP\Sigma\rt\bP^\top\left((\bP(\Sigma^2 - \sigma^2\Sigma)_+ \bP^\top + \frac{\sigma^4}{4}\Id)\rt + \frac{\sigma^2}{2}\Id\right)^{-1} \bP\Sigma\rt\bP^\top \\
&= \Id - \bP\Sigma\rt\left(((\Sigma^2 - \sigma^2\Sigma)_+  + \frac{\sigma^4}{4}\Id)\rt + \frac{\sigma^2}{2}\Id\right)^{-1} \Sigma\rt\bP^\top \\
&= \Id - \bP\Sigma\rt\left((\Sigma - \sigma^2\Id)_+   + \sigma^2\Id\right)^{-1} \Sigma\rt\bP^\top \\
&= \bP(\Id - \Sigma\rt\left((\Sigma - \sigma^2\Id)_+   + \sigma^2\Id\right)^{-1} \Sigma\rt)\bP^\top \\
&= \frac{1}{\sigma^2}\bP(\sigma^2\Id - \Sigma)_+ \bP^\top
\end{align*}
Thus, given that $(\Sigma - \sigma^2\Id)_+(\sigma^2\Id - \Sigma)_+ = 0$, it holds, for any $\bH \in \Scal^d_{+}$: 
\begin{align*}
\langle \bH - \bA_0, \nabla_\bA \phi_\bB(\bA_0)\rangle &= \langle \bP^\top\bH\bP - (\Sigma - \sigma^2\Id)_+, (\sigma^2\Id - \Sigma)_+\rangle \\
&= \langle \bP^\top\bH\bP, (\sigma^2\Id - \Sigma)_+\rangle \\
&= \tr(\bP^\top\bH\bP (\sigma^2\Id - \Sigma)_+) \geq 0
\end{align*}
Where the last inequality holds since both matrices are positive semi-definite. Given that $\phi_\bB$ is convex, the first order optimality condition holds so $\phi_\bB$ is minimized at $\bA_0$. 
\end{proof}

\paragraph{Proof of \Cref{thm:barycenters}}
\begin{proof}
This theorem is a generalization of \cite[Thm 3]{janati20} for multivariate Gaussians. First we are going to break it down using the centering lemma \ref{lemma:centering}. For any probability measure $\mu$, let $\bar{\mu}$ denote its centered transformation. The debiased barycenter problem is equivalent to:
\begin{align}
\label{eq:centering-bar}
\begin{split}
&\min_{\beta \in \Gcal} \sum_{k=1}^K w_k S_\sigma(\alpha_k, \beta) \\
&= \min_{\beta \in \Gcal} \sum_{k=1}^K w_k \OT_\sigma(\alpha_k, \beta) - \frac{1}{2}(\OT_\sigma(\alpha_k, \alpha_k) + \OT_\sigma(\beta, \beta))\\
&=
\min_{\beta \in \Gcal} \sum_{k=1}^K w_k \|\ba_k - \Esp_\beta(X)\|^2 + w_k\OT_\sigma(\bar{\alpha_k}, \bar{\beta}) - \frac{1}{2}(w_k\OT_\sigma(\bar{\alpha_k}, \bar{\alpha_k}) + \OT_\sigma(\bar{\beta}, \bar{\beta}))\\
&= \min_{\substack{\bb \in \RR^d \\ \beta \in \Gcal, \Esp_\beta(\bX) = 0}} \sum_{k=1}^K w_k \|\ba_k - \bb\|^2 + w_k\OT_\sigma(\bar{\alpha_k}, \beta) - \frac{1}{2}(w_k\OT_\sigma(\bar{\alpha_k}, \bar{\alpha_k}) + \OT_\sigma(\beta, \beta))
\end{split}
\end{align}
Therefore, since both arguments are independent, we can first minimize over $\bb$ to obtain $\Esp_\beta(\bX) = \bb = \sum_{k=1}^K w_k\ba_k$. Without loss of generality, we assume from now on that $\ba_k = 0$ for all $k$.

The rest of this proof is adapted from \citep{janati20}, Thm 3 to $d \geq 1$. \citet{janati20} showed that $S_\sigma$ is differentiable and convex (w.r.t. one measure at a time) on sub-Gaussian measures where the notion of differentiability is different from the usual Fréchet differentiability: a function $F: \Gcal \to \RR$ is differentiable at $\alpha$ if there exists $\nabla F(\alpha) \in \Ccal(\RR^d)$ such that for any displacement $t\delta\alpha$ with $t>0$ and $\delta\alpha= \alpha_1 - \alpha_2$ with $\alpha_1, \alpha_2 \in \Gcal$, and
\begin{equation}
\label{eq:differentiability}
F(\alpha + t \delta\alpha)  = F(\alpha) + t\langle \delta\alpha, \nabla F(\alpha)\rangle + o(t)\enspace,
\end{equation}
where $\langle \delta \alpha,\nabla F(\alpha)\rangle = \int_{\RR^d} \nabla F(\alpha) \dd \delta \alpha$.

Moreover, $F$ is convex if and only if for any $\alpha, \alpha' \in \Gcal$:
\begin{equation}
\label{eq:convexity}
F(\alpha)  \geq F(\alpha') + \langle \alpha - \alpha', \nabla F(\alpha')\rangle \enspace,
\end{equation}
Let $(f_k, g_k)$ denote the potentials associated with $\OT_\sigma(\alpha_k, \beta)$ and $h_{\beta}$ the autocorrelation potential associated with $\OT_\sigma(\beta, \beta)$. If $\beta$ is sub-Gaussian, it holds: $\nabla_\beta S_\sigma(\alpha_k, \beta) = g_k - h$. Therefore, from \eqref{eq:convexity} a probability measure $\beta$ is the debiased barycenter if and only if for any direction $\mu \in \Gcal$, the optimality condition holds:
\begin{align}
    \label{eq:optim-condition-bar}
    \begin{split}
    \langle \sum_{k=1}^K w_k \nabla_{\beta} S_\sigma(\alpha_k, \beta), \mu - \beta\rangle \geq 0 \\ 
    \Leftrightarrow \sum_{k=1}^K w_k\langle g_k - h_\beta, \mu - \beta\rangle\geq 0\enspace
    \end{split}
\end{align}
Moreover, the potentials $(f_k), (g_k)$ and $h$ must verify the Sinkhorn optimality conditions \eqref{eq:optimality-potentials} for all $k$ and for all x $\beta$-a.s and y $\alpha$-a.s:
\begin{equation}
	\left\{
	\begin{array}{ll}
	\label{seq:kkt-sdiv}
    e^{\tfrac{f_k(x)}{2\sigma^2}} \left(\int_{\RR^d} e^{\tfrac{-\|x-y\|^2 + g_k(y)}{2\sigma^2}}  \dd\beta(y)\right) = 1, \quad 
    e^{\tfrac{g_k(x)}{2\sigma^2}} \left(\int_{\RR^d} e^{\tfrac{-\|x-y\|^2 + f_k(y)}{2\sigma^2}}  \dd\alpha_k(y)\right) = 1. \\     e^{\tfrac{h(x)}{2\sigma^2}} \left(\int_{\RR^d} e^{\tfrac{-\|x-y\|^2 + h_\beta(y)}{2\sigma^2}}  \dd\beta(y)\right) = 1.
    \end{array}
	\right.
\end{equation}
We are going to show that for the Gaussian measure $\beta$ given in the statement of the theorem is well-defined and verifies all optimality conditions \eqref{seq:kkt-sdiv}. Indeed, assume that $\beta$ is a Gaussian measure given by $\Ncal(\bB)$ for some unknown $\bB \in S^d_+$ (remember that $\beta$ is necessarily centered, following the developments \eqref{eq:centering-bar}). The Sinkhorn equations can therefore be written as a system on positive definite matrices:
\begin{align*}
    \bF_k = \sigma^2\bA_k^{-1} + \bG_k^{-1}, \enspace \bG_k = \sigma^2\bB + \bF_k^{-1}, \enspace \bH = \sigma^2 \bB + \bH^{-1}
\end{align*}
where for all $k$:
\begin{align}
\label{eq:system-bar}
    \begin{split}
\frac{f_k}{2\sigma^2} = \Qcal(\frac{1}{\sigma^2}(\bG_k^{-1} - \Id)) + f_k(0) \\
\frac{g_k}{2\sigma^2} = \Qcal(\frac{1}{\sigma^2}(\bF_k^{-1} - \Id)) + g_k(0)\\
\frac{h}{2\sigma^2} = \Qcal(\frac{1}{\sigma^2}(\bH^{-1} - \Id)) + h_\beta(0)\\
\end{split}
\end{align}
Moreover, provided $\bB$ exists and is positive definite, the system \eqref{eq:system-bar} has a unique set of solutions $(\bF_k)_k, (\bG_k)_k, \bH$ given by:
\begin{align}
\label{eq:system-bar-closed-form}
    \begin{split}
\bF_k = \bB\bC_k^{-1}, \enspace \bG_k = \bC_k^{-1}\bA_k, \enspace \bH = \bB^{-1}\bJ
\end{split}
\end{align}
where $\bC_k = (\bA_k\bB + \frac{\sigma^4}{4}\Id)\rt - \frac{\sigma^2}{2}\Id$ and $\bJ = (\bB^2 + \frac{\sigma^4}{4}\Id)\rt + \frac{\sigma^2}{2}\Id$.
Therefore, the gradient in \eqref{eq:optim-condition-bar} can be written:
\begin{align}
\label{eq:opt-bar-matrices}
\begin{split}
&\sum_{k=1}^K w_k\langle g_k - h_\beta = \Qcal(\frac{1}{\sigma^2}(\sum_{k=1}^K w_k \bF_k^{-1} - \bH^{-1})) + \sum_{w=1}^K w_kg_k(0) - h_\beta(0) \\
&\propto \Qcal(\sum_{k=1}^K w_k \bC_k\bB^{-1} - \bJ^{-1}\bB) + \sum_{w=1}^K w_kg_k(0) - h_\beta(0)
\end{split}
 \end{align}
 and
 \begin{align}
\label{eq:opt-bar-matrices-grad}
\begin{split}
&\sum_{k=1}^K w_k \bC_k\bB^{-1} - \bJ^{-1}\bB \\
&= \sum_{k=1}^K w_k \bB\mrt(\bB\rt\bA_k\bB\rt + \frac{\sigma^4}{4}\Id)\rt \bB\mrt - \bB^{-1}(\bB^2 + \frac{\sigma^4}{4}\Id)\rt \\
&= \sum_{k=1}^K w_k \bB\mrt(\bB\rt\bA_k\bB\rt + \frac{\sigma^4}{4}\Id)\rt \bB\mrt - \bB\mrt(\bB^2 + \frac{\sigma^4}{4}\Id)\rt\bB\mrt \\
&= \bB\mrt\left(\sum_{k=1}^K w_k (\bB\rt\bA_k\bB\rt + \frac{\sigma^4}{4}\Id)\rt  - (\bB^2 + \frac{\sigma^4}{4}\Id)\rt \right)\bB\mrt
\end{split}
 \end{align}
 which is null if $\bB$ is a solution of the equation:
 \begin{equation}
     \label{eq:bar-fixed}
     \sum_{k=1}^K w_k (\bB\rt\bA_k\bB\rt + \frac{\sigma^4}{4}\Id)\rt  = (\bB^2 + \frac{\sigma^4}{4}\Id)\rt.
 \end{equation}
Therefore, for any probability measure $\mu \in \Gcal$:
\begin{align}
        \begin{split}
    \langle \sum_{k=1}^K w_k \nabla_{\beta} S_\sigma(\alpha_k, \beta), \mu - \beta\rangle 
    &= \langle \sum_{k=1}^K w_k g_k - h_\beta, \mu - \beta\rangle \\
    &= \langle \sum_{k=1}^K w_k g_k(0) - h_\beta, \mu - \beta\rangle \\
    &= \langle \sum_{w=1}^K w_kg_k(0) - h_\beta(0), \mu - \beta\rangle \\
    &= \left(\sum_{w=1}^K w_kg_k(0) - h_\beta(0)\right) \int (\dd\mu - \dd\beta) \\
    &= 0
    \end{split} 
\end{align}
since both measures integrate to 1. Therefore, the optimality condition holds.

To end the proof, all we need to show is that \eqref{eq:bar-fixed} admits a  positive definite solution. To show the existence of a solution, the same proof of \citet{agueh11} applies. Indeed, let $\lambda_k$ and $\Lambda_k$ denote respectively the smallest and largest eigenvalue of $\bA_k$. Let $\lambda = \min_k \lambda_k$ and $\Lambda = \max_k \Lambda_k$.
Let $K_{\lambda, \Lambda}$ be the convex compact subset of positive definite matrices $\bB$ such that $\Lambda\Id \succeq \bB \succeq \lambda\Id $. Define the map:
\begin{align*}
    T:& K_{\lambda, \Lambda} \to \Scal^d_{++} \\
      & \bB \mapsto \left(\left(\sum_{k=1}^K w_k (\bB\rt\bA_k\bB\rt + \frac{\sigma^4}{4}\Id)\rt\right)^2 - \frac{\sigma^4}{4}\Id\right)\rt
\end{align*}
Now for any $\bB \in K_{\lambda, \Lambda}$, it holds:
\begin{align}
 \lambda \Id \preceq T(\bB) \preceq  \Lambda \Id.
\end{align}
$T$ is therefore a continuous function that maps $K_{\lambda, \Lambda}$ to itself, thus Brouwer’s fixed-point theorem guarantees the existence of a solution.
\end{proof}
\paragraph{Proof of \Cref{prop:unbalanced-loss-opt}}

\begin{proof}
Using Fubini-Tonelli along with the optimality conditions \eqref{eq:opt-conditions-unbalanced}, the double integral can be written:
\begin{align*}
   \pi(\RR^d \times \RR^d) &= \int_{\RR^d\times\RR^d}e^{\frac{-\|x-y\|^2 + f(x) + g(y)}{2\sigma^2}}\dd\alpha(x)\dd\beta(y) \\ &= 
        \int_{\RR^d} \left( \int_{\RR^d} e^{\frac{-\|x-y\|^2 + f(x)}{2\sigma^2}}\dd\alpha(x)\right)e^{\frac{g(y)}{2\sigma^2}}\dd\beta(y) \\
        &=  \int_{\RR^d} e^{\frac{g(y)}{2\sigma^2}(1 - \frac{1}{\tau})}\dd\beta(y) \\
        &=  \int_{\RR^d} e^{-\frac{g(y)}{\gamma}}\dd\beta(y)
\end{align*}
And similarly:  $\pi(\RR^d \times \RR^d) =
\int_{\RR^d} e^{-\frac{f(x)}{\gamma}}\dd\alpha(x)$. 
Therefore, the three integrals in the dual objective \eqref{eq:unbalanced_ot_dual} are equal to $\pi(\RR^d \times \RR^d)$ which ends the proof.
\end{proof}


\begin{lemma}
\label{lem:sum-forms}[Sum of factorized quadratic forms]
Let $\bA, \bB \in S_d$ such that $\bA \neq \bB$ and $\ba, \bb \in \RR^d$. Denote $\alpha = (\bA, \ba)$ and $\beta = (\bB, \bb)$.
Let $P_\alpha(\bx) = -\frac{1}{2}(\bx-\ba)^\top \bA (\bx - \ba)$ and $P_\beta(\bx) = -\frac{1}{2}(\bx-\bb)^\top \bB (\bx - \bb)$. Then:
\begin{equation}
    \label{eq:sum-forms}
P_\alpha(x) + P_\beta(x) = -\frac{1}{2}\left((\bx - \bc)^\top\bC(\bx-\bc) + q_{\alpha, \beta}\right)
\end{equation}
where:
\begin{align}
    \label{eq:sum-forms-coefs}
    \left\{
    \begin{array}{ll}
    \bC &= \bA + \bB \\
    (\bA + \bB)\bc &= (\bA\ba + \bB\bb) \\
    q_{\alpha, \beta} &= \ba^\top \bA \ba + \bb^\top \bB \bb - c^\top \bC \bc
    \end{array}
    \right.
\end{align}
In particular, if $\bC = \bA + \bB$ is invertible, then:
\begin{align}
    \left\{
    \begin{array}{ll}
    &\bc = \bC^{-1}(\bA\ba + \bB\bb) \\
    &\bc^\top\bC\bc = (\bA\ba + \bB\bb)^\top \bC^{-1}(\bA\ba + \bB\bb)
    \end{array}
    \right.
\end{align}
\end{lemma}
\begin{proof}
On one hand,
\begin{align*}
    P_\alpha(x) + P_\beta(x) &=  -\frac{1}{2}\left((\bx - \ba)^\top\bA(\bx-\ba) + (\bx - \bb)^\top\bB(\bx-\bb)\right) \\
    &= -\frac{1}{2}\left(\bx^\top(\bA + \bB)\bx - 2\bx^\top(\bA\ba +\bB\bb) + \ba^\top\bA\ba + \bb^\top\bB\bb\right) \\
\end{align*}
On the other hand, for an arbitrary $\gamma=(\bc, \bC)$ and $q\in\RR$:
\begin{align*}
    P_\gamma(x) - \frac{q}{2} &= -\frac{1}{2}\left((\bx - \bc)^\top\bC(\bx-\bC) + q\right)  \\
    &= -\frac{1}{2}\left(x^\top\bC x - 2x^\top \bC\bc + \bc^\top\bC\bc + q\right) 
\end{align*}
If $\bA \neq \bB$, identification of the parameters of both quadratic forms leads to \eqref{eq:sum-forms-coefs}. 
\end{proof}
\begin{lemma}
\label{lem:convolution}[Gaussian convolution of factorized quadratic forms]
Let $\bA \in S_d$ and $\ba \in \RR^d$ and $\sigma > 0$ such that $\sigma^2 \bA + \Id \succ 0$.
Let $Q_\alpha(\bx) = -\frac{1}{2}(\bx-\ba)^\top \bA (\bx - \ba)$. Then the convolution of $e^{\Qcal_\alpha}$ by the Gaussian kernel $\Ncal(0, \frac{\Id}{\sigma^2})$  is given by:
\begin{equation}
    \label{eq:convolution}
     \Ncal(0, \frac{\Id}{\sigma^2}) \star \exp\left(\Qcal_\alpha\right) \defeq \int_{\RR^d}  \frac{1}{(2\pi\sigma^2)^{\frac{n}{2}}}\exp\left(-\frac{1}{2\sigma^2}\|.-y\|^2 + \Qcal_\alpha(y)\right)\dd y = c_{\alpha} \exp(\Qcal(\ba, \bJ))
\end{equation}
where:
\begin{align*}
    \bJ &= (\sigma^2\bA + \Id)^{-1}\bA\\
    c_\alpha &= \frac{1}{\sqrt{\det(\sigma^2 \bA + \Id)}}
\end{align*}
\end{lemma}
\begin{proof}
Using \Cref{lem:sum-forms} one can write for any $x \in \RR^d$ considered fixed:
\begin{align*}
    -\frac{1}{2\sigma^2}\|x-y\|^2 + \Qcal_\alpha(y) &= \Qcal(x, \frac{\Id}{\sigma^2})(y) + \Qcal(\ba, \bA)(y) \\
    &= \Qcal(\bA\ba + \frac{x}{\sigma^2}, \bA + \frac{\Id}{\sigma^2})(y) + h(x)
\end{align*}
with $h(x)= -\frac{1}{2}\left(\ba^\top \bA \ba + \frac{1}{\sigma^2}\|x\|^2 - \frac{1}{\sigma^2}(\sigma^2 \bA\ba + x)^\top(\sigma^2\bA + \Id)^{-1}(\sigma^2\bA\ba + x)\right)$.
Therefore, the convolution integral is finite if and only if $\bA + \frac{\Id}{\sigma^2} \succ 0$ in which case we get the integral of a Gaussian density:
\begin{align*}
   \frac{1}{(2\pi\sigma^2)^{\frac{n}{2}}} \int_{\RR^d} \exp\left(\Qcal(\bA\ba + \frac{x}{\sigma^2}, \bA + \frac{\Id}{\sigma^2})(y) + h(x)\right)\dd (y)  &= \sqrt{\frac{\det(2\pi (\bA + \frac{\Id}{\sigma^2})^{-1})}{(2\pi\sigma^2)^{n}}} e^{h(x)}\\
   &= \frac{e^{h(x)}}{\sqrt{\det(\sigma^2\bA + \Id)}}
\end{align*}
For the sake of clarity, let's separate the terms of $h$ depending on their order in $x$: $h(x) = -\frac{1}{2}\left(h_2(x) + h_1(x) + h_0\right)$ where:
\begin{align*}
    h_2(x) &= \frac{1}{\sigma^2}(\|x\|^2 - x^\top (\sigma^2\bA + \Id)^{-1} x \\
    h_1(x) &= -2x^\top (\sigma^2\bA + \Id)^{-1}\bA\ba \\
    h_0 &= \ba\bA\ba - \sigma^2\ba^\top\bA(\sigma^2\bA + \Id)^{-1}\bA\ba
\end{align*}
Finally, we can factorize $h_2$ and $h_0$ using Woodbury's matrix identity which holds even for a singular matrix $\bA$:
\begin{equation}
    \label{eq:woodbury}
    \tag{Woodbury's identity}
   (\sigma^2\bA + \Id)^{-1}= \Id - \sigma^2(\sigma^2\bA + \Id)^{-1}\bA
\end{equation}
Let $\bJ = (\sigma^2\bA + \Id)^{-1}\bA$.
\begin{align*}
    h_2(x) &= \frac{1}{\sigma^2}(\|x\|^2 - x^\top (\Id - \sigma^2(\sigma^2\bA + \Id)^{-1}\bA)x \\
    &= x^\top (\sigma^2\bA + \Id)^{-1}\bA x\\
    &= x^\top \bJ x\\
    h_1(x) &= -2x^\top \bJ \ba \\
    h_0 &= \ba\bA\ba - \sigma^2\ba^\top\bA(\sigma^2\bA + \Id)^{-1}\bA\ba\\
    &=\ba^\top \bA(\Id - \sigma^2(\sigma^2\bA + \Id)^{-1}\bA) \ba\\
    &= \ba^\top \bA(\sigma^2\bA + \Id)^{-1}\ba \\
    &= \ba^\top (\sigma^2\bA + \Id)^{-1}\bA\ba\\
    &=\ba^\top\bJ\ba
\end{align*}
Therefore, $h(x) = -\frac{1}{2}\left(x^\top\bJ x -2x^\top\bJ\ba + \ba^\top\bJ\ba \right) = -\frac{1}{2}(x-\ba)^\top\bJ(x-\ba) = \Qcal(\ba, \bJ)(x)$.
\end{proof}
\begin{lemma}
\label{lem:convolution-2}[Gaussian convolution of generic quadratic forms]
Let $\bA \in S_d$ and $\ba \in \RR^d$ and $\sigma > 0$ such that $\sigma^2 \bA + \Id \succ 0$.
Let $Q_\alpha(\bx) = -\frac{1}{2}(\bx^\top \bA \bx - 2\bx^\top\ba)$. Then the convolution of $e^{\Qcal_\alpha}$ by the Gaussian kernel $\Ncal(0, \frac{\Id}{\sigma^2})$  is given by:
\begin{equation}
    \label{eq:convolution-2}
     \Ncal(0, \frac{\Id}{\sigma^2}) \star \exp\left(\Qcal_\alpha\right) \defeq \int_{\RR^d}  \frac{1}{(2\pi\sigma^2)^{\frac{n}{2}}}\exp\left(-\frac{1}{2\sigma^2}\|.-y\|^2 + \Qcal_\alpha(y)\right)\dd y = c_{\alpha} \exp(\Qcal(\bG\ba, \bG\bA))
\end{equation}
where:
\begin{align*}
    \bG &= (\sigma^2\bA + \Id)^{-1}\\
    c_\alpha &= \frac{e^{\frac{\sigma^2\ba^\top\bG\ba}{2}}}{\sqrt{\det(\sigma^2 \bA + \Id)}}
\end{align*}
\end{lemma}
\begin{proof}
Using \Cref{lem:sum-forms} one can write for any $x \in \RR^d$ considered fixed:
\begin{align*}
    -\frac{1}{2\sigma^2}\|x-y\|^2 + \Qcal_\alpha(y) &= \Qcal(x, \frac{\Id}{\sigma^2})(y) + \Qcal(\ba, \bA)(y) \\
    &= \Qcal(\ba + \frac{x}{\sigma^2}, \bA + \frac{\Id}{\sigma^2})(y) - \frac{1}{2\sigma^2}\|x\|^2 \\
    &= \Qcal f((\sigma\ba + \frac{x}{\sigma^2}, \bA + \frac{\Id}{\sigma^2})(y) + h(x)
\end{align*}
with $h(x)= -\frac{1}{2}\left(\frac{1}{\sigma^2}\|x\|^2 - \frac{1}{\sigma^2}(\sigma^2 \ba + x)^\top(\sigma^2\bA + \Id)^{-1}(\sigma^2\ba + x)\right)$.
Therefore, the convolution integral is finite if and only if $\bA + \frac{\Id}{\sigma^2} \succ 0$ in which case we get the integral of a Gaussian density:
\begin{align*}
   \frac{1}{(2\pi\sigma^2)^{\frac{n}{2}}} \int_{\RR^d} \exp\left(\Qcal f(\ba + \frac{x}{\sigma^2}, \bA + \frac{\Id}{\sigma^2})(y) + h(x)\right)\dd (y)  &= \sqrt{\frac{\det(2\pi (\bA + \frac{\Id}{\sigma^2})^{-1})}{(2\pi\sigma^2)^{n}}} e^{h(x)}\\
   &= \frac{e^{h(x)}}{\sqrt{\det(\sigma^2\bA + \Id)}}
\end{align*}
For the sake of clarity, let's separate the terms of $h$ depending on their order in $x$: $h(x) = -\frac{1}{2}\left(h_2(x) + h_1(x) + h_0\right)$ where:
\begin{align*}
    h_2(x) &= \frac{1}{\sigma^2}(\|x\|^2 - x^\top (\sigma^2\bA + \Id)^{-1} x \\
    h_1(x) &= -2x^\top (\sigma^2\bA + \Id)^{-1}\ba \\
    h_0 &= - \sigma^2\ba^\top(\sigma^2\bA + \Id)^{-1}\ba
\end{align*}
Finally, we can factorize $h_2$ and $h_0$ using Woodbury's matrix identity which holds even for a singular matrix $\bA$:
\begin{equation}
    \label{eq:woodbury}
    \tag{Woodbury's identity}
   (\sigma^2\bA + \Id)^{-1}= \Id - \sigma^2(\sigma^2\bA + \Id)^{-1}\bA
\end{equation}
Let $\bG = (\sigma^2\bA + \Id)^{-1}$.
\begin{align*}
    h_2(x) &= \frac{1}{\sigma^2}(\|x\|^2 - x^\top (\Id - \sigma^2(\sigma^2\bA + \Id)^{-1}\bA)x \\
    &= x^\top (\sigma^2\bA + \Id)^{-1}\bA x\\
    &= x^\top \bG\bA x\\
    h_1(x) &= -2x^\top \bG \ba \\
    h_0 &=  - \sigma^2\ba^\top(\sigma^2\bA + \Id)^{-1}\ba\\
    &=-\sigma^2\ba^\top \bG \ba
\end{align*}
Therefore, $h(x) = -\frac{1}{2}\left(x^\top\bG\bA x -2x^\top\bG\ba  -\sigma^2\ba^\top \bG \ba \right) = \Qcal(\bG\ba, \bG\bA)(x) + \frac{\sigma^2\ba^\top\bG\ba}{2}$.
\end{proof}


\subsection{Proof of theorem \ref{thm:unbalanced}}
In the balanced case, we showed that Sinkhorn's transform is stable for quadratic potentials and that the resulting sequence is a contraction. Similarly, the following proposition shows that the unbalanced Sinkhorn transform is stable for quadratic potentials.
M
\begin{proposition}
\label{prop:sinkhorn-transform-unbalanced}
Let $\alpha$ be an unbalanced Gaussians given by$ m_\alpha \Ncal(\ba, \bA)$. Let $\tau = \frac{\gamma}{2\sigma^2 + \gamma}$. Define the unbalanced Sinkhorn transform $T: \RR^{\RR^d} \to \RR^{\RR^d}$: 
\begin{equation}
    T_\alpha(h)(x) \defeq - \tau\log\int_{\RR^d} e^{\tfrac{-\|x- y\|^2}{2\sigma^2} + h(y)} \dd\alpha(y)
\end{equation}
Let $\bU \in \Scal_d$, $\bu \in \RR^d$ and $m_u > 0$.
If $h = \log(m_u) + \Qcal(\bu, \bU)$ i.e  $h(x) = \log(m_u) - \frac{1}{2} (x^\top \bU x -  2x^\top\bu)$, then $T_\alpha(h)$ is well defined if and only if  $\bF \defeq \sigma^2\bU + \sigma^2\bA^{-1} + \Id \succ 0$, in which case  $T_\alpha(h)= \Qcal(\bv, \bV) + \log(m_v)$ with the identified parameters:
\begin{align}
\bV &= \tau\frac{1}{\sigma^2}(\bF^{-1} - \Id)\\
\bv &= -\tau \bF^{-1}(\bA^{-1}\ba + \bu)\\
m_v &= \left(\frac{\sqrt{\det(\bA)\det(\bF)}}{m_u m_\alpha e^{\frac{q_{u, \alpha}}{2}}\sigma^{2d}}\right)^\tau
\end{align}
where $q_{u, \alpha} = \frac{\sigma^2}{\tau^2}\bv^\top\bF\bv - \ba^\top\bA^{-1}\ba$.
\end{proposition}
\begin{proof}
The exponent inside the integral can be written as:
\begin{align*}
    e^{\tfrac{-\|x- y\|^2}{2\sigma^2} + h(y)} \dd\alpha(y) &\propto  e^{\tfrac{-\|x- y\|^2}{2\sigma^2} - \frac{1}{2}(y^\top\bX y - y^\top \bA^{-1}y)} \dd y \\
    &\propto e^{-\tfrac{1}{2}(y^\top(\frac{\Id}{\sigma^2} + \bX +\bA^{-1})y) + \frac{x^\top y}{\sigma^2}} \dd y
\end{align*}
which is integrable if and only if $\bU + \bA^{-1} + \frac{1}{\sigma^2}\Id 
\succ 0 \Leftrightarrow \bF \succ 0$.
Moreover, up to a multiplicative factor, the exponentiated Sinkhorn transform is equivalent to a Gaussian convolution of an exponentiated quadratic form. Lemma \ref{lem:convolution-2} applies:  
\begin{align*}
    e^{-T_\alpha(h)} &= \int_{\RR^d} e^{\tfrac{-\|x- y\|^2}{2\sigma^2} + f(y)}\dd \alpha(y) \\
&= m_u m_\alpha \frac{\exp(-\frac{1}{2} \ba^\top\bA^{-1}\ba)}{\sqrt{\det(2\pi\bA)}}   \int_{\RR^d} e^{\tfrac{-\|x- y\|^2}{2\sigma^2} + \Qcal(\bu, \bU)(y) + \Qcal(\bA^{-1}\ba, \bA^{-1})(y)} \dd y \\
    &= m_u m_\alpha \frac{\exp(-\frac{1}{2} \ba^\top\bA^{-1}\ba)}{\sqrt{\det(2\pi\bA)}} \sqrt{(2\pi \sigma^2)^{2d}} \exp\left(\Ncal(\sigma^2\Id)\right) \star \exp\left(\Qcal(\bu + \bA^{-1}\ba, \bU + \bA^{-1})\right)\\
     &=m_u m_\alpha \frac{\sigma^{2d}\exp(-\frac{1}{2} \ba^\top\bA^{-1}\ba)}{\sqrt{\det(\bA)}} \exp\left(\Ncal(\sigma^2\Id)\right) \star \exp\left(\Qcal(\bu + \bA^{-1}\ba, \bU + \bA^{-1})\right)\\
     &= m_u m_\alpha \frac{\sigma^{2d}\exp(-\frac{1}{2} \ba^\top\bA^{-1}\ba) }{\sqrt{\det(\bA)}} c_\alpha\exp\left(\Qcal(\bF^{-1}(\bu + \bA^{1}\ba), \bF^{-1}(\bU + \bA^{-1})\right).\\
     &=m_u m_\alpha \frac{\sigma^{2d}\exp(-\frac{1}{2} \ba^\top\bA^{-1}\ba) }{\sqrt{\det(\bA)}} c_\alpha\exp\left(\Qcal(\bF^{-1}(\bu + \bA^{1}\ba),\frac{1}{\sigma^2}\bF^{-1}(\bF - \Id)\right).\\
&=m_u m_\alpha \frac{\sigma^{2d}\exp(-\frac{1}{2} \ba^\top\bA^{-1}\ba) }{\sqrt{\det(\bA)}} c_\alpha\exp\left(\Qcal(\bF^{-1}(\bu + \bA^{1}\ba),\frac{1}{\sigma^2}(\Id - \bF^{-1})\right).
\end{align*}
where $c_\alpha = \frac{\exp(\frac{1}{2}\sigma^2 (\bu + \bA^{-1}\ba)^\top \bF^{-1}(\bu + \bA^{-1}\ba))}{\sqrt{\det(\bF)}}$.

Therefore, by applying $-\tau \log$ we can identify $\bV$ and $\bv$. Substituting $\bu + \bA^{-1}\ba$ by $-\frac{1}{\tau}\bF\bv$ leads to the equation of $m_v$.
\end{proof}
Unlike the balanced case, the unbalanced Sinkhorn iterations require 2 more parameters ($\bv$ and $m_v$) with tangled updates. Proving the convergence of the resulting algorithm is more challenging. Instead, we directly solve the optimality conditions and show that a pair of quadratic potentials verifies \eqref{eq:opt-conditions-unbalanced}.

\begin{proposition}
\label{prop:optim-unbalanced-params}
The pair of quadratic forms $(f, g)$ of \eqref{eq:potentials-as-quads} verifies the optimality conditions \eqref{eq:opt-conditions-unbalanced} if and only if:
\begin{align}
    \label{eq:necessary-conditions-unbalanced}
    \begin{split}
   \bF\defeq\sigma^2 \bA^{-1} +\sigma^2 \bU + \Id \succ 0 \\
      \bG\defeq\sigma^2 \bB^{-1} +\sigma^2 \bV +  \Id \succ 0,
      \end{split}
\end{align}
\begin{minipage}{0.49\linewidth}
\begin{align*}
    \begin{split}
        &m_v \left(\frac{m_u m_\alpha e^{\frac{q_{u, \alpha}}{2}}\sigma^d}{\sqrt{\det(\bA)\det(\bF)}}\right)^\tau = 1 \\
    &\bv = - \tau\bF^{-1}(\bA^{-1}\ba + \bu) \\
    \bG &= \tau \bF^{-1} + \sigma^2 \bB^{-1} + (1-\tau)\Id\\
    q_{u, \alpha} &= \frac{\sigma^2}{\tau^2}\bv^\top\bF\bv - \ba^\top\bA^{-1}\ba
    \end{split} 
\end{align*}
\end{minipage}
\begin{minipage}{0.49\linewidth}
\begin{align}
    \label{eq:optim-unbalanced-params}
    \begin{split}
        &m_u \left(\frac{m_v m_\beta e^{\frac{q_{v, \beta}}{2}}\sigma^d}{\sqrt{\det(\bB)\det(\bG)}}\right)^\tau = 1 \\
    &\bu = - \tau\bG^{-1}(\bB^{-1}\bb + \bv) \\
    \bF &= \tau \bG^{-1} + \sigma^2 \bA^{-1} + (1-\tau)\Id \\
    q_{v, \beta} &= \frac{\sigma^2}{\tau^2}\bu^\top\bG\bu -\bb^\top\bB^{-1}\bb
    \end{split} 
\end{align}
\end{minipage}
\end{proposition}
\begin{proof}
The equations on $m_u, m_v, \bu, \bv$ follow immediately from \Cref{prop:sinkhorn-transform-unbalanced}. Using the definition of $\bF$ and $\bG$, substituting $\bU$ and $\bF$ leads to the equations in $\bF$ and $\bG$
\end{proof}

We now turn to solve the system \eqref{eq:optim-unbalanced-params}. Notice that in general, the dual potentials can only be identified up to a an additive constant. Indeed, if a pair $(f, g)$ is optimal, then $(f + K, g - K)$ is also optimal for any $K\in\RR$ (the transportation plan does not change). Thus, at optimality, it is sufficient to obtain the product $m_u m_v$. We start by identifying $(\bF, \bG)$ then $(\bu, \bv)$ and finally $m_u m_v$.
%
%
\paragraph{Identifying $\bF$ and $\bG$.}
The equations in $\bF$ and $\bG$ can shown to be equivalent to those of the balanced case up to some change of variables. Let $\lambda = \frac{1-\tau}{\sigma^2}$
\begin{align*}
    &\left\{\begin{array}{ll}
    \bF &= \tau \bG^{-1} + \sigma^2 \bA^{-1} + (1-\tau)\Id \\
    \bG &= \tau \bF^{-1} + \sigma^2 \bB^{-1} + (1-\tau)\Id
    \end{array}\right. \\
    & \Leftrightarrow   \left\{\begin{array}{ll}
    \bF &= \left(\frac{\bG}{\tau}\right)^{-1} + \frac{\sigma^2}{\tau} \tau(\bA^{-1} +\frac{1}{\lambda}\Id) \\
    \frac{\bG}{\tau} &= \bF^{-1} + \frac{\sigma^2}{\tau} (\bB^{-1} +  \frac{1}{\lambda}\Id)
    \end{array}\right.\\
    & \Leftrightarrow   \left\{\begin{array}{ll}
    \bF &= \widetilde{\bG}^{-1} + \sigma^2 (\frac{\widetilde{\bA}}{\tau})^{-1} \\
    \widetilde{\bG} &= \bF^{-1} + \sigma^2\widetilde{\bB}^{-1}
    \end{array}\right.
\end{align*}
which correspond to the balanced OT fixed point equations \eqref{eq:optim-fg} associated with the pair $(\frac{\widetilde{\bA}}{\tau}, \widetilde{\bB})$ with the change of variables:
\begin{align}
    \widetilde{\bG} &\defeq \frac{\bG}{\tau} \\
    \widetilde{\bA} &\defeq \tau(\bA^{-1} + \frac{1}{\lambda}\Id)^{-1} \\
    \widetilde{\bB} &\defeq \tau(\bB^{-1} +  \frac{1}{\lambda}\Id)^{-1}
\end{align}
Notice that since $0 < \tau < 1$, $\widetilde{\bA}$ and $\widetilde{\bB}$ are well-defined and positive definite. Therefore, \Cref{prop:closed_form_matrix} applies and we can write in closed form:
\begin{align}
\label{eq:unbalanced-C}
    \begin{split}
    \bC\defeq\widetilde{\bA}\widetilde{\bG}^{-1} &= \left(\frac{1}{\tau}\widetilde{\bA}\widetilde{\bB} + \frac{\sigma^4}{4}\Id\right)^{\frac{1}{2}} - \frac{\sigma^2}{2}\Id \\
    &= \widetilde{\bA}\rt\left(\frac{1}{\tau}\widetilde{\bA}\rt\widetilde{\bB}\widetilde{\bA}\rt + \frac{\sigma^4}{4}\Id\right)^{\frac{1}{2}}\widetilde{\bA}\mrt - \frac{\sigma^2}{2}\Id
    \end{split}
\end{align}
And similarly by symmetry:
\begin{equation}
    \widetilde{\bB}\bF^{-1} = \left(\frac{1}{\tau}\widetilde{\bB}\widetilde{\bA} + \frac{\sigma^4}{4}\Id\right)^{\frac{1}{2}} - \frac{\sigma^2}{2}\Id = \bC^\top
\end{equation}
Therefore we obtain $\bF$ and $\bG$ in closed form:
\begin{align}
    \label{eq:fg-closed-form-unbalanced}
    \bF &= \widetilde{\bB}\bC^{-1} \\
    \bG &= \bC^{-1}\widetilde{\bA}
\end{align}
Finally, to obtain the formulas of $\widetilde{\bA}$ and $\widetilde{\bB}$ of \Cref{thm:unbalanced}, use Woodburry's identity to write:
\begin{align*}
    \widetilde{\bB} &= \tau\lambda(\Id - \lambda (\bB + \lambda\Id)^{-1}) \\
     &= \frac{\gamma}{\gamma + 2\sigma^2}\frac{2\sigma^2 + \gamma}{2} (\Id - \lambda(\bB + \lambda\Id)^{-1})\\
     &= \frac{\gamma}{2}(\Id - \lambda(\bB + \lambda\Id)^{-1})
\end{align*}
the same applies for $ \widetilde{\bA}$.
%
\paragraph{Identifying $\bu$ and $\bv$.}
Combining the equations in $\bu$ and $\bv$ leads to:
\begin{align*}
    &\bv = - \tau\bF^{-1}(\bA^{-1}\ba + \tau\bu) \\
    &\Leftrightarrow \bF\bv= -\tau \bA^{-1}\ba - \tau\bu \\
    &\Leftrightarrow \bF\bv = -\tau\bA^{-1}\ba + \tau^2\bG^{-1}(\bB^{-1}\bb + \bv) \\
    &\Leftrightarrow \bG\bF\bv = -\tau\bG\bA^{-1}\ba + \tau^2(\bB^{-1}\bb + \bv)\\
    &\Leftrightarrow (\bG\bF - \tau^2\Id)\bv = -\tau\bG\bA^{-1}\ba +\tau^2\bB^{-1}\bb
\end{align*}
Similarly, $(\bF\bG - \tau^2\Id)\bu = -\tau\bF\bB^{-1}\bb +\tau^2\bA^{-1}\ba$. Moreover, since $0 <\tau < 1$, it holds$ (\bF - \tau^2\bG^{-1}) \succ (\bF - \tau\bG^{-1}) = \sigma^2 \widetilde{\bA}^{-1} \succ 0$. Therefore, $(\bF\bG - \tau^2\Id) = (\bF - \tau^2\bG^{-1}\Id)\bG$ is invertible. The same applies for $(\bG\bF - \tau^2\Id)$.

Finally, both equations can be vectorized:
\begin{equation}
    \label{eq:vectorized-uv}
    \begin{pmatrix}
    \bG\bF - \tau^2\Id & 0 \\
    0 & \bF\bG - \tau^2\Id
    \end{pmatrix}
    \begin{pmatrix}
    \bv \\ \bu
    \end{pmatrix} = 
    \begin{pmatrix}
    -\tau \bG &  \tau^2\Id \\
    \tau^2\Id & -\tau\bF
    \end{pmatrix}
    \begin{pmatrix}
    \bA^{-1} & 0 \\
    0 & \bB^{-1}
    \end{pmatrix}
    \begin{pmatrix}
    \ba \\
    \bb
    \end{pmatrix}
\end{equation}
\paragraph{Identifying $m_u m_v$.}
Now that $\bF, \bG, \bu$ and $\bv$ are given in closed form, $m_u m_v$ is obtained by taking the product of both equations:
\begin{align}
    \label{eq:mumv}
    (m_u m_v)^{\tau + 1} = \left(\frac{\sqrt{\det(\bA\bB)\det(\bF\bG)}}{\sigma^{2d} m_\alpha m_\beta}\right)^\tau \exp(-\frac{\tau}{2}(q_{u, \alpha} + q_{v, \beta}))
\end{align}

\paragraph{Transportation plan.}
Let $\omega \defeq \frac{m_\alpha m_\beta}{\sqrt{\det(4\pi^2 \bA\bB)}} m_u m_v e^{-\frac{1}{2}(\ba^\top\bA^{-1}\ba + \bb^\top\bB^{-1}\bb)}$.
At optimality, the transport plan $\pi$ is given by:
\begin{align*}
\frac{\dd\pi}{\dd x\dd y}(x, y) &= \exp\left(\frac{f(x) + g(y) - \|x - y\|^2}{2\sigma^2}\right) \frac{\dd\alpha}{\dd x}(x) \frac{\dd\beta}{\dd y}(y)\\
&= \omega \exp\left(\Qcal(\bA^{-1}\ba + \bu, \bA^{-1} + \bU)(x) - \frac{\|x - y\|^2}{2\sigma^2} + \Qcal(\bB^{-1}\bb + \bv, \bB^{-1} + \bV)(y)\right) \\
    &= \omega \exp\left(\Qcal(\bU + \bA^{-1})(x) + \Qcal(\bV + \bB^{-1})(y) + \Qcal(\begin{smallmatrix}\frac{\Id}{\sigma^2} & - \frac{\Id}{\sigma^2} \\ - \frac{\Id}{\sigma^2} & \frac{\Id}{\sigma^2} \end{smallmatrix})(x, y)\right)\\
    &= \omega\exp\left(\Qcal\left(\begin{pmatrix} \bA^{-1}\ba + \bu \\ \bB^{-1}\bb + \bv \end{pmatrix}, \begin{pmatrix}\bU + \bA^{-1} + \frac{\Id}{\sigma^2} & 0 \\ 0 & \bV + \bB^{-1} + \frac{\Id}{\sigma^2} \end{pmatrix}\right)(x, y)\right) \\
    &= \omega\exp\left(\Qcal\left(\begin{pmatrix} \bA^{-1}\ba + \bu \\ \bB^{-1}\bb + \bv \end{pmatrix}, \frac{1}{\sigma^2}\begin{pmatrix}\bF & -\Id \\ -\Id & \bG \end{pmatrix}\right)(x, y)\right) \\
&=\omega\exp\left( \Qcal(\mu, \Gamma)(x, y)\right)
\end{align*}
with $\mu \defeq \begin{pmatrix} \bA^{-1}\ba + \bu \\ \bB^{-1}\bb + \bv \end{pmatrix}$ and $\Gamma \defeq \begin{pmatrix}\frac{\bF}{\sigma^2} & - \frac{\Id}{\sigma^2}  \\ - \frac{\Id}{\sigma^2} &  \frac{\bG}{\sigma^2} \end{pmatrix}$. Let's show that $\Gamma \succ 0$.  
Since $\frac{\bG}{2\sigma^2} \succ 0$ , it is sufficient to show that Schur complement  $\frac{\bF}{\sigma^2} - \frac{1}{\sigma^2}\bG^{-1}  \succ 0$. 
On one hand, with 
\begin{align*}
    \frac{\bF - \bG^{-1}}{\sigma^2} &= \tau \widetilde{\bA}^{-1} - \frac{1}{\lambda} \bG^{-1}
\end{align*}
On the other hand, almost by definition $\widetilde{\bA} \prec \tau \lambda\Id$ and $\widetilde{\bB} \prec \tau\lambda\Id$. Thus for any $x\in\RR^d$:
\begin{equation*}
    x^\top \frac{\widetilde{\bA}\rt\widetilde{\bB}\widetilde{\bA}\rt}{\tau} x \leq  \lambda\|\widetilde{\bA}\rt x\|^2 = \lambda x^\top \widetilde{\bA} x \leq \tau\lambda^2 \|x\|^2,
\end{equation*}
which implies
\begin{equation*}
\left(\frac{\widetilde{\bA}\rt\widetilde{\bB}\widetilde{\bA}\rt}{\tau}  + \frac{\sigma^4}{4}\Id\right)\rt  \prec \sqrt{\tau\lambda^2 + \frac{\sigma^4}{4}}\Id = \frac{\lambda}{2}(\sqrt{4\tau + (1-\tau)^2})\Id =\frac{\lambda(1+\tau)}{2} \Id.
\end{equation*}
Therefore, using the second equality of \eqref{eq:unbalanced-C} and inverting \eqref{eq:fg-closed-form-unbalanced} to obtain $\bG^{-1}$:
\begin{align*}
    x^\top \bG^{-1} x&=  x^\top \widetilde{\bA}\mrt\left(\left(\frac{\widetilde{\bA}\rt\widetilde{\bB}\widetilde{\bA}\rt}{\tau}  + \frac{\sigma^4}{4}\Id\right)^{\frac{1}{2}} - \frac{\sigma^2}{2}\Id)\right)\widetilde{\bA}\mrt x \\ 
    &=  (\widetilde{\bA}\mrt x)^\top \left(\left(\frac{\widetilde{\bA}\rt\widetilde{\bB}\widetilde{\bA}\rt}{\tau}  + \frac{\sigma^4}{4}\Id\right)^{\frac{1}{2}} - \frac{\lambda(1-\tau)}{2}\Id)\right) (\widetilde{\bA}\mrt x) \\
    &\leq (\widetilde{\bA}\mrt x)^\top \left(\frac{\lambda(1+\tau)}{2}\Id - \frac{\lambda(1-\tau)}{2}\Id)\right) (\widetilde{\bA}\mrt x) \\
    &= \tau\lambda x^\top \widetilde{\bA}^{-1}x.
\end{align*}
Thus $\bG^{-1} \prec \tau\lambda\widetilde{\bA}^{-1}$. We can therefore conclude that the Schur complement $\frac{1}{\sigma^2}(\bF - \bG^{-1})$ is positive definite. By completing the square, we can factor $\frac{\dd \pi}{\dd x\dd x}$ as a Gaussian density. Let $z\defeq (\begin{smallmatrix} x \\ y\end{smallmatrix})$:

\begin{align*}
\frac{\dd\pi}{\dd x\dd y}(x, y) &=\omega\exp\left( \Qcal(\mu, \Gamma)(x, y)\right) \\
&= \omega\exp\left(-\frac{1}{2}(z^\top \Gamma z - 2z^\top\mu)\right)\\
&= \omega\exp\left(\frac{1}{2}\mu^\top \Gamma^{-1}\mu- \frac{1}{2}(z - \Gamma^{-1}\mu)^\top\Gamma(z-\Gamma^{-1}\mu))\right)\\
&= \omega e^{\frac{1}{2}\mu^\top \Gamma^{-1}\mu} \Ncal(\bH\mu, \bH)(z),
\end{align*}
where $\bH= \Gamma^{-1}$.
\paragraph{Detailed expressions.}
To conclude the proof of \Cref{thm:unbalanced}, we need to simplify the formulas of $m, \bH\mu$ and $\bH$. First, we will start with the mean $\bH\mu$.

\paragraph{$\bH\mu$}
Using the optimality conditions of \Cref{prop:optim-unbalanced-params} and the closed form formula of $\bv$ and $\bu$:
\begin{align}
\label{eq:finding-mu}
    \begin{split}
    \mu &= \begin{pmatrix} \bA^{-1}\ba + \bu \\
    \bB^{-1}\bb + \bv 
    \end{pmatrix} 
    \\
    &= -\frac{1}{\tau}\begin{pmatrix} 
    \bF\bv \\
    \bG\bu 
    \end{pmatrix} 
    \\
    &= -\frac{1}{\tau}
    \begin{pmatrix} 
    \bF & 0  \\
    0 & \bG
    \end{pmatrix}
    \begin{pmatrix} 
    \bv \\
    \bu 
    \end{pmatrix}
    \\
    &= -\frac{1}{\tau}
    \begin{pmatrix} 
    \bF & 0  \\
    0 & \bG
    \end{pmatrix}
    \begin{pmatrix} 
    \bG\bF - \tau^2\Id & 0  \\
    0 & \bF\bG - \tau^2\Id
    \end{pmatrix}^{-1}
    \begin{pmatrix}
    -\tau\bG &  \tau^2\Id \\
    \tau^2\Id & -\tau\bF
    \end{pmatrix}
    \begin{pmatrix}
    \bA^{-1} & 0 \\
    0 & \bB^{-1}
    \end{pmatrix}
    \begin{pmatrix}
    \ba \\
    \bb
    \end{pmatrix}
    \\
    &=
    \begin{pmatrix} 
    \bF & 0  \\
    0 & \bG
    \end{pmatrix}
    \begin{pmatrix} 
    \bG\bF - \tau^2\Id & 0  \\
    0 & \bF\bG - \tau^2\Id
    \end{pmatrix}^{-1}
    \begin{pmatrix}
    \bG & - \tau\Id \\
    -\tau\Id & \bF
    \end{pmatrix}
    \begin{pmatrix}
    \bA^{-1} & 0 \\
    0 & \bB^{-1}
    \end{pmatrix}
    \begin{pmatrix}
    \ba \\
    \bb
    \end{pmatrix}
    \\
    &=
    \begin{pmatrix} 
    \bF & 0  \\
    0 & \bG
    \end{pmatrix}
    \begin{pmatrix}
    (\bF - \tau^2\bG^{-1})^{-1} &  -\tau (\bG\bF - \tau^2\Id)^{-1} \\
    -\tau (\bF\bG - \tau^2\Id)^{-1} & (\bG - \tau^2\bF^{-1})^{-1}
    \end{pmatrix}
    \begin{pmatrix}
    \bA^{-1} & 0 \\
    0 & \bB^{-1}
    \end{pmatrix}
    \begin{pmatrix}
    \ba \\
    \bb
    \end{pmatrix}
    \\
    &=
    \begin{pmatrix} 
    \bF & 0  \\
    0 & \bG
    \end{pmatrix}
    \begin{pmatrix}
    \bF &  \tau \Id  \\
    \tau \Id & \bG
    \end{pmatrix}^{-1}
    \begin{pmatrix}
    \bA^{-1} & 0 \\
    0 & \bB^{-1}
    \end{pmatrix}
    \begin{pmatrix}
    \ba \\
    \bb
    \end{pmatrix}
    \\
    &=
    \begin{pmatrix}
    \Id &  \tau \bG^{-1}  \\
    \tau \bF^{-1} & \Id
    \end{pmatrix}^{-1}
    \begin{pmatrix}
    \bA^{-1} & 0 \\
    0 & \bB^{-1}
    \end{pmatrix}
    \begin{pmatrix}
    \ba \\
    \bb
    \end{pmatrix}
    \end{split}
\end{align}
Therefore:
\begin{align}
\label{eq:finding-Hmu}
\begin{split}
    \bH\mu  &= \sigma^2
    \begin{pmatrix}
    \bF &   -\Id  \\
     -\Id & \bG
    \end{pmatrix}^{-1}
    \begin{pmatrix}
    \Id &  \tau \bG^{-1} \Id  \\
    \tau \bF^{-1} \Id & \Id
    \end{pmatrix}^{-1}
    \begin{pmatrix}
    \bA^{-1} & 0 \\
    0 & \bB^{-1}
    \end{pmatrix}
    \begin{pmatrix}
    \ba \\
    \bb
    \end{pmatrix}
\\
&= \sigma^2
    \left(
    \begin{pmatrix}
    \Id &  \tau \bG^{-1} \Id  \\
    \tau \bF^{-1} \Id & \Id
    \end{pmatrix}
    \begin{pmatrix}
    \bF &   -\Id  \\
    -\Id & \bG
    \end{pmatrix}\right)^{-1}
    \begin{pmatrix}
    \bA^{-1} & 0 \\
    0 & \bB^{-1}
    \end{pmatrix}
    \begin{pmatrix}
    \ba \\
    \bb
    \end{pmatrix}
    \\
&= \sigma^2
    \begin{pmatrix}
    \bF-\tau\bG^{-1} &  -(1-\tau)\Id   \\ 
    -(1-\tau) \Id &\bG-\tau \bF^{-1}
    \end{pmatrix}^{-1}
    \begin{pmatrix}
    \bA^{-1} & 0 \\
    0 & \bB^{-1}
    \end{pmatrix}
    \begin{pmatrix}
    \ba \\
    \bb
    \end{pmatrix}
    \\
&= \sigma^2
    \begin{pmatrix}
    \sigma^2 \bA^{-1} + (1-\tau)\Id &  -(1-\tau)\Id   \\ 
    -(1-\tau) \Id &\sigma^2 \bB^{-1} + (1-\tau)\Id
    \end{pmatrix}^{-1}
    \begin{pmatrix}
    \bA^{-1} & 0 \\
    0 & \bB^{-1}
    \end{pmatrix}
    \begin{pmatrix}
    \ba \\
    \bb
    \end{pmatrix}
    \\
&= 
    \begin{pmatrix}
     \bA^{-1} + \Id &  -\lambda\Id   \\ 
    -\lambda \Id &\bB^{-1} + \lambda\Id
    \end{pmatrix}^{-1}
    \begin{pmatrix}
    \bA^{-1} & 0 \\
    0 & \bB^{-1}
    \end{pmatrix}
    \begin{pmatrix}
    \ba \\
    \bb
    \end{pmatrix}
    \end{split}
\end{align}
Let's compute the inverse of: 
\begin{equation}
    \label{eq:Z}
\bZ \defeq \begin{pmatrix}
     \bA^{-1} + \frac{1}{\lambda}\Id &  - \frac{1}{\lambda}\Id   \\ 
    - \frac{1}{\lambda} \Id &\bB^{-1} +  \frac{1}{\lambda}\Id
    \end{pmatrix}.
\end{equation}
Let $\bS$ and $\bS'$ be the respective Schur complements of $\bA^{-1} +  \frac{1}{\lambda} \Id$ and $\bB^{-1} +  \frac{1}{\lambda} \Id$ in $\bZ$.
The block inverse formula writes:
\begin{equation*}
      \bZ^{-1}=\begin{pmatrix}
    \bS &   \frac{1}{\lambda} \bS (\bB^{-1}+ \frac{1}{\lambda} \Id)^{-1}   \\  \frac{1}{\lambda} (\bA^{-1}+ \frac{1}{\lambda} \Id)^{-1}\bS & \bS'
    \end{pmatrix}.
\end{equation*}
Using Woodbury's identity twice and denoting $\bX\defeq \bA + \bB + \lambda\Id$:
\begin{align*}
    \bS &= (\bA^{-1} +  \frac{1}{\lambda} \Id -  \frac{1}{\lambda^2}(\bB^{-1} +  \frac{1}{\lambda}\Id)^{-1})^{-1} \\
        &= (\bA^{-1} + (\bB + \lambda\Id)^{-1})^{-1}\\
        &= (\bA - \bA(\bA + \bB + \lambda\Id)^{-1}\bA)\\
        &= \bA - \bA\bX^{-1}\bA.
\end{align*}
And similarly: $\bS'= \bB - \bB\bX^{-1}\bB$. The off-diagonal blocks can be simplified as well:
\begin{align*}
     \frac{1}{\lambda} \bS (\bB^{-1}+ \frac{1}{\lambda} \Id)^{-1}&=  \frac{1}{\lambda}(\bA^{-1} + (\bB + \lambda\Id)^{-1})^{-1}
    (\bB^{-1}+ \frac{1}{\lambda} \Id)^{-1} \\
    &= (\bA^{-1} + (\bB + \lambda\Id)^{-1})^{-1}
    (\lambda\Id+ \bB\Id)^{-1}\bB \\
    &= \left((\bB + \lambda\Id) - (\bB + \lambda\Id)(\bA + \bB + \lambda\Id)^{-1}(\bB + \lambda\Id)\right)
    (\lambda\Id+ \bB\Id)^{-1}\bB\\
    &= \bB - (\bB + \lambda\Id)\bX^{-1}\bB
    \\
    &= \bB - (\bX - \bA)\bX^{-1}\bB\\
    &= \bA\bX^{-1}\bB.
\end{align*}
Similarly, $\frac{1}{\lambda} (\bA^{-1}+\frac{1}{\lambda}\Id)^{-1}\bS = \bB\bX^{-1}\bA$.
Thus, the inverse of $\bZ$ is given by:
\begin{equation}
    \label{eq:invZ}
    \bZ^{-1} = \begin{pmatrix}
    \bA - \bA\bX^{-1}\bA &  \bA\bX^{-1}\bB   \\ \bB\bX^{-1}\bA & \bB - \bB\bX^{-1}\bB
    \end{pmatrix}.
\end{equation}
and finally:
\begin{align*}
    \bH\mu &= \bZ^{-1}
        \begin{pmatrix}
    \bA^{-1} & 0 \\
    0 & \bB^{-1}
    \end{pmatrix}
    \begin{pmatrix}
    \ba \\
    \bb
    \end{pmatrix}=
    \begin{pmatrix}
            \Id - \bA\bX^{-1} & \bA\bX^{-1} \\
            \bB\bX^{-1} & \Id - \bB\bX^{-1}
    \end{pmatrix}
    \begin{pmatrix}
    \ba \\
    \bb
    \end{pmatrix}\\
    &=    \begin{pmatrix}
    \ba + \bA\bX^{-1}(\bb - \ba)\\
    \bb + \bB\bX^{-1}(\ba - \bb)
    \end{pmatrix}
\end{align*}
\paragraph{Finding the covariance matrix $\bH$.}
To compute $\bH = \left(\frac{1}{\sigma^2}\begin{pmatrix}\bF & -\Id \\ - \Id & \bG\end{pmatrix}\right)^{-1}$ one may
use the block inverse formula. However, the Schur complement $(\bF - \bG^{-1})^{-1}$ is not easy to manipulate. Instead notice that the following holds:
\begin{align*}
\frac{1}{\sigma^2}
\begin{pmatrix}
    \bF &  -\Id   \\ 
    - \Id &\bG
    \end{pmatrix}  
    \begin{pmatrix}
    \Id &  \tau\bF^{-1}   \\ 
     \tau\bG^{-1} &\Id
    \end{pmatrix} &= \frac{1}{\sigma^2}
      \begin{pmatrix}
    \bF-\tau\bG^{-1} &  -(1-\tau)\Id   \\ 
    -(1-\tau) \Id &\bG-\tau \bF^{-1}
    \end{pmatrix} 
    \\
    &= \begin{pmatrix}
     \bA^{-1} + \frac{1}{\lambda}\Id &  -\frac{1}{\lambda}\Id   \\ 
    -\frac{1}{\lambda} \Id &\bB^{-1} + \frac{1}{\lambda}\Id
    \end{pmatrix},
\end{align*}
where the last equality follows from the optimality conditions \eqref{eq:optim-unbalanced-params}.
Therefore:
\begin{align*}
    \bH &=    
    \begin{pmatrix}
    \Id &  \tau\bF^{-1}   \\ 
     \tau\bG^{-1} &\Id
     \end{pmatrix}
     \begin{pmatrix}
     \bA^{-1} + \frac{1}{\lambda}\Id &  -\frac{1}{\lambda}\Id   \\ 
    -\frac{1}{\lambda} \Id &\bB^{-1} + \frac{1}{\lambda}\Id
    \end{pmatrix}^{-1}.
\end{align*}
Notice that we have already computed the inverse matrix on the right side above in the developments of $\bH\mu$. Thus:
\begin{align*}
    \bH &=    
    \begin{pmatrix}
    \Id &  \tau\bF^{-1}   \\ 
     \tau\bG^{-1} &\Id
     \end{pmatrix}
     \begin{pmatrix}
     \bA - \bA\bX^{-1}\bA &  \bA\bX^{-1}\bB  \\ 
    \bB\bX^{-1}\bA & \bB - \bB\bX^{-1}\bB
    \end{pmatrix} \\
    &= 
    \begin{pmatrix}
    \Id &  \tau \bC\widetilde{\bB}^{-1}   \\ 
     \bC^\top\widetilde{\bA}^{-1}  &\Id
     \end{pmatrix}
     \begin{pmatrix}
     \bA - \bA\bX^{-1}\bA &  \bA\bX^{-1}\bB  \\ 
    \bB\bX^{-1}\bA & \bB - \bB\bX^{-1}\bB
    \end{pmatrix}
    \\
    &= 
    \begin{pmatrix}
    \Id &   \bC(\bB^{-1} +\frac{1}{\lambda} \Id)   \\ 
     \bC^\top(\bA^{-1} + \frac{1}{\lambda}\Id)  &\Id
     \end{pmatrix}
     \begin{pmatrix}
     \bA - \bA\bX^{-1}\bA &  \bA\bX^{-1}\bB  \\ 
    \bB\bX^{-1}\bA & \bB - \bB\bX^{-1}\bB
    \end{pmatrix}
    \\
    &= 
    \begin{pmatrix}
    \Id &   \bC(\bB^{-1} + \frac{1}{\lambda} \Id)   \\ 
     \bC^\top(\bA^{-1} + \frac{1}{\lambda}\Id)  &\Id
     \end{pmatrix}
     \begin{pmatrix}
     \bA - \bA\bX^{-1}\bA &  \bA\bX^{-1}\bB  \\ 
    \bB\bX^{-1}\bA & \bB - \bB\bX^{-1}\bB
    \end{pmatrix}
\\
&=
    \begin{pmatrix}
    \Id &   \frac{1}{\lambda}\bC(\lambda\Id  + \bB)\bB^{-1}   \\ 
     \frac{1}{\lambda}\bC^\top\bC(\lambda\Id  + \bA)\bA^{-1}   &\Id
     \end{pmatrix}
     \begin{pmatrix}
     \bA - \bA\bX^{-1}\bA &  \bA\bX^{-1}\bB  \\ 
    \bB\bX^{-1}\bA & \bB - \bB\bX^{-1}\bB
    \end{pmatrix}
\\
&=
    \begin{pmatrix}
    \Id &   \frac{1}{\lambda}\bC(\bX-\bA)\bB^{-1}   \\ 
     \frac{1}{\lambda}\bC^\top(\bX-\bB)\bA^{-1}   &\Id
     \end{pmatrix}
     \begin{pmatrix}
     \bA - \bA\bX^{-1}\bA &  \bA\bX^{-1}\bB  \\ 
    \bB\bX^{-1}\bA & \bB - \bB\bX^{-1}\bB
    \end{pmatrix}
\\
&=
    \begin{pmatrix}
     \bA - \bA\bX^{-1}\bA +   \frac{1}{\lambda}\bC(\bA-\bA\bX^{-1}\bA) &\bA\bX^{-1}\bB + \frac{1}{\lambda}\bC(\bX-\bA)(\Id - \bX^{-1}\bB)  \\ 
     \frac{1}{\lambda}\bC^\top(\bX-\bB)(\Id - \bX^{-1}\bA) + \bB\bX^{-1}\bA   & \frac{1}{\lambda}\bC^\top(\bX-\bB)\bX^{-1}\bB + \bB - \bB\bX^{-1}\bB
     \end{pmatrix}
\\
&=
    \begin{pmatrix}
     (\Id +  \frac{1}{\lambda}\bC)(\bA-\bA\bX^{-1}\bA) &\bA\bX^{-1}\bB + \frac{1}{\lambda}\bC(\bX - \bA -\bB + \bA\bX^{-1}\bB)  \\ 
     \lambda\bC^\top(\lambda\Id + \bB\bX^{-1}\bA) +\bB\bX^{-1}\bA  & \frac{1}{\lambda}\bC^\top(\bX-\bB)\bX^{-1}\bB + \bB - \bB\bX^{-1}\bB
     \end{pmatrix}
\\
&=
    \begin{pmatrix}
     (\Id + \frac{1}{\lambda}\bC)(\bA-\bA\bX^{-1}\bA) &\bA\bX^{-1}\bB + \frac{1}{\lambda}\bC(\lambda\Id + \bA\bX^{-1}\bB)  \\ 
     \bC^\top + \frac{1}{\lambda}\bC^\top\bB\bX^{-1}\bA + \bB\bX^{-1}\bA   & (\Id + \frac{1}{\lambda}\bC^\top)(\bB-\bB\bX^{-1}\bB)
     \end{pmatrix}
\\
&=
    \begin{pmatrix}
     (\Id + \frac{1}{\lambda}\bC)(\bA-\bA\bX^{-1}\bA) &\bC +  (\Id + \frac{1}{\lambda}\bC)\bA\bX^{-1}\bB  \\ 
     \bC^\top + (\Id + \frac{1}{\lambda}\bC^\top)\bB\bX^{-1}\bA   & (\Id +  \frac{1}{\lambda}\bC^\top)(\bB-\bB\bX^{-1}\bB)
     \end{pmatrix}.
\end{align*}

\paragraph{Finding the mass of the plan $\pi$.}
The optimal transport plan is given by:
\begin{align}
\label{eq:unbalanced-density}
\frac{\dd\pi}{\dd x\dd y}(x, y) 
&= \omega e^{\frac{1}{2}\mu^\top \Gamma^{-1}\mu}
\sqrt{\det(2\pi\bH)}\Ncal(\bH\mu, \bH)(z),
\end{align}
where 
\begin{align*}
    \omega &= \frac{m_\alpha m_\beta}{\sqrt{\det(4\pi^2 \bA\bB)}} m_u m_v e^{-\frac{1}{2}(\ba^\top\bA^{-1}\ba + \bb^\top\bB^{-1}\bb)}\\
     &= \frac{m_\alpha m_\beta}{\sqrt{\det(4\pi^2 \bA\bB)}}
     \left(\frac{\sqrt{\det(\bA\bB)\det(\bF\bG)}}{\sigma^{2d} m_\alpha m_\beta}\right)^\frac{\tau}{\tau + 1} e^{
     -\frac{\tau}{2(\tau + 1)}(q_{u, \alpha} + q_{v, \beta})}
     e^{-\frac{1}{2}(\ba^\top\bA^{-1}\ba + \bb^\top\bB^{-1}\bb)}\\
    &= \frac{1}{(2\pi)^{d}}\left(\frac{m_\alpha m_\beta}{\sqrt{\det( \bA\bB)}}\right)^{\frac{1}{\tau + 1}}
     \left(\frac{\sqrt{\det(\bF\bG)}}{\sigma^{2d}}\right)^\frac{\tau}{\tau + 1} e^{
     -\frac{\tau}{2(\tau + 1)}(q_{u, \alpha} + q_{v, \beta})}
     e^{-\frac{1}{2}(\ba^\top\bA^{-1}\ba + \bb^\top\bB^{-1}\bb)}.
\end{align*}
First, let's simplify the argument of the exponential terms. Isolating the terms that depend only on the input means $\ba, \bb$ it holds:
$q_{u, \alpha} + q_{v, \beta} = \frac{\sigma^2}{\tau^2}(\bv^\top\bF\bv +\bu^\top\bG\bu) + \ba^\top\bA^{-1}\ba+ \bb^\top\bB^{-1}\bb$. Therefore, the full exponential argument is given by:
\begin{equation}
    \label{eq:exparg}
   \phi \defeq \mu^\top\Gamma^{-1}\mu - \frac{\tau}{\tau + 1}\frac{\sigma^2}{\tau^2}(\bv^\top\bF\bv +\bu^\top\bG\bu) - \frac{1}{\tau + 1}(\ba^\top\bA^{-1}\ba+ \bb^\top\bB^{-1}\bb)
\end{equation}

On one hand, using \Cref{eq:finding-Hmu} we replace $\mu$:
\begin{align*} \mu^\top\Gamma^{-1}\mu &=
    \mu^\top\bH\mu
    \\
    &=
    \sigma^2
    \begin{pmatrix}
    \bA^{-1}\ba \\
    \bB^{-1}\bb
    \end{pmatrix}^\top
    \begin{pmatrix}
    \Id &  \tau \bF^{-1}  \\
    \tau \bG^{-1} & \Id
    \end{pmatrix}^{-1}
    \begin{pmatrix}
        \bF & -\Id\\ -\Id& \bG
    \end{pmatrix}^{-1}
    \begin{pmatrix}
    \Id &  \tau \bG^{-1}  \\
    \tau \bF^{-1} & \Id
    \end{pmatrix}^{-1}
    \begin{pmatrix}
    \bA^{-1}\ba \\
    \bB^{-1}\bb
    \end{pmatrix}
\end{align*}
On the other hand:
\begin{align*}
 \frac{\sigma^2}{\tau^2}(\bv^\top\bF\bv + \bu^\top\bG\bu) 
&= \sigma^2 ((\bA^{-1}\ba + \bu)^\top\bF^{-1}(\bA^{-1}\ba + \bu)+(\bB^{-1}\bb + \bv)^\top\bG^{-1}(\bB^{-1}\bb + \bv))\\
&= \sigma^2\mu^\top 
\begin{pmatrix}
        \bF^{-1} & 0 \\ 0 & \bG^{-1}
    \end{pmatrix}
\mu\\
&= 
  \sigma^2
    \begin{pmatrix}
    \bA^{-1}\ba \\
    \bB^{-1}\bb
    \end{pmatrix}^\top
    \begin{pmatrix}
    \Id &  \tau \bF^{-1}  \\
    \tau \bG^{-1} & \Id
    \end{pmatrix}^{-1}
    \begin{pmatrix}
        \bF^{-1} & 0 \\ 0 & \bG^{-1}
    \end{pmatrix}
    \begin{pmatrix}
    \Id &  \tau \bG^{-1}  \\
    \tau \bF^{-1} & \Id
    \end{pmatrix}^{-1}
    \begin{pmatrix}
    \bA^{-1}\ba \\
    \bB^{-1}\bb
    \end{pmatrix}
\end{align*}
Let $\bJ =     \begin{pmatrix}
    \Id &  \tau \bG^{-1}  \\
    \tau \bF^{-1} & \Id
    \end{pmatrix}
   $ and $\bK =     \begin{pmatrix}
        \bF & 0 \\ 0 & \bG
    \end{pmatrix}$. It holds:
\begin{align*}
   \mu^\top\Gamma^{-1}\mu - \frac{\tau}{\tau + 1}\frac{\sigma^2}{\tau^2}(\bv^\top\bF\bv +\bu^\top\bG\bu) &=
    \begin{pmatrix}
    \bA^{-1}\ba \\
    \bB^{-1}\bb
    \end{pmatrix}^\top
    {\bJ^\top}^{-1}(\bH -\frac{\sigma^2 \tau}{\tau+1}\bK^{-1})\bJ^{-1}
        \begin{pmatrix}
    \bA^{-1}\ba \\
    \bB^{-1}\bb
    \end{pmatrix}
\end{align*}
Let's compute the matrix ${\bJ^\top}^{-1}(\bH -\frac{\tau\sigma^2}{\tau+1}\bK^{-1})\bJ^{-1}$. First keep in mind that $\bJ\bK = \begin{pmatrix} \bF &\tau\Id \\ \tau\Id & \bG\end{pmatrix}$. Now 
using Woodburry's identity:
\begin{align*}
    \left({\bJ^\top}^{-1}(\bH -\frac{\tau}{\tau+1}\bK^{-1})\bJ^{-1}\right)^{-1} &= \bJ(\bH -\frac{\tau\sigma^2}{\tau+1}\bK^{-1})^{-1}\bJ^\top \\
    &= \bJ\left(-\frac{\tau+1}{\tau\sigma^2}\bK - \left(\frac{\tau+1}{\tau\sigma^2}\right)^2\bK(\bH^{-1} - \frac{\tau+1}{\tau\sigma^2}\bK)^{-1}\bK\right)\bJ^\top\\
    &=\frac{\tau+1}{\tau\sigma^2} \left( -\bJ\bK\bJ^\top - \frac{\tau+1}{\tau\sigma^2}\bJ\bK(\begin{pmatrix} - \frac{\bF}{\tau\sigma^2} & -\frac{1}{\sigma^2}\Id \\ -\frac{1}{\sigma^2}\Id & -\frac{\bG}{\tau\sigma^2}\end{pmatrix}^{-1}(\bJ\bK^\top)^\top\right) \\
    &=\frac{\tau+1}{\tau\sigma^2} \left( -\bJ\bK\bJ^\top + (\tau+1)\bJ\bK(\begin{pmatrix} \bF & \tau\Id \\ \tau\Id & \bG\end{pmatrix}^{-1}(\bJ\bK^\top)^\top\right) 
    \\
  &=\frac{\tau+1}{\tau\sigma^2} \left( -\begin{pmatrix} \bF &\tau\Id \\ \tau\Id & \bG\end{pmatrix}\begin{pmatrix}
    \Id &  \tau \bF^{-1}  \\
    \tau \bG^{-1} & \Id
    \end{pmatrix} + (\tau+1)\begin{pmatrix} \bF &\tau\Id \\ \tau\Id & \bG\end{pmatrix}\right)
  \\
   &=\frac{\tau+1}{\tau\sigma^2} \begin{pmatrix} -\bF -\tau^2\bG^{-1} + (\tau + 1)\bF &(-2\tau + \tau(\tau+1))\Id \\ (-2\tau + \tau(\tau+1))\Id & -\bG -\tau^2\bF^{-1} + (\tau + 1)\bG\end{pmatrix}
   \\
     &=\frac{\tau+1}{\sigma^2} \begin{pmatrix} \bF -\tau\bG^{-1}  & - (1-\tau)\Id \\ -(1-\tau)\Id & \bG -\tau\bF^{-1}\end{pmatrix}
   \\
     &=(\tau+1) \begin{pmatrix} \bA^{-1} + \frac{1}{\lambda} \Id  & -\frac{1}{\lambda}\Id \\ -\frac{1}{\lambda}\Id & \bB^{-1} + \frac{1}{\lambda}\Id\end{pmatrix}
     \\&= (\tau + 1)\bZ
\end{align*}
Therefore:
\begin{equation}
    \mu^\top\Gamma^{-1}\mu - \frac{\tau}{\tau + 1}\frac{\sigma^2}{\tau^2}(\bv^\top\bF\bv +\bu^\top\bG\bu) = \frac{1}{\tau + 1}
         \begin{pmatrix}
    \bA^{-1}\ba \\
    \bB^{-1}\bb
    \end{pmatrix}^\top
    \bZ^{-1}
        \begin{pmatrix}
    \bA^{-1}\ba \\
    \bB^{-1}\bb
    \end{pmatrix}
\end{equation}

The full exponential argument $\phi$ defined in \Cref{eq:exparg} is given by:
\begin{align*}
    \phi &= \frac{1}{\tau + 1}\left(
         \begin{pmatrix}
    \bA^{-1}\ba \\
    \bB^{-1}\bb
    \end{pmatrix}^\top
    \bZ^{-1}
        \begin{pmatrix}
    \bA^{-1}\ba \\
    \bB^{-1}\bb
    \end{pmatrix} - \ba^\top\bA^{-1}\ba - \bb^\top\bB^{-1}\bb\right)\\
    &=
\frac{1}{\tau + 1}
         \begin{pmatrix}
    \ba \\
    \bb
    \end{pmatrix}^\top
 \begin{pmatrix}
    \bA^{-1} & 0\\
   0 &  \bB^{-1}
    \end{pmatrix}
    \left(\bZ^{-1} - \begin{pmatrix}
    \bA & 0\\
   0 &  \bB
    \end{pmatrix}\right)
 \begin{pmatrix}
    \bA^{-1} & 0\\
   0 &  \bB^{-1}
    \end{pmatrix}
     \begin{pmatrix}
    \ba \\
    \bb
    \end{pmatrix}
\\
    &=
\frac{1}{\tau + 1}
         \begin{pmatrix}
    \ba \\
    \bb
    \end{pmatrix}^\top
 \begin{pmatrix}
    \bA^{-1} & 0\\
   0 &  \bB^{-1}
    \end{pmatrix}
 \begin{pmatrix}
    -\bA\bX^{-1}\bA & \bA\bX^{-1}\bB\\
   \bB\bX^{-1}\bA &  -\bB\bX^{-1}\bB
    \end{pmatrix}
 \begin{pmatrix}
    \bA^{-1} & 0\\
   0 &  \bB^{-1}
    \end{pmatrix}
     \begin{pmatrix}
    \ba \\
    \bb
    \end{pmatrix}
\\
    &=
\frac{1}{\tau + 1}
         \begin{pmatrix}
    \ba \\
    \bb
    \end{pmatrix}^\top
 \begin{pmatrix}
    -\bX^{-1} & \bX^{-1}\\
   \bX^{-1} &  -\bX^{-1}
    \end{pmatrix}
     \begin{pmatrix}
    \ba \\
    \bb
    \end{pmatrix}\\
    &= -\frac{1}{\tau + 1}(\ba-\bb)^\top\bX^{-1}(\ba-\bb)\\
    &= \frac{1}{\tau + 1}\|\ba-\bb\|^2_{\bX^{-1}}
\end{align*}
Substituting in \eqref{eq:unbalanced-density} leads to:
\begin{align*}
    m_\pi &\defeq \pi(\RR^d \times \RR^d) \\
    &= \sqrt{\det(\bH)}\left(\frac{m_\alpha m_\beta}{\sqrt{\det( \bA\bB)}}\right)^{\frac{1}{\tau + 1}}
     \left(\frac{\sqrt{\det(\bF\bG)}}{\sigma^{2d}}\right)^\frac{\tau}{\tau + 1}
    e^{
     -\frac{1}{2(\tau + 1)}(\|\ba-\bb\|^2_{\bX^{-1}})}.
\end{align*}

The determinants can be easily expressed as functions of $\bC$. First notice that:
\begin{align*}
    \det(\bH) = \frac{1}{\det(\Gamma)} = \frac{\sigma^{4d}}{\det(\bF\bG -\Id)},
\end{align*}
and using the definition of $\bC$, it holds that
\begin{align*}
    \bF\bG = \widetilde{\bB}\bC^{-2}\widetilde{\bA}.
\end{align*}
Therefore, $
    \det(\bF\bG) = \frac{\det(\widetilde{\bA}\widetilde{\bB} )}{\det(\bC)^2}$.
Keeping in mind that the closed form expression of $\bC$ given in
\eqref{eq:fg-closed-form-unbalanced} is applied to the pair $(\frac{1}{\tau}\widetilde{\bA}, \widetilde{\bB})$ in the unbalanced case, it holds:
$
    \bC^2 + \sigma^2 \bC = \frac{1}{\tau}\widetilde{\bA}\widetilde{\bB}$.
Thus:
\begin{align*}
\bF\bG-\Id &= \widetilde{\bB}\bC^{-2}\widetilde{\bA}( \Id -\widetilde{\bA}^{-1}\bC^2\widetilde{\bB}^{-1}) \\
&= \widetilde{\bB}\bC^{-2}\widetilde{\bA}( \Id -\widetilde{\bA}^{-1}(\frac{1}{\tau}\widetilde{\bA}\widetilde{\bB} - \sigma^2 \bC)\widetilde{\bB}^{-1}) \\
&= \widetilde{\bB}\bC^{-2}\widetilde{\bA}(\frac{(1-\tau)}{\tau}\Id + \sigma^2 \widetilde{\bA}^{-1}\bC\widetilde{\bB}^{-1})
 \\
&= \sigma^2\widetilde{\bB}\bC^{-2}\widetilde{\bA}(-\frac{2}{\gamma}\Id + \widetilde{\bA}^{-1}\bC\widetilde{\bB}^{-1})
 \\
&= \sigma^2\widetilde{\bB}\bC^{-2}(-\frac{2}{\gamma}\widetilde{\bA}\widetilde{\bB} + \bC)\widetilde{\bB}^{-1},
\end{align*}
therefore $$\det(\bF\bG - \Id) = \sigma^{2d}\frac{\det((-\frac{2}{\gamma}\widetilde{\bA}\widetilde{\bB} + \bC)}{\det(\bC)^2}.$$
Replacing the determinant formulas of $\bF\bG$ and $\bF\bG - \Id$ and re-arranging the common terms $\det(\bC)$ and $\sigma$ leads to:
\begin{align}
    \label{eq:mass-det}
    \begin{split}
 \pi(\RR^d \times \RR^d) &= \frac{\left(m_\alpha m_\beta \sigma^{2d} \det(\bC)\sqrt{\frac{\det(\widetilde{\bA}\widetilde{\bB})^\tau}{\det( \bA\bB)}}\right)^{\frac{1}{\tau + 1}}}
    {\sqrt{\frac{\det(\bC-\frac{2}{\gamma}\widetilde{\bA}\widetilde{\bB})}{\sigma^{2d}}}}
     e^{
     -\frac{1}{2(\tau + 1)}(\|\ba-\bb\|^2_{\bX^{-1}})} \\
     &= \sigma^{d(\frac{2}{\tau + 1} - 1)}\frac{\left(m_\alpha m_\beta \det(\bC)\sqrt{\frac{\det(\widetilde{\bA}\widetilde{\bB})^\tau}{\det( \bA\bB)}}\right)^{\frac{1}{\tau + 1}}}
    {\sqrt{\det(\bC-\frac{2}{\gamma}\widetilde{\bA}\widetilde{\bB})}}
     e^{
     -\frac{1}{2(\tau + 1)}(\|\ba-\bb\|^2_{\bX^{-1}})}
\\
     &= \sigma^{d\frac{1-\tau}{\tau + 1}} \frac{\left(m_\alpha m_\beta \det(\bC)\sqrt{\frac{\det(\widetilde{\bA}\widetilde{\bB})^\tau}{\det( \bA\bB)}}\right)^{\frac{1}{\tau + 1}}}
    {\sqrt{\det(\bC-\frac{2}{\gamma}\widetilde{\bA}\widetilde{\bB})}}
     e^{
     -\frac{1}{2(\tau + 1)}(\|\ba-\bb\|^2_{\bX^{-1}})}
\\
     &= \sigma^{\frac{d\sigma^2}{\sigma^2 + \gamma}} \frac{\left(m_\alpha m_\beta \det(\bC)\sqrt{\frac{\det(\widetilde{\bA}\widetilde{\bB})^\tau}{\det( \bA\bB)}}\right)^{\frac{1}{\tau + 1}}}
    {\sqrt{\det(\bC-\frac{2}{\gamma}\widetilde{\bA}\widetilde{\bB})}}
     e^{
     -\frac{1}{2(\tau + 1)}(\|\ba-\bb\|^2_{\bX^{-1}})}
     \end{split}
\end{align}

\paragraph{Deriving a closed form for $\UOT_\sigma$.}

Using \Cref{eq:mass-det}, a direct application of \Cref{prop:unbalanced-loss-opt} yields    
\begin{equation}
\UOT_\sigma(\alpha, \beta) = \gamma (m_\alpha + m_\beta) + 2\sigma^2(m_\alpha m_\beta) - 2(\sigma^2+ 2\gamma)m_{\pi^\star}.
\end{equation}

This ends the proof of \Cref{thm:unbalanced}.
}{}

\end{document}